\newif\ifshowsaveblocks
\showsaveblocksfalse

\newif\ifshowcomplete
\showcompletefalse

\documentclass[leqno,oneside,english]{elsarticle}

\usepackage{PreamblePageSetup}
\usepackage{PreambleDefinitions}
\usepackage{PreambleEnvironments}
\usepackage[verbose]{geometry}

\hypersetup{
 hidelinks,colorlinks=false,breaklinks=true,
 pdfkeywords={monads, homotopy theory, spectral sequences, obstruction theory, operads},
 pdfauthor={Niles Johnson, Justin Noel},
}

\usepackage{graphicx}

\renewcommand{\alg}{\mathrm{alg}}
\newcommand{\model}[2]{#2\text{-}#1}
\newcommand{\Res}{\mathrm{Res}}
\renewcommand{\Ind}{\mathrm{Ind}}

\newcommand{\nilesemail}{niles@math.osu.edu}
\newcommand{\nilesurl}{http://www.nilesjohnson.net}
\newcommand{\justinemail}{justin.noel@mathematik.uni-regensburg.de}
\newcommand{\justinurl}{http://www.nullplug.org/}

\newcounter{enumisaved}
\newlength{\thmsaved}
\newlength{\thmnow}

\makeatletter
\newcommand*{\strippt}[1]{\strip@pt#1}
\makeatother
\newcommand{\setthmsaved}[1]{\setlength{\thmsaved}{#1pt}}
\newcommand{\setthmnow}[1]{\setlength{\thmnow}{#1pt}}

\newcommand{\complete}[1]{\ifshowcomplete
#1
\fi}

\newcommand{\completeblock}[2][0]{%
  \setcounter{enumisaved}{\theenumi}%
  \setthmsaved{\thethm}%
  \ifshowcomplete
  \ \todo[inline, bordercolor=blue!40!black,
  backgroundcolor=green!10!white]{Begin completeblock #1}
  #2
  \setthmnow{\thethm}
  \ \todo[inline, bordercolor=red!40!black,
  backgroundcolor=green!10!white]{End completeblock #1}
  \ifthenelse{\lengthtest{\thmnow > \thmsaved}}{
    \setcounter{thm}{\strippt{\thmsaved}}  % cannot set counter to non-integral value
  }{%else
  }
  \setcounter{enumi}{\theenumisaved}
  \else%
  \fi
}

\newcommand{\longsaveblock}[2][0]{
  \setcounter{enumisaved}{\theenumi}
  \setthmsaved{\thethm}
  \ifshowsaveblocks
  \ \todo[inline, bordercolor=green!40!black,
  backgroundcolor=yellow!30!white]{\tiny Begin long saveblock #1}
  #2
  \setthmnow{\thethm}
  \ \todo[inline, bordercolor=red!40!black,
  backgroundcolor=yellow!30!white]{\tiny End long saveblock #1}
  \ifthenelse{\lengthtest{\thmnow > \thmsaved}}{
    \setcounter{thm}{\strippt{\thmsaved}}  % cannot set counter to non-integral value
  }{%else
  }
  \setcounter{enumi}{\theenumisaved}
  \else
  \fi
}

\newcommand{\saveblock}[2][0]{\ifshowsaveblocks
\todo[inline, bordercolor=black,
  backgroundcolor=yellow!30!white]{[\texttt{saveblock #1}] #2}
\fi}

\newcommand{\saveblockZERO}{\saveblock[intro]{

  \textbf{Commentary}

  This version is compiled with 
  additional notes by setting the $\backslash$showsaveblocks flag to
  true.  These notes are shown
  completely enclosed in yellow bubbles, unless doing so created
  compile errors, in which case there are begin/end 
  markers.  These are called ``long saveblocks'' because some of 
  them are longer, although not all of them are.
}}

\newcommand{\saveblockA}{\saveblock[A]{
  Monadic implies the assumptions is easy. We need to show the canonical
  functor to the category of algebras is an equivalence. To see that
  the functor is full one applies \cite[Prop.~4.2.1]{Bor94a}. To see that
  the functor is essentially surjective one applies the canonical
  presentation and \cref{it:barr-beck-b}. To show that the functor is
  faithful one shows that the composite with the forgetful functor
  which is $U$ is faithful. Assuming that $UX$ is given by applying
  $U$ to the canonical presentation one sees that maps out of $UX$ are
  uniquely determined from the colimit presentation (two different
  maps in the source which become equivalent after applying $U$ induce
  the same maps between canonical presentations so one needs to show
  that structure map is an epimorphism this follows from the structure
  map being a coequalizer). Of course up to
  isomorphism every algebra is given by the canonical presentation and
  then we use \cref{it:barr-beck-a} to see there is a unique map.  
}}

\newcommand{\saveblockB}{\saveblock[B]{
To prove the equivalence one uses 3.8.12 Bor94a which shows every
$T$-algebra is a filtered colimit of finitely presentable ones, which are
reflexive coequalizers of free algebras. So it suffices to prove the
equivalence on the subcategories of free algebras.

}}

\newcommand{\saveblockC}{\saveblock[C]{
Taking a skeleton of $\FinSet^S$, each object $(x_s)_{s \in S}$ of
$(\FinSet^S)^{\op}$ is isomorphic to $(\underline{n_s})_{s \in S} \iso
(\underline{1}^{\times n_s})_{s \in S}$ where $|x_s| = n_s$ and
$\underline{n_s}$ is a fixed set with $n_s$ elements.

Thus, if $i\cn(\FinSet^s)^{\op} \to \cT$ is a product preserving
functor and we have $i((\underline{1})_{s \in S}) = \prod_s i_s$ for
objects $i_s \in \cT$, then for each $x \in (\FinSet^S)^{\op}$ we have
\[
i(x)\cong \prod_{s\in S} i_s^{\times n_s}.
\]

}}

\newcommand{\saveblockD}{\saveblock[D]{
      
      Fixing a $\bZ$-graded abelian group $A$, the category of
      $\bZ$-graded abelian groups over $A$ is the category of
      $\Set$-valued models for a theory $\cT^{\bZ}_{\Ab}\downarrow A$
      graded by $S=\coprod_{j\in \bZ}S_j$ where
      \[S_j=\AbGroup^\bZ(\bZ[j],A).\] 
      The objects of this theory can be identified with the graded abelian
      groups over $A$ which are finite direct sums of graded abelian
      groups of the form $\bZ^{n}[j]$.
    
}}

\newcommand{\saveblockF}{\saveblock[F]{
  How to construct overcategories as graded theories.

  Example:

  Given any singly-graded theory $\cT$ and a $\cT$-model $X \in
  \model{\Set}{\cT}$, the category of $\cT$-models over $X$ is an
  $S$-graded theory where $S$ is the set of $\cT$-model maps from the
  free $\cT$-model $\cT\{\underline{1}\}$ to $X$. This theory is denoted $\cT
  \downarrow X$.  More generally, if $\cT$ is an $S_0$-graded theory
  then the category of $\cT$-models over $X$,
  $(\model{\Set}{\cT})_{\downarrow X}$, is a theory graded by $S =
  \coprod_{j \in S_0} S_j$ where each $S_j$ is the set of $\cT$-model
  maps from $\cT\{[j]\} = \cT(i([j]),-)$ to $X$.
  
}}

\newcommand{\saveblockG}{\saveblock[G]{
All of this passes to the undercategory situation as well (suppose we
are looking at the category of models under Z). The model
category, (co)limit, and simplicial parts all dualize. The simplicial
hom objects are again defined by the obvious pullback (no error).
Colimits in this category are calculated as colimits in the underlying
category after we adjoin the object we are under to the diagram.
Limits are calculated in the underlying category. The tensor product
can be calculated as a pushout  $Z \leftarrow Z\otimes K \to X\otimes K$ where
the first map is induced by $K \to *$. We can use the cotensor form of
axiom SM7 and that cotensors and pullbacks are calculated in the
underlying category to see that the undercategory is a simplicial
model category when the underlying category is simplicial.

The cofibrantly generated part does not dualize. But if $I$ and $J$ are
the generating cofibrations and acyclic cofibrations respectively of
$\cC$ then the new generating cofibrations and acyclic cofibrations are
$Z\coprod I$ and $Z\coprod J$.

It is now straightforward to check Hovey's axioms. 

Now for the based case. We can use the above to see the category of
objects over and under $Z$ is a cofibrantly generated simplicial model
category. Since this category is pointed this enrichment restricts to
an enrichment in based simplicial sets. We can see that the
composition product, which is defined using the cartesian product
factors over the smash product. 

Now a $\pTop(K,\cC(X,Y)) \subseteq \Top(K, \cC(X,Y))\cong \cC(X,Y^K)
\cong \cC(K\otimes X, Y)$.
Which subset is this? Well every map of topological spaces is over the
terminal object, but not under a point. So we are looking for the maps
under a point, which is given by a pullback  of
\[
* \to \Top(*, \cC(X,Y)) \leftarrow \Top(K, \cC(X,Y)) 
\]
Where the first arrow picks out the zero morphism and the second is
induced by restriction along $* \to K$.  The cotensor looks like it is
calculated as the pullback $* \to Y \leftarrow Y^K$ in the underlying category.
Dually the tensor should be calculated as the pushout $* \leftarrow X \to
K\otimes X$. 

Now we need to verify that SM7 holds. We use the version listed in
terms of mapping spaces. Pullbacks in the category of pointed
simplicial sets are calculated in the unpointed category.  Since
cofibrations, fibrations, and weak
equivalences in the pointed version of $\cC$ as well as pointed
simplicial sets are the same as in the unpointed categories. It
follows that mapping space form of SM7 holds in the pointed category
because it holds in the unbased category.

}}

\newcommand{\saveblockH}{\saveblock[H]{See also: Kriz-May [KM95Operads] and Fresse [Fre09Modules]}}

\newcommand{\saveblockI}{\saveblock[I]{we work this out in more detail
when $F$ is the tensor with the simplicial category}}

\newcommand{\saveblockJ}{\saveblock[J]{

  Proof:

  The retraction follows by taking the geometric realization of the
  following simplicial homotopy $h = \{h_i\}$ between $\id_{T^{\bul+1}
  UX}$ and
  $e \circ \epz$:
  \[
  h_i = e^{i+1}\mu^i : T^{n} UX \to T(T^n UX) = T^{n+1} UX
  \]
  where $\mu^i$ denotes the composite
  \[
  T^i W \fto{\mu_{T^{i-1} W}} T^{i-1} W \fto{\mu_{T^{i-2}W}} \cdots
  \fto{\mu_{T^2 W}} T^2 W
  \fto{\mu_{TW}} TW \fto{\mu} W
  \]
  with $W = T^{n-i} UX$, and $e^{i+1}$ denotes the composite
  \[
  W \fto{e_W} TW \fto{e_{TW}} T^2 W \fto{e_{T^2 W}} \cdots
  \fto{e_{T^i W}} T^{i+1} W.
  \]

}}

\newcommand{\saveblockK}{\saveblock[K]{Niles's notes on partial latching objects}}

\newcommand{\saveblockL}{\saveblock[L]{General cellular latching lemma}}

\newcommand{\saveblockM}{\saveblock[M]{Version of this argument in EKMM}}

\newcommand{\saveblockN}{\saveblock[N]{%
  Okay so this is sweeping some details under the rug. The spectral
  sequence is always supposed to consist of based sets so what we are
  saying here does not make strict sense. The remark after the corollaries below is supposed
  to let the reader know that the spectral sequence consists of based
  sets of  maps. (but not a set of based maps!)
  Strictly speaking there is a single map taking a homotopy class of
  map in $\pi_0 (\cC_T(B_0 X,Y),\epsilon)$ to the 
  unbased homotopy classes of maps $S^0\rightarrow
  \cC_T(F_TTUX,Y)\cong \cC(TUX,UY)$ 
  via $d^0$ and $d^1$. Moreover as described by Bousfield this unbased
  map satisfies the property that if you restrict along the unit 
  $UX\rightarrow TUX$, the two maps are connected by a path which is
  necessarily $\epsilon$ at one end. This is the role of normalization
  He is just saying the map satisfies an additional property.
  In his description $[f]$ survives to $E_2$ precisely when the unbased map
  is trivial, that is when there is a path connecting $d^0[f]$ and
  $d^1[f]$ note that $d_1[f]$ is always the base point of the latter.
  So all of these abuses are justified as we can see the underlying
  sets are in fact equal.}}

\newcommand{\saveblockO}{\saveblock[O]{
  The coequalizer formula for $F_{T_3}$ in \cite[Lem.~4.5.5]{Bor94a}
  is something like the following:
  \[ 
  T_3 X \iso \coeq( T_2 X\leftleftarrows T_2 T_1 X)
  \]  
}}

\newcommand{\saveblockP}{\saveblock[P]{
One can see that the this map agrees with the induced map on totalized mapping spaces by showing that $U_3$ commutes with geometric realization and then applying the other maps above. To see that the canonical map $|U_3 B_\bul|_{\cC_{T_1}} \to U_3 |B_\bul|_{\cC_{T_2}}$ is an isomorphism for any simplicial $T_2$ algebra $B_\bul$, it suffices to show that it is an isomorphism after applying $U_1$ since $U_1$ reflects isomorphisms. This map fits into the following composite
  \[
  |U_1U_3 B_\bul|_{\cC} \to U_1|U_3 B_\bul|_{\cC_{T_1}} \to
  U_1U_3 |B_\bul|_{\cC_{T_2}} .
  \]  
  The first map and the composite are isomorphisms because $U_1$
  and $U_2 = U_1U_3$ commute with geometric realization.  Hence the
  second map is an isomorphism.
  
}}

\newcommand{\saveblockQ}{\saveblock[Q]{
We really need $Y$ to be fibrant for the statement below. Otherwise
there is no guarantee that the natural map $Y^{S^t}\rightarrow
(Y^{\prime})^{S^t}$ is a weak equivalence for a fibrant replacement
$Y^{\prime}$.

}}

\newcommand{\saveblockR}{\saveblock[R]{
It will suffice to have the following:
Since $T_{\alg}$ is monadic over set we obtain a combinatorial
model structure on $sT_{\alg}$ which is monadic over $\sSet$. This
guarantees enough projectives. 

We also need that the canonical map $T_{\alg}Y\rightarrow Y$ is an
effective epiomorphism of simplicial $T_{\alg}$-algebras. Since the
category is algebraic it suffices to see that this map is a
coequalizer of some diagram (i.e., it is a regular epimorphism). Since
it is a coequalizer of the canonical resolution we see it is an
effective epimorphism.
 
We also need to know that $T_{\alg}Y$ is a projective
$T_{\alg}$-algebra for any $Y$. By definition to be projective we
need for any effective epimorphisms $V\rightarrow W$ the induced
map \[sT_{\alg}(T_{\alg} Y, V)\rightarrow sT_{\alg}(T_{\alg} Y, W)\] is
surjective. Since $U$ preserves pullbacks and reflexive
coequalizer diagrams the argument from the previous
paragraph shows that a morphism of $T_{\alg}$-algebras is an
effective epimorphism if and only if it is an effective epimorphism
after applying $U$. The result now follows by adjointness. 

When conditions \eqref{item:universal-coefficient} and \eqref{item:algebraic-description} are satisfied, we can apply the Hurewicz homomorphism levelwise
to the cosimplicial space defining the $T$-algebra spectral sequence and
obtain an analogous cosimplicial graded object in the (computationally
accessible) category $\cD$.  The associated $E_2$ term is the cohomotopy of
this graded object as in \cite{Bou89}.  for well-chosen $\cD$, gives a purely
algebraic description of the homotopy $T$-algebra maps.

When the category of algebras $\cD_{T_{\alg}}$ has enough projectives and
is closed under finite limits we can
apply \cite[\S~1]{Qui70} to see that the category $sT_{alg,\pi_*Y}$ of
simplicial $T$-algebras over $\pi_*Y$ and its subcategory of abelian
group objects admit model structures satisfying the conditions of
\cite[II.5.(1),(2),(4)]{Qui67}.

}}

\newcommand{\saveblockS}{\saveblock[S]{
For the $n=3$ argument, we have elements on the 2-line but
  the obstruction class lives in $E_2^{2,1}$ which is spanned by duals
  of products $\sigma x_i \sigma x_j,  (i\neq j)$ (which has bidegree $(2,-2i-2j)$) times an element of $R$ in degree $1+2i+2j$. Since $R$ is concentrated in even degrees there is no such element.
}}

\newcommand{\saveblockT}{\saveblock[T]{See the comment at the end of the previous example.}}

\newcommand{\saveblockU}{\saveblock[U]{
Under this hypothesis the tor spectral sequence collapses so $\pi_*
(R^{\otimes_k n})\cong (\pi_*R)^{\otimes_{\pi_*k}n}$.  The homotopy fixed point spectral sequence collapses because everything is rational. This shows that $\pi_*TR=T_\alg \pi_*R$. Moreover $\pi_*R^{\otimes_{\pi_*k}n}/\Sigma_n$ is a free $\pi_* k$ module (we have the standard monomial basis), so the Ext spectral sequence also collapses. This gives the requisite criteria for \refthmB.

}}

\newcommand{\saveblockV}{\saveblock[V]{
Checked signs $dz=xy, -dz=yx\implies \alpha=xz, \beta=-yz, x(-yz)-y(xz)=-xyz+xyz=0$. An alternative proof uses juggling $\alpha y=y\alpha=<x,x,y> y=x<x,y,y>=x\beta$.
}}

\newcommand{\saveblockW}{\saveblock[W]{The homotopy type of the fibre does not depend on this choice}}

\newcommand{\saveblockX}{\saveblock[X]{
  One can construct this lift directly: the sphere maps to
  it's $K(1)$-localization by an $E_\infty$-ring map. The
  $K(1)$-localization sits in a fibration 
  \[
  S_{K(1)}\rightarrow K_p \xrightarrow{\psi^q-1} K_p
  \]
  There are compatible ring maps to the left and middle terms which
  induce the above map on 1 components. By this construction this is
  an infinite loop map.
}}

\newcommand{\saveblockY}{\saveblock[Y]{
Since the last map was an infinite loop map, $\Coker J$ is an infinite loop space and the fiber inclusion is an infinite loop map.
}}

\newcommand{\saveblockZ}{\saveblock[Z]{
Again the $D$-invariant comes from a construction as in the first
paragraph but instead we are taking 
\[
S_{K(1)}\rightarrow KO_2 \xrightarrow{\psi^3-1} KO_2.
\] 
Examination of the LES in homotopy gives the homotopy of the
connective cover of $KO\wedge \Sigma^{-1}KO$ through degree 2. Taking
a 1-connected cover we just recover $\eta^2$ and the rest of the
$Im~J$ spectrum.  This 1 connected cover can be constructed by taking
a 1-connected cover of the middle term and and a two connected cover
of the right term (note that $\psi^3$ acts identically on $\eta^2$ so the
difference lifts to the 2-connected cover).

To compare with HoS78: There are three $J$'s constructed there (first
additively as below and then their tensor analogues).

$J_1\times Z\rightarrow BO_2\times Z\rightarrow BSpin_2$: A
1-connected cover of this J is constructed as the fiber of the
analogous map using a 1-connected cover of $BO_2\times Z$. This agrees
with the above definition.

$J_2\rightarrow BSO_2\rightarrow BSpin_2$: This agrees with the above definition.

$J_3 \rightarrow BSO_2 \rightarrow BSO_2$: This incurs an extra '$\eta$' in degree 0. Again the one connected cover is constructed by replacing $BSO_2$ with $BSpin$.

}}

\newcommand{\saveblockZA}{\saveblock[ZA]{
The first positive stable stem is $\alpha$ which occurs in degree $2p-3\ge3$ and this is in the image of $J$ so it is detected by the $D$-invariant.
}}

\newcommand{\saveblockZB}{\saveblock[ZB]{
3-connected is saying we detect $\eta^2$ and the 2-primary $\bZ/8$ given by $\nu$. This is stated both at the end of 5.3.7 in Ravenel's book and 8.15 of his localization paper. The latter paper talks about $L_1 S$. But the answer for $K(1)$-local sphere only differs by my the torsion free + divisible parts which get moved around by the chromatic fracture square.
}}

\newcommand{\saveblockZC}{\saveblock[ZC]{
We are sweeping some more details under the rug here. These $T_{\alg}$
monads act on a graded $k$-modules, not sets, and they are graded so
we can't pull these examples out directly from Quillen. However all of
these examples come from finitary monadic categories over some product
of the category $\Set$ and so define algebraic categories in Quillen's
sense. 

}}

\newcommand{\completeblockZERO}{\completeblock[(initial)]{
  \textbf{Director's Cut}

  This version is compiled with additional content cut from the
  publication ($\backslash$showsaveblocks flag is set to true).  The content is offset by ``Begin
  completeblock \emph{name}'' and ``End completeblock \emph{name}''
  bubbles.  This content was more polished, but may not be completely
  compatible with later updates of the document.

  Item counters are reset after completeblocks, and no environments in
  completeblocks are numbered (doing so raises an error), so numbering
  outside of the completeblocks is identical to the published
  version.  

}}

\newcommand{\completeblockA}{
  \complete{\item[]}
  \completeblock[A]{
    \item $E_1^{0,0}=\ho\cC(UX,UY)$.
    \item $E_2^{0,0}=(\ho\cC)_{\h T}(UX, UY)$.  That is, a homotopy class
     $[f] \colon UX \to UY$ survives to the $E_2$ page if and only
     if it is a map of $\h T$-algebras\footnote{Note that the
     notion of \emph{homotopy} $T$-algebra is entirely distinct from
     the notion of a homotopy algebra in the context of algebras over
     operads. If $T$ comes from an operad then our $T$-algebras are
     homotopy algebras.} in the homotopy category.
    
}}

\newcommand{\completeblockB}{\completeblock[B]{
\begin{defn*}
  \begin{samepage}
   Let $G \cn \cC \to \cD$ be a functor, and let $D \cn \cI \to \cC$
   be a diagram in $\cC$.
   \begin{itemize}
   \item $G$ \emph{preserves} colimits of $D$ if $GD \to GA$ is a
     colimit in $\cD$ whenever $D \to A$ is a colimit in $\cC$.  
   \item $G$ \emph{reflects} colimits of $D$ if $D \to A$ is a
     colimit in $\cC$ whenever $GD \to GA$ is a colimit in $\cD$.
   \item $G$ \emph{creates} colimits of $D$ if $D$ has a colimit
     whenever $GD$ has a colimit and $G$ preserves and
     reflects colimits of $D$.
   \end{itemize}
  \end{samepage}
 We say that $G$ preserves, reflects, or creates colimits if it does
 so for all diagrams.  Similar terminology is used for preservation,
 reflection, or creation of limits.
\end{defn*}

}}

\newcommand{\completeblockC}{\completeblock[C]{
\begin{proof}
  Let  $\alpha\colon \colim TX_i \rightarrow T\colim X_i$ be the map
  induced by applying $T$ to the maps $X_i\rightarrow \colim X_i$.
  Since $T$ commutes with reflexive coequalizers, every colimit in
  $\cC_T$ can be calculated via the following formula in $\cC$:
  \begin{equation*}\label{eqn:colimits-of-t-algebras}
	\xymatrix@C=13ex{
	  \coeq \big[ T\colim TX_i\ar@<1ex>[r]^-{T\colim
	  \mu_{X_i}}\ar@<-1ex>[r]_-{\mu\circ \alpha} &
	  T\colim X_i\ar@/_1.85pc/[l]_-{Te_{X_i}} \big]\\
	}
  \end{equation*}
\end{proof}

}}

\newcommand{\completeblockD}{\completeblock[D]{\begin{proof}
  Since $\cC$ and $\cD$ are locally presentable the first two
  conditions guarantee that $U$ admits a left adjoint $F$
  \cite[Thm.~5.5.7]{Bor94a}. The third condition is stronger than the two
  conditions of the monadicity theorem, and thus we have $\cD \hty
  \cC_T$. The last two conditions guarantee that $T=UF$ commutes with
  the stated colimits.
\end{proof}
}}

\newcommand{\completeblockE}{
  \complete{\item[]}
  \completeblock[E]{
    \item If $\cT=\FinSet^{\op}$ and $i$ is the identity functor then
         $\model{\cC}{\cT}$ is the category of theories in $\cC$.}
}

\newcommand{\completeblockF}{\completeblock[F]{
    \item Let $\cT_{\Ab}$ be the theory defined as in \eqref{enum:group}
      but with 
      \[
      \cT_{\Ab}(m,n)=\AbGroup(F\{n\},F\{m\}),
      \]
      where $F\{m\}$ is the free abelian group on $m$ elements, then
      $\model{\Set}{\cT_{\Ab}}$ is equivalent to the category of abelian groups.
    
}}

\newcommand{\completeblockG}{\completeblock[G]{ \item Let $\cT_{Comm}$ be the theory defined as in \eqref{enum:group}
      but with 
      \[
      \cT_{Comm}(m,n)=\Comm(F\{n\},F\{m\}),
      \]
      where $F\{m\}$ is the free commutative ring on $m$ elements, then
      $\model{\Set}{\cT_{Comm}}$ is equivalent to the category of commutative
      rings.
}}

\newcommand{\longsaveblockH}{\longsaveblock[H]{
    \item \ \\
    \saveblockD{}
}}

\newcommand{\longsaveblockJ}{\longsaveblock[J]{
\begin{thm*}
  There are correspondences between 
  \begin{enumerate}
    \item Finitary monadic categories over $\Set^S$.
    \item Categories of models in $\Set$ of $S$-sorted theories.
    \item Algebraic categories, in the sense of Quillen, generated by $S$ elements (i.e.,
    cocomplete and finitely complete categories, with a set $S$ of small projective
    generators).
  \end{enumerate}
\end{thm*}
\begin{proof}
  The first two are equivalent by the book of Adamek, Rosicky, and
  Vitale.

First we recall some relevant definitions from
\cite{Qui67}.

Let $\cC$ be a category with finite limits. A morphism $f\in \cC(X,Y)$
is called an \emph{effective epimorphism} if it fits into a coequalizer
diagram 
\[ X\times_{Y}X\rightrightarrows X\xrightarrow{f} Y \] 
where the arrows to be coequalized are the two projections to $X$. 
Note that the diagonal map $X\rightarrow X\times_Y X$ makes this into a
reflexive coequalizer diagram. 

An object $P\in \cC$ is \emph{projective} if for every effective epimorphism
$f\colon X\rightarrow Y$ the induced homomorphism \[ 
\cC(P,X)\rightarrow \cC(P,Y) \] is surjective. We will say that
$\cC$ has \emph{enough projectives} if for every object $X$ there is an
effective epimorphism $P\rightarrow X$ with $P$ projective. 

  The category of models over an $S$-sorted theory has the
  representable functors as small projective generators and is closed
  under colimits so it is algebraic.

  Using the set of small projective generators we can construct a
  functor \[U\colon\cC \rightarrow \Set^S \] sending $X$ to $\prod
  \cC(Fs,X)$. Since the generators are small this preserves filtered
  colimits. Since they are generators this functor is faithful and
  conservative. Now $\cC$ is an exact category so we can apply Theorem 18.1
  from Adamek, Rosicky, and Vitale To see that $U$ preserves reflexive
  coequalizers if and only if it preserves regular epimorphisms, which
  in a category with finite limits, such as $\cC$ are the same as
  effective epimorphisms. In $\Set^S$ effective epimorphisms are the
  same as epimorphisms so we need to know that $U$ takes effective
  epimorphisms to surjections. This is the definition of what it means
  for $Fs$ to be projective. Since $\cC$ and $\Set^S$ are locally
  presentable and $U$ clearly preserves limits we can now apply
  \Cref{prop:barr-beck-for-presentable-categories} to see that $\cC$
  is finitary monadic over $\Set^S$. 

  By using the construction given and the correspondence between the
  first two collections we can easily see that we obtain equivalent
  categories when we go from $\cC$ to its category of models over a
  theory and then back to an algebraic category.

  \end{proof}
  Note that as a consequence of these correspondences we see that any
  algebraic category is complete.
  We can also see that our finitary forgetful functor preserves all
  sifted colimits and not just the filtered colimits.

}}

\newcommand{\completeblockK}{\completeblock[K]{
\begin{rmk*}\label{rem:colimits-in-theories}
  We can apply the discussion of limits and colimits in
 \Cref{sec:monads} to analyze these constructions in $\Set^{S}_T$.  For
example, since $\Set^{S}$ is complete, $\Set^{S}_T$ is complete and the
  forgetful functor creates all limits. By applying
\eqref{eqn:monad-associated-to-a-theory} we see that since sifted
 colimits commute with products in sets and all colimits commute with
 coends, $T$ commutes with sifted colimits. Hence the forgetful
 functor reflects all sifted colimits. The remaining colimits can be
 constructed via \eqref{eqn:colimits-of-t-algebras}.
\end{rmk*}

}}

\newcommand{\completeblockL}{\completeblock[L]{
\begin{prop*}\label{prop:simplicial-over-categories}
  Suppose that $\cC$ is a bicomplete simplicial category and $X\in \cC$.
  Then the overcategory $\cC_{\downarrow X}$ admits the structure of a
  bicomplete simplicial category such that:
  \begin{enumerate}
  \item The forgetful functor $U\colon \cC_{\downarrow X}\rightarrow
    \cC$ creates tensors and colimits.
  \item The limit of a diagram $F\colon \cI\rightarrow \cC_{\downarrow
    X}$ may be calculated as the limit of the diagram \[F^{\prime}
    \colon \cJ \rightarrow \cC \] where $\cJ$ is obtained by adjoining a
    terminal object to $\cI$ and $F^{\prime}$ is the extension of $F$
    which maps the terminal object to $X$.
  \item The mapping spaces $\cC_{\downarrow X}(Y,Z)$ can be
    constructed as the following pullback in $\sSet$:
    \[\xymatrix{
    \cC_{\downarrow X}(Y, Z)\ar[r]\ar[d]& \cC(Y, Z) \ar[d]^-{g} \\
    {*} \ar[r]^-{f} & \cC(Y,X)
    }\]
    Here $f$ picks out the structure map $Y\rightarrow X$ and $g$ is
    induced by the structure map for $Z$.
  \item For a simplicial set $K$ and an object $Y\rightarrow X$ of
    $\cC_{\downarrow X}$ the cotensor $Y^{K}_{\downarrow X}$ in
    $\cC_{\downarrow X}$ can be constructed as the following pullback in 
    $\cC$:
    \[\xymatrix{ 
      Y^{K}_{\downarrow X}\ar[d]\ar[r]& Y^K \ar[d]^-{g}\\
      X \ar[r]^-{f} & X^K
    }\]
    Here $f$ is induced by the constant map $K\rightarrow *$ and $g$
    is induced by the structure map for $Y$.
  \end{enumerate}
\end{prop*}
\begin{proof}
  Verifying the statements concerning limits and colimits is immediate
  and left to the reader. It is also clear that the mapping spaces have
  the correct underlying hom-sets and that the tensor and cotensor
  are functorial constructions landing in $\cC_{\downarrow X}$. 
 
  By taking products of the pullback diagrams defining the mapping
  spaces $\cC_{\downarrow X}(-,-)$ and using the composition product of
  mapping spaces in $\cC$ we obtain a composition product of mapping
  spaces in $\cC_{\downarrow X}$. This product is associative and unital
  because the composition product on $\cC$ is associative and unital.

  We now check compatibility of the mapping space, tensor, and cotensor
  functors. By definition, the following diagram of sets is a pullback 
  \[\xymatrix{
    \cC_{\downarrow X}(K\otimes Y, Z)\ar[r]\ar[d] &\cC(K\otimes Y,
    Z)\cong \sSet(K,\cC(Y,Z)) \ar[d]\\
    \ast\cong \sSet(K,\ast) \ar[r] &\cC(K\otimes Y, X)\cong \sSet(K,\cC(Y, X))
  }
  \]
  where the isomorphisms follow from the compatibility with the tensor
  and simplicial mapping objects in $\cC$. Using these isomorphisms we
  see that this pullback is, up to canonical isomorphism, obtained by
  applying the functor $\sSet(K,-)$ to the pullback defining the mapping
  space 
  $\cC_{\downarrow X}(Y, Z)$.  So \[\sSet(K,\cC_{\downarrow X}(Y,Z))\cong
  \cC_{\downarrow X}(K\otimes Y, Z)\] as desired. 

  To check compatibility with the cotensor we apply the functor
  $\cC(Y,-)$ to the pullback diagram defining $Z^K_{\downarrow X}$ to
  obtain the right hand square in the following diagram:
  \[
  \xymatrix{
    \cC_{\downarrow X}(Y, Z^{K}_{\downarrow X})\ar[r]\ar[d] & \cC(Y, Z^{K}_{\downarrow X})\ar[r]\ar[d] &\cC(Y, Z^K) \cong \sSet(K,\cC(Y,Z)) \ar[d]\\
    \ast \cong \sSet(K,\ast) \ar[r] & \cC(Y,X) \ar[r] &\cC(Y, X^K)\cong
    \sSet(K,\cC(Y,X))}
  \]
  The left hand square is the defining pullback for the hom-set and we
  apply the reasoning of the previous paragraph to conclude
\[\sSet(K,\cC_{\downarrow X}(Y,Z))\cong
  \cC_{\downarrow X}(Y, Z^{K}_{\downarrow X}).\]
\end{proof}

\saveblockG{}
It is straightforward to see that with the simplicial structure above
the forgetful functor $U\colon \cC_{\downarrow X}\rightarrow \cC$ is
simplicial and admits a simplicial right adjoint 
\[Z\mapsto (Z\times X\xrightarrow{\pi_{X}} X). \]

\begin{rmk*}
  \Cref{prop:simplicial-over-categories} clearly dualizes to show that
  undercategories in simplicial categories are also simplicial categories
  in a natural way.
\end{rmk*}

}}

\newcommand{\completeblockM}{\completeblock[M]{
  \begin{proof}
  Since all limits and colimits in $\sSet^S$ are computed level-wise, and
  sifted colimits commute with products in $\Set$, we see that sifted
  colimits commute with products in $\sSet^S$. It follows that our
  discussion of limits and colimits in \Cref{rem:colimits-in-theories}
  extends to this context and $T$ commutes with reflexive coequalizers
  and filtered colimits, so we can apply
  \Cref{prop:bicomplete-cat-algebras}.
  
  The only remaining claim to check is the explicit identification of the
  tensor. By \eqref{eqn:tensor} we see 
  \[ 
  K\otimes X =\coeq\big[F_T(K\otimes UX)\leftleftarrows F_T(T(K\otimes UX
  ))\big]. 
  \] 
  Where this coequalizer can be calculated in simplicial sets; in
  particular it can be evaluated levelwise. 
  Now the tensor in simplicial sets is just cartesian product. Since
  this is part of a closed cartesian monoidal structure we obtain  
  \[
  (K\otimes UX)_n \cong \coprod_{k\in K_n} X_n. 
  \]
  So our defining coequalizer diagram for the tensor product levelwise
  is of the form in \eqref{eqn:colimits-of-t-algebras} which is the
  defining coequalizer diagram for the coproduct in
  $\model{\Set}{\cT}$.
\end{proof}
}}

\newcommand{\completeblockN}{\completeblock[N]{ An \emph{operad} is a symmetric sequence such
that for each partition of $n$ with $k$ parts, $i_1 + \cdots + i_k =
n$, there is a structure map
\begin{align*}
  C(k) \times C(i_1) \times \cdots \times C(i_k) & \longrightarrow C(n).\\
  (f , g_1, \ldots , g_n) & \longmapsto f\{g_1, \ldots , g_k\}
\end{align*}
These structure maps satisfy the following axioms \cite{May72}:
\saveblockH{} 
\begin{hypothenumerate}
\item For $g \in C(n)$, $1\{g\} = g$.
\item For $f \in C(k)$, $f\{1, \ldots , 1\} = f$.
\item For $f \in C(k)$, $g_j \in C(i_j)$, and $h_{j,l} \in
  C(m_{j,l})$,
  \[
  f\{g\{h_{1,1}, \ldots, h_{1,i_1}\}, \ldots, f\{g\{h_{k,1}, \ldots, h_{k,i_k}\}\}
  = (f\{g_1, \ldots, g_k\})\{h_{1,1}, \ldots, h_{k, i_k}\}
  \]
\item For $f$ and $g$ as above and $\sigma \in \Sigma_k$,
  \[
  (f\sigma)\{g_1, \ldots, g_k\} = f\{g_{\sigma(1)}, \ldots,
  g_{\sigma(k)}\}
  \]
\item For $f$ and $g$ as above and $\sigma_j \in \Sigma_{i_j}$,
  \[
  f\{g_1 \sigma_1 , \ldots , g_k \sigma_k\} = (f\{g_1, \ldots,
  g_k\}) \circ (\sigma_1 \times \cdots \times \sigma_k).
  \]
\end{hypothenumerate}

Given our emphasis on monads, it will be helpful to have an
alternative definition which we now summarize---additional details are
found in \cite[\S~11]{Rez97}. 
}}

\newcommand{\completeblockO}{\completeblock[O]{
\begin{proof}
  To unclutter the notation we will surpress the symmetric monoidal
  functor $i$ in the arguments below.

  Let $T_{\sO}$ be the monad associated to the operad $\sO$ as in \eqref{eq:monad-associated-to-operad}.  Suppose that $I$ indexes either a reflexive coequalizer or filtered colimit
  diagram, depending on which part of hypothesis \eqref{it:tensor-commutes}
  holds.  Since colimits commute with coproducts, we see that
  $T_{\sO}$ commutes with colimits over $I$:
  \begin{align*}
    T_{\sO}\colim_{I} X_i  &= \coprod_{n\geq 0} \sO(n)\otimes
    (\colim_{I} X_i)^{\otimes n }\\
      &\cong \coprod_{n\geq 0} \sO(n)\otimes \colim_{I} ( X_i^{\otimes n })\\
      &\cong \colim_{I} \left(\coprod_{n\geq 0} \sO(n)\otimes
      X_i^{\otimes n}\right)\\
      &= \colim_{I} T_{\sO} X_i.
  \end{align*}
  Hence by \cref{prop:bicomplete-cat-algebras}, $T_{\sO}$ has a
  bicomplete category of algebras.

  Similarly, we obtain natural maps $\eta_{K,X}$ via the following
  composite
  \begin{alignat*}{2}
    K\otimes T_{\sO} X  & = K\otimes \coprod_{n\geq 0} \sO(n)\otimes
    X^{\otimes n } & &\\
      & \cong \coprod_{n\geq 0} K\otimes (\sO(n)\otimes X^{\otimes n
      }) 
      & \;\longrightarrow & \coprod_{n\geq 0} K^n \otimes (\sO(n)\otimes
      X^{\otimes n})\\
      & &\cong & \coprod_{n\geq 0} \sO(n)\otimes
      (K\otimes X)^{\otimes n}\\
      & & = & T_{\sO} (K\otimes X).
  \end{alignat*}
  Pullback along these maps defines a morphism of simplicial mapping spaces given
  on $k$-simplices by 
  \[
  \sC(\Delta^k\otimes X,Y) \xrightarrow{T_{\sO}}
  \sC(T_{\sO}(\Delta^k\otimes
  X),T_{\sO}Y)\xrightarrow{\eta_{\Delta^k,X}^*}\sC(\Delta^k \otimes
  T_{\sO}X, T_{\sO}Y).
  \]
  This map clearly preserves simplicial units.  Compatibility with
  composition is verified by the commutative diagram in
  \cref{fig:operad-monad-compatibility-with-composition} where we let
  $d\colon \Delta^k\rightarrow \Delta^{k}\times \Delta^{k}$ denote the
  diagonal map.

\begin{figure}
  \[ \resizebox{15cm}{!}{ \xymatrix
    @C=-85pt
    @R=40pt
    {
      \sC(\Delta^k\otimes Y, Z)\times \sC(\Delta^k\otimes X, Y)
      \ar[dd]_{1\times (\Delta^k\otimes 1)} 
      \ar[rd]^{T_{\sO}}& &
      \sC(\Delta^k\otimes T_{\sO}(Y), T_{\sO}Z)\times
      \sC(\Delta^k\otimes T_{\sO}(X), T_{\sO}Y)
      \ar[dd]^{1\times (\Delta^k\otimes 1)} \\ 
      & 
      \sC(T_{\sO}(\Delta^k\otimes Y), T_{\sO}Z)\times
      \sC(T_{\sO}(\Delta^k\otimes X), T_{\sO}Y)
      \ar[dd]_(.3){1\times T_{\sO}(\Delta^k\otimes 1)} 
      \ar[ru]^{\eta^* \times \eta^*} & \\
      \sC( \Delta^k\otimes Y, Z)\times \sC( (\Delta^k\times
      \Delta^k)\otimes X, \Delta^k\otimes Y)
      \ar[dd]_{\circ} 
      \ar[rd]^{T_{\sO}} & &
      \hspace{-3ex}\sC( \Delta^k\otimes T_{\sO} Y, T_{\sO}Z)\times \sC( (\Delta^k\times
      \Delta^k)\otimes T_{\sO}X, \Delta^k\otimes T_{\sO} Y)
      \ar[dd]^{\circ}\hspace{3ex} \\ 
      &
      \hspace{2ex}\sC( T_{\sO}(\Delta^k\otimes Y), T_{\sO}Z)\times \sC( T_{\sO}((\Delta^k\times
      \Delta^k)\otimes X), T_{\sO}(\Delta^k\otimes Y))
      \ar[dd]_{\circ} \hspace{-2ex}& \\
      \sC( (\Delta^k\times \Delta^k)\otimes X, Z)
      \ar[dd]_{d^*}
      \ar[rd]^{T_{\sO}} & & 
      \sC( (\Delta^k\times \Delta^k)\otimes (T_{\sO}X), T_{\sO}Z)
      \ar[dd]^{d^*}\\ 
      &
      \sC( T_{\sO}((\Delta^k\times \Delta^k)\otimes X), T_{\sO}Z)
      \ar[dd]_{d^*}
      \ar[ru]^{\eta^*} & \\
      \sC( \Delta^k\otimes X, Z)
      \ar[rd]^{T_{\sO}} & &
      \sC( \Delta^k\otimes T_{\sO}(X), T_{\sO}Z)\\
      &
      \sC( T_{\sO}(\Delta^k\otimes X), T_{\sO}Z)
      \ar[ru]^{\eta^*} &      
    }} \]  
\caption*{Diagram to verify compatibility of $T_{\sO}$ with
  composition.}
\label{fig:operad-monad-compatibility-with-composition}
\end{figure}

  To verify that the structure maps for $T_{\sO}$ are simplicial natural
  transformations, one checks that the following diagrams
  commute where $e$ and $\mu_T$ are the unit and multiplication map for
  our monad $T_{\sO}$ and $f\in \sC(\Delta^k\otimes X,Y)$:

  \[\xymatrix{
   & \sC(T_{\sO}(\Delta^k \otimes X),T_{\sO}Y)\ar[d]^{\eta^*} 
   \\ 
   \sC(\Delta^k\otimes X,Y) \ar[dr]^{e}\ar[ur]^{T_{\sO}} &
   \sC(\Delta^k \otimes (T_{\sO}X),T_{\sO}Y)\ar[d]^{e^*}\\
    & \sC(\Delta^k \otimes X,T_{\sO}Y)
   }
  \]
  \[\xymatrix{
    \Delta^k\otimes T^2_{\sO} X \ar[r]^{\eta} \ar[d]^{\mu_T} & T_{\sO}(\Delta^k\otimes
    T_{\sO}X) \ar[r]^{\eta} & T^2_{\sO} (\Delta^k \otimes X) \ar[d]^{\mu_T}\ar[r]^{T_{\sO}^2 f} &
    T_{\sO}^2 Y\ar[d]^{\mu_T} \\
    \Delta^k\otimes T_{\sO} X \ar[rr]^{\eta} & &T_{\sO} (\Delta^k \otimes X)
    \ar[r]^{T_{\sO}f} & T_{\sO} Y
  } \]

\end{proof}

}}

\newcommand{\completeblockP}{\completeblock[P]{
\begin{defn*}\label{def:retract}
  A morphism $f\in \cC(A,B)$ is a {\em{retract}} of a morphism
  $g\in \cC(C,D)$ provided there are morphisms making the following
  diagram commute.
  \[\xymatrix{
  A\ar[d]^f\ar[r]\ar@/^/[rr]^{id} & C\ar[d]^g\ar[r] & A\ar[d]^f \\
  B\ar[r]\ar@/_/[rr]_{id} & D\ar[r] & B 
  }\]
\end{defn*}

\begin{defn*}
  Let $f \cn A \to B$ and $g \cn C \to D$ be maps in $\sC$.  We say
  that $f$ has the \emph{left lifting property} (LLP) with respect to
  $g$ or equivalently that $g$ has the \emph{right lifting property}
  (RLP) with respect to $f$ if for any maps $A \to C$ and $B \to D$
  making the square below commute, there is a lift of $f$ indicated by
  the dashed arrow commuting with the other maps in the diagram.
  \[\xymatrix{
    A \ar[d]_-{f} \ar[r] & C \ar[d]^-{g} \\
    B \ar[r] \ar@{-->}[ur] & D
  }\]
\end{defn*}

\begin{defn*}\label{def:model-category}
  A {\em{model category}} is a category $\cC$ with three distinguished
  classes of morphisms, the weak equivalences, the fibrations, and the
  cofibrations. Each of these classes contains all identity morphisms
  and is closed under composition. Moreover, these morphisms satisfy
  the axioms listed below:
  \begin{itemize}
	\item[MC1:]\label{it:MC1} $\cC$ is bicomplete.
  	\item[MC2:] If $f$ and $g$ are two composable morphisms in $\cC$
	  then if any two of $f, g,$ and $gf$ are weak equivalences then so
	  is the third.
	\item[MC3:] Weak equivalences, fibrations, and cofibrations are
	  closed under retracts.
	\item[MC4:] \begin{itemize}
		\item Every cofibration has the LLP with respect to all acyclic
		  fibrations, i.e., fibrations which are also weak equivalences.
		\item Every fibration has the RLP with respect to all acyclic
		  cofibrations, i.e., cofibrations which are also weak
		  equivalences.
	  \end{itemize}
	\item[MC5:]\label{it:MC5} Any map can be factorized through an
	  acyclic cofibration followed by a fibration and by a cofibration
	  followed by an acylic fibration.
  \end{itemize}
\end{defn*}

}}

\newcommand{\completeblockQ}{\completeblock[Q]{
\begin{defn*}
  A functor between two model categories is called \emph{left Quillen}
  if it is a left adjoint and preserves cofibrations and acyclic
  cofibrations. It is called \emph{right Quillen} if it is a right
  adjoint and preserves fibrations and acyclic fibrations. An adjoint
  pair is called a \emph{Quillen adjunction} if the left
  adjoint is left Quillen or, equivalently, the right adjoint is right Quillen.
\end{defn*}

}}

\newcommand{\completeblockR}{\completeblock[R]{
In the interest of being self-contained we recall
some of the relevant definitions and results from
\cite[\S~2.1]{Hov99}.
\begin{defn*}\label{def:I-injective}
  Let $I$ be a class of maps in $\cC$. 
  \begin{itemize}
  \item A map is $I${\em{-injective}} if it has the RLP with respect
    to every map in $I$. The class of $I$-injective maps is denoted
    $I$-inj.
  \item A map is $I${\em{-projective}} if it has the LLP with respect
    to every map in $I$. The class of $I$-projective maps is denoted
    $I$-proj.
  \item A map is an $I${\em{-cofibration}} if it has the LLP with
    respect to every map in $I$-inj.  The class of $I$-cofibrations
    is $(I$-inj)-proj and denoted $I$-cof.
  \item A map is an $I${\em{-fibration}} if it has the RLP with
    respect to every map in $I$-proj. The class of $I$-fibrations 
    is $(I$-proj)-inj and denoted $I$-fib.
  \end{itemize}
\end{defn*}

Recall that for a set of maps $I$ in a category $\cC$ a
{\em{relative}} $I${\em{-cell complex}} is a transfinite composition
of pushouts of elements of $I$. We say that an object is an $I${\em{-cell
complex}} if the map from the initial object to it is a relative
$I$-cell complex. If $I$ is the set of generating cofibrations then we
just say that the map is a subcell complex and we say an object is
cellular if the map from the initial object to it is a subcell complex. 

\begin{defn*}\label{def:relative-smallness}
  Let $I$ be a collection of morphisms in a cocomplete category $\cC$ 
  and $\kappa$ a cardinal. We say that $A\in \cC$ is
  $\kappa${\em{-small relative to }}$I$ if, for all $\kappa$-filtered
  ordinals $\lambda$ and all $\lambda$-sequences $X_i$ such that each
  map $X_{\beta}\rightarrow X_{\beta+1}\in I$ for $\beta+1<\lambda,$
  the map of sets 
  \[
  \colim_{\beta<\lambda}\cC(A,X_{\beta})\rightarrow
 \cC(A,\colim_{\beta<\lambda}X_{\beta})
  \]
  is an isomorphism. 
  
  We say that $A$ is {\em{small relative}} to
  $I$ if it is $\kappa$-small relative $I$ for some
  $\kappa.$ 
\end{defn*}

\begin{defn*}\label{def:cofibrantly-generated}
  A model category $\cC$ is {\em{cofibrantly generated}} provided
  there are sets of maps $I$ (called the {\em{generating cofibrations}})
  and $J$ (called the {\em{generating trivial cofibrations}}) such that
  \begin{itemize}
  	\item The domains of the maps of $I$ are small relative to
	  $I$-cell.
  	\item The domains of the maps of $J$ are small relative to
	  $J$-cell.
	\item The class of fibrations is $J$-inj.
	\item The class of trivial fibrations is $I$-inj.
  \end{itemize}
\end{defn*}

With this terminology available, we can now state conditions under which
the maps in $I$ and $J$ generate a model structure.
\begin{thm*}{\cite[Thm.~2.1.19]{Hov99}}
  \label{thm:identifying-cofibrantly-generated-model-categories}
  Suppose $\cC$ is a bicomplete category with a subcategory $\cW$ and
  distinguished sets $I$ and $J$ of maps, then there is a cofibrantly
  generated model structure on $\cC$ with $I$ (resp.~$J$) the generating
  (trivial) cofibrations, and $\cW$ the weak equivalences if and only
  if the following conditions are satisfied:
  \begin{hypothenumerate}
  \item \label[property]{enum:two-out-of-three} The subcategory $\cW$
    satisfies the two out of three property and is closed under
    retracts.
  \item \label[property]{enum:i-small} The domains of $I$ are small relative
    to $I$-cell.
  \item \label[property]{enum:j-small} The domains of $J$ are small relative
    to $J$-cell.
  \item \label[property]{enum:i-injectives} $I\text{-inj} \subset (\cW\cap
    J\text{-inj})$.
  \item \label[property]{enum:j-injectives} Either $(\cW\cap I\text{-cof})
    \subset J\text{-cof}$ or $(\cW\cap J\text{-inj}) \subset
    I\text{-inj}$.
  \item \label[property]{enum:acyclic-cells} $J\text{-cell} \subset (\cW\cap
    I\text{-cof})$.
  \end{hypothenumerate}
\end{thm*}

\begin{prop*}\label{prop:overcategory-is-a-model-category}
  Suppose that $\cC$ is a bicomplete simplicial model category and
  $X\in \cC$. Then the overcategory $\cC_{\downarrow X}$ is a
  bicomplete simplicial model category such that a morphism is a
  cofibration, resp.~fibration, resp.~weak equivalence if and only if
  it is a cofibration, resp.~fibration, resp.~weak equivalence in
  $\cC$. Moreover if $\cC$ is cofibrantly generated then so is
  $\cC_{\downarrow X}$.
\end{prop*}
\begin{proof}

  By \Cref{prop:simplicial-over-categories} we know that
  $\cC_{\downarrow X}$ is a bicomplete simplicial category. With
  cofibrations, fibrations, and weak equivalences defined as in the
  statement, verifying the axioms in \Cref{def:model-category} is
  immediate given that they are satisfied in $\cC$.

  From \Cref{prop:simplicial-over-categories} we know that the
  forgetful functor $\cC_{\downarrow X}\rightarrow \cC$ preserves
  simplicial tensors and pushouts. So the induced pushout-product map
  from Axiom [SM7a] in $\cC_{\downarrow X}$ is just the
  pushout-product map in $\cC$.  Since a map in $\cC_{\downarrow X}$
  is a cofibration, resp.~weak equivalence if and only if the map in
  $\cC$ is such, we see that this model structure is simplicial.

  If $I$ is a set of generating (acyclic) cofibrations in $\cC$, then
  we set the generating (acyclic) cofibrations in $\cC_{\downarrow X}$
  to be the set of diagrams of the form 
  \[ \xymatrix{ Y\ar[rr]^{f}\ar[rd] & & Z\ar[ld]\\ & X & } \]
  where $f\in I$. If $A$ is the domain of a generating cofibration in
  $\cC_{\downarrow X}$ and $Z=\colim_{\beta<\lambda} Z_{\beta}$ is an
  $I$-cell, then we need to verify that the canonical map
  \[
  \iota\colon \colim_{\beta<\lambda}\cC_{\downarrow X}(A,Z_{\beta})\rightarrow
  \cC_{\downarrow X}(A,\colim_{\beta<\lambda}Z_{\beta})
  \] 
  is an isomorphism. Since colimits in $\cC_{\downarrow X}$ may be
  calculated in $\cC$, an $I$-cell diagram in $\cC_{\downarrow X}$
  forgets to a corresponding cell diagram in $\cC$.  The hom sets in
  $\cC_{\downarrow X}$ are computed by a pullback of those in $\cC$
  and filtered colimits commute with finite limits in $\Set$. Now
  $\iota$ is an isomorphism because $A$ is small relative to $I$-cell
  diagrams in $\cC$.

  The remaining lifting conditions of \Cref{def:cofibrantly-generated}
  are immediate from the definitions of the generating (acyclic)
  cofibrations, fibrations, and weak equivalences.
\end{proof}
}}

\newcommand{\completeblockS}{\completeblock[S]{
\begin{proof}
  For the model category structure we apply
  \Cref{thm:identifying-cofibrantly-generated-model-categories}. 
  By \Cref{prop:bicomplete-cat-algebras} we know that $\cC_T$ is a
  bicomplete category. Since we want $U$ to right induce the model
  structure we set the weak equivalences to be those maps in $\cC_T$
  which project to weak equivalences in $\cC$. We set $F_T I$ and $F_T
  J$ to be the generating cofibrations and acyclic cofibrations. 

  Property \eqref{enum:two-out-of-three} is satisfied in $\cC_T$ because it is
  satisfied in $\cC$. While \eqref{enum:i-small} and \eqref{enum:j-small} are
  satisfied by assumption. By applying the adjunction we see the
  classes $F_TI$-inj and $F_TJ$-inj are precisely those maps in
  $\cC_T$ that map to $I$-inj and $J$-inj respectively under
  $U$. Consequently, property \eqref{enum:i-injectives} and the second condition
  in property \eqref{enum:j-injectives} hold in $\cC_T$ because they hold in
  $\cC$.

  To verify property \eqref{enum:acyclic-cells} we note that by assumption
  $F_TJ$-cell$\subset \cW$. To see that every $F_T J$-cell is in $F_T
  I$-cof we must show that each such map has the LLP with respect to
  $F_T I$-inj. By property \eqref{enum:i-injectives} we know these maps are $F_T
  J$-inj so each map in $F_T J$ has the LLP with respect to these
  maps, which we can use to inductively construct lifts in cellular
  pushout diagrams of the following form, where $f$ is in $F_T I$-inj
  and the lefthand square is a pushout:
  \[ \xymatrix{ 
    {\coprod} F_T A_{i} \ar[d]  \ar[r] & X_{\alpha}\ar[d] \ar[r]&
    Y \ar[d]^{f}\\ 
    {\coprod} F_T B_i\ar[r]\ar@{-->}|(.565){\hole}[urr] &
    X_{\alpha+1}\ar[r]\ar@{-->}[ur] & Z 
  }\]
  This and transfinite induction can then be used to prove we get a lift
  in the general cellular diagram:
  \[ \xymatrix{ X_{-1}\ar[d] \ar[r]& Y \ar[d]\\
     \colim_{\beta<\lambda} X_{\beta}\ar[r]\ar@{-->}[ur] & Z
  }\]
  as desired.

  All that remains is to check Axiom (SM7) which is immediate
  since cotensors, fibrations, and weak equivalences are calculated in
  $\cC$, which we assumed satisfied this axiom.
\end{proof}
}}

\newcommand{\longsaveblockT}{\longsaveblock[T]{
  Assuming the fibrant replacement functor we need to show that a $J$
  cell diagram $i\colon X\rightarrow Y$ is a weak equivalence. We will show
  this by constructing, up to a zig-zag of weak equivalences, a
  morphism $Y\rightarrow X$ such that the two composites are represent
  the identity map in the homotopy category. We can use the
  small object argument to construct  lifts 
  \[\xymatrix{ 
    X\ar[d]\ar[r] & QX\ar[d]\\ 
    Y\ar[r]\ar[ur]^{f} & \ast
  }\]
  The map $f$ is our map back and it factors the canonical map
  $X\rightarrow QX$ which represents the identity in the homotopy
  category. The other composite is $Y\rightarrow QX\xrightarrow{Qi}
  QY$ we will show that this map is simplicially homotopic (using the
  path object) to the map $Y\rightarrow QY$ representing the identity.
  To do so we need to construct a lift of the map $Y\rightarrow
  QY\times QY$, which is our given composite on one component and the
  `identity' map on the other, to $QY^{\Delta^1}$ this follows since
  by assumption (or the definition of the cotensor structure and the
  class of fibrations) the map $QY^{\Delta^{1}}\rightarrow QY\times QY$  is a
  fibration so our acyclic cofibration lifts. Since both composites
  are simplicially homotopic to the identity and $QY$ is fibrant they
  represent the identity in the homotopy category of $\cC$ from which
  it follows that the original map is a cofibration.
}}

\newcommand{\completeblockU}{\completeblock[U]{
  \begin{hypothenumerate}
  \item $B_nX=(F_TU)^{n+1}X$.
  \item The $i$th face map
    \[ 
    d_i\colon (F_TU)^{n+1}X\rightarrow (F_T U)^{n}X,\ 0\leq i \leq n
    \] 
    is obtained by applying the counit in position $i+1$:
    \[
    (F_TU)^i \; F_TU \; (F_TU)^{n-i} X \to 
    (F_TU)^i \; \Id \; (F_TU)^{n-i} X.
    \]
  \item The $i$th degeneracy map
    \[ 
    s_i\colon (F_TU)^{n+1}X\rightarrow (F_T U)^{n+2}X,\ 0\leq i\leq n
    \] 
    is obtained by applying the unit between positions $i$ and $i+1$:
    \[ 
    F_T(UF_T)^{i} \; \Id \: (UF_T)^{n-i}UX\rightarrow
    F_T(UF_T)^{i} \; UF_T \; (UF_T)^{n-i}UX.
    \]
  \end{hypothenumerate}
  It is straightforward to check the simplicial identities from the
  monad structure on $T=UF_T$ and the $T$-algebra structure on $X$.
  
}}

\newcommand{\completeblockV}{\completeblock[V]{
\begin{defn*}\label{def:geometric-realization}
  Let $\cC$ be a cocomplete simplicial category. If $X_\bullet$ is a
  simplicial object in $\cC$, the \emph{geometric realization} of
  $X_\bullet$ is the following coequalizer:
  \begin{align*}
    |X_\bullet|_\cC &= \int^{[n]\in \Delta} X_{n}\otimes \Delta^n\\
    &= \coeq \left(
      \coprod_{[k]\rightarrow [n]}X_n\otimes
      \Delta^k\rightrightarrows \coprod_{[n]} X_n\otimes \Delta^n
    \right)
  \end{align*}
  where the two maps in the coequalizer diagram are given by applying
  $X_{\bullet}$  and $\Delta^{\bullet} $ respectively to each morphism
  $[k]\rightarrow [n]$. 
  
  The $m$-skeleton of $|X_{\bul}|$ is defined to be the analogous
  coequalizer
  \begin{align*}
    \mathrm{sk}_m|X_\bullet|_\cC &= \int^{[n]\in \Delta_{\leq m}} X_{n}\otimes \Delta^n\\
    &= \coeq \left(
      \coprod_{\genfrac{}{}{0pt}{}{[k]\rightarrow [n]}{ n\leq m}}X_n\otimes
      \Delta^k\rightrightarrows \coprod_{\genfrac{}{}{0pt}{}{[n]}{n\leq m}} X_n\otimes \Delta^n
    \right)
  \end{align*}
  where $\Delta_{\leq m}$ is the full subcategory of $\Delta$ containing
  the objects $[0],[1],\ldots,[m]$.
\end{defn*}

}}

\newcommand{\completeblockW}{\completeblock[W]{
\saveblockI{}
To show that $T$ commutes with geometric
realization we use the fact that
geometric realization commutes with finite products and Fubini's theorem
for iterated coends as follows:
\begin{align*}
  T|UX_{\bul}|_{\cC} 
  &= \int^{\underline{m}\in \FinSet} |UX_{\bul}|_{\cC}^{m} 
  \otimes \cT(i(\underline{m}),i(\underline{1}))\\
  &\cong\int^{\underline{m}\in \FinSet}
  |UX_{\bul}^{m}|_{\cC}
  \otimes \cT(i(\underline{m}),i(\underline{1}))\\
  &= \int^{\underline{m}\in \FinSet} \left(\int^{n\in \Delta}
  (UX_n^{m}\otimes \Delta^n)\right)\otimes 
  \cT(i(\underline{m}),i(\underline{1}))\\ 
  &\cong \int^{\Delta} \left( \int^{\underline{m}\in
    \FinSet} UX_n^{m}\otimes \cT(i(\underline{m}),i(\underline{1}))\right)
  \otimes \Delta^n\\
  &= |TUX_{\bul}|_{\cC}.
\end{align*}

Note that the commutation of products with geometric realization
plays a key role in the above result. This is easy to verify in the
case where $\cC$ is simplicial sets, since the geometric realization
of a bisimplicial set is isomorphic to its diagonal. The result is
non-trivial but still true in the case of compactly generated weak
Hausdorff spaces. From these cases one can deduce that the smash
product on simplicial objects in pointed simplicial sets, compactly
generated pointed spaces, or categories of spectra built from these
categories also commutes with geometric realization. As a
consequence similarly defined monads will also commute with
geometric realization.

}}

\newcommand{\longsaveblockY}{\longsaveblock[Y]{
\begin{lem*}
  The morphisms of $\Delta_0$ are generated by all of the surjections
  $s_i : [n] \to [n-1]$ for $0 \leq i \leq n-1$, and by the injections
  $d_i : [n] \to [n+1]$ for $0 < i \leq n+1$ and for all $n$.
\end{lem*}
\begin{proof}
  Certainly $s_i(0) = 0$ for all $i$, and $d_i(0) = 0$ when $i \neq
  0$, so $\Delta_0$ contains these morphisms.  Conversely, suppose
  $\tau \in \Delta_0$ so $\tau(0) = 0$.  Let $\tau_+$ and $\tau_-$ be
  the uniquely determined maps in $\Delta$ such that $\tau_+$ is an
  injection, $\tau_-$ is a surjection, and $\tau = \tau_+ \tau_-$.
  Note that $\tau_- \in \Delta_0$ because $\Delta_0$ contains every
  surjection of $\Delta$.  Since $\tau(0) = 0$, we must have
  $\tau_+(0) = 0$ and hence we can factor $\tau_+$ as a sequence of
  coface maps $d_i$ with $i > 0$.
\end{proof}

Note that since $\Delta_0$ is obtained from $\Delta$ by omitting the
surjections dual to the $d_0$ maps, the Reedy structure on $\Delta$
restricts to a Reedy structure on $\Delta_0$. Hence
$\cC^{\Delta_0^{\op}}$ has a Reedy model structure.

\begin{lem*}
  The category $\Delta_0$ is Reedy.
\end{lem*}
\begin{proof}
  This is immediate from the Reedy structure of $\Delta$:  The \emph{degree}
  of an object $[n]$ is its maximal element, $n$, the \emph{direct} morphisms
  are the injections and the \emph{inverse} morphisms are the
  surjections.  The only morphisms which are both direct and inverse
  are the identities.  Just as in $\Delta$, every morphism in $A$ factors
  uniquely as a surjection followed by an injection.
\end{proof}

}}

\newcommand{\completeblockZ}{\completeblock[Z]{
\saveblockK{}
Let $X$ be a simplicial object or, more generally, a
$\Delta_{\mathrm{surj.}}^{\op}$-shaped diagram.  We define the partial
latching object for $0 \leq i < n$:
\begin{equation*}
L_n^i X_\bul = \coeq \left( 
\xymat@C=4pc{\displaystyle 
  \coprod_{\genfrac{}{}{0pt}{}{s_ks_j \colon [n] \twoheadrightarrow [n-2]}{k < j \leq i}} 
  X_{n-2} \ar[r]<1ex>^-{\coprod \text{incl}_k \circ s_{j-1}}
  \ar[r]<-1ex>_-{\coprod \text{incl}_j \circ s_{k}}
  & \displaystyle
  \coprod_{\genfrac{}{}{0pt}{}{s_j \colon [n] \twoheadrightarrow [n-1]}{j \leq i}}
  X_{n-1} } \right)
\end{equation*}
where $\text{incl}_i$ denotes inclusion to the $i$th term in
$\coprod_j X_{n-1}$.  The maps in this coequalizer encode the
simplicial identities $s_j s_k = s_k s_{j-1}$ for $k < j$.  For
example, $L^3_5 X_\bul$ is shown below:
\begin{figure}
  \begin{center}
    \includegraphics[scale=1]{old/latching-diagram}
  \end{center}
    \caption*{Cocone diagram for the partial latching object $L_5^3$.}
\end{figure}
These are special cases of the generalized latching objects discussed
in \cite[\S~VII.1]{GoJ99}.  Note that the full latching object $L_n
X_\bul$ is given by $i = n-1$ and each $L^i_n X_\bul$ is a pushout
\[\xymatrix{
L^{i-1}_{n-1} X_\bul \ar[r] \ar[d] & L^{i-1}_n X_\bul \ar[d] \\
X_{n-1} \ar[r] & L^i_n X_\bul
}\]
where the horizontal map out of $L^{i-1}_{n-1}$ is given by
\[
\coprod_{j \leq i-1} (s_{i-1} \cn X_{n-2} \to X_{n-1}),
\] 
the vertical map out of $L^{i-1}_{n-1}$ is given by 
\[
\coprod_{j \leq i-1} s_j \cn \coprod_{j \leq i-1} X_{n-2} \to X_{n-1},
\]
the right vertical map is given by an inclusion of
diagrams, and the bottom horizontal map is given by inclusion of
$X_{n-1}$ into $\coprod_{j \leq i} X_{n-1}$ at the position indexed by
$i$.

}}

\newcommand{\completeblockZA}{\completeblock[ZA]{
\saveblockL{}
\begin{lem*}
  \label{lem:partial-latching-subcell-inclusions}
  Let $I$ be a set of maps and $X_\bul$ a simplicial object in
  $\cC$.  Suppose that relative cell complexes in $\cC$ are
  monomorphisms. If each of the degeneracies $s_i$ is an $I$-cellular
  inclusion, then the latching maps of $X_\bul$ are $I$-cellular
  inclusions.
\end{lem*}
\begin{proof}
  Let $s^i X_n$ be the subcell complex which is the union of the cells
  for $X_{n-1}$ coming from the different inclusions $s_0,\ldots,s_i$
  and well defined by \cite[Prop.~10.6.10]{Hir03}. 

  Then by induction $s^i X_n$ fits into the following commutative square of
  subcomplexes of $X_n$:
  \begin{equation*}\label{eqn:pushout-square}\xymatrix{
    s^{i-1} X_{n-1} \ar[d] \ar[r] & s^{i-1} X_{n} \ar[d] \\
    X_{n-1} \ar[r] & s^i X_n
  }\end{equation*}
   We will show that this square is a pullback square and that the
   pullback models the combinatorial intersection of cells. By
   \cite[Prop.~12.2.3]{Hir03} this pullback square will also be a pushout
   square and hence $s^i X_n$ is the $i$th partial latching object and a
   subcomplex of $X_n$.

 To show \eqref{eqn:pushout-square} is pullback we consider the
 induced map from $s^{i-1}X_{n-1}$ to the actual pullback. Since the
 top horizontal map admits a retraction which is induced by $d_{i+1}$
 we obtain a map back from the $s^{i-1}X_{n-1}$. Using this map we
 verify $s^{i-1}X_{n-1}$ satisfies the pullback property. Suppose we
 had two maps to $s^{i-1}X_{n-1}$ that become equal after passing to
 either the $s^{i-1}X_n$ or $X_n$ then then they are equal because the
 top horizontal and left vertical maps are subcell inclusions and
 hence monomorphisms. It follows that \eqref{eqn:pushout-square} is a
 pullback square of subcell complexes. 

 To see $s^{i-1}X_{n-1}$ models the combinatorial inclusion, one
 applies transfinite induction as in \cite[Prop.~10.6.11]{Hir03}. The
 crucial point is that if we have a cell landing in $X_\beta$ which
 lies in the intersection of the two complexes it must land in
 $s^{i-1}X_{n-1}$ by the pullback property. Since
 \eqref{eqn:pushout-square} is a square of subcell complexes this map
 must land in the $\beta$ level of the filtration on
 $s^{i-1}X_{n-1}$.

  \saveblockM{}
  \ifshowsaveblocks

  Now each of the maps
  \[
  s_0X_n = s^0 X_n \to \cdots \to s^{n-2} X_n \to s^{n-1} X_n
  \]
  is a cellular inclusion.
  
  \[\xymatrix{
    s^{i-1}  X_{n-1} \ar[d] \ar[r] 
    & s^{i-1} X_n \cap s_i X_{n-1} \ar[d] \ar[r]
    & s^{i-1} X_n \ar[d]\\
    X_{n-1} \ar[r] 
    & s_{i} X_{n-1} \ar[r]
    & s^i X_n
  }\]
  where the intersection term is the subcellular object indexed by the
  intersection of the two indicated indexing sets in $X_n$.  The two
  left-most horizontal maps are isomorphisms induced by $s_i$ with
  isomorphisms induced by $d_i$.
  \fi
  \end{proof}

}}

\begin{document}

%%
%% Stuff for elsarticle class
%%
\begin{frontmatter}
  \title{Lifting homotopy $T$-algebra maps to strict maps}
  \date{June 30, 2014} %doesn't affect date in elsarticle

  \author{Niles Johnson}
  \address{
    The Ohio State University Newark;
    1179 University Drive;
    Newark, OH 43055;
    USA
  }
  \ead{\nilesemail}
  \ead[url]{\nilesurl}

  \author{Justin Noel\corref{CA}}
    \cortext[CA]{Corresponding author}
  \address{
    Universität Regensburg; 
    Universitätsstr. 31;
    Regensburg D-93040;
    Germany 
  }
  \ead{\justinemail}
  \ead[url]{\justinurl}
  \begin{abstract}
     The settings for homotopical algebra---categories such as simplicial
groups, simplicial rings, $A_\infty$ spaces, $E_\infty$ ring spectra,
etc.---are often equivalent to categories of algebras over
some monad or triple $T$. In such cases, $T$ is acting on a nice
simplicial model category in such a way that $T$ descends to a
monad on the homotopy category and defines a category of
\emph{homotopy} $T$-algebras. In this setting there is a forgetful
functor from the homotopy category of $T$-algebras to the category of
homotopy $T$-algebras.

Under suitable hypotheses we provide an obstruction theory, in the form of a
Bousfield-Kan spectral sequence, for lifting a homotopy $T$-algebra map to a
strict map of $T$-algebras. Once we have a map of $T$-algebras to serve as a
basepoint, the spectral sequence computes the homotopy groups of the space of
$T$-algebra maps and the edge homomorphism on $\pi_0$ is the aforementioned
forgetful functor. We discuss a variety of settings in which the required
hypotheses are satisfied, including monads arising from algebraic theories and
operads. We also give sufficient conditions for the $E_2$-term to be
calculable in terms of Quillen cohomology groups.

We provide worked examples in $G$-spaces, $G$-spectra, rational $E_\infty$ algebras,
and $A_\infty$ algebras. Explicit calculations, connected to rational unstable homotopy theory,
show that the forgetful functor from the homotopy category of $E_\infty$
ring spectra to the category of $H_\infty$ ring spectra is generally neither
full nor faithful.  We also apply a result of the second named author and Nick
Kuhn to compute the homotopy type of the space $\ei(\Sigma^\infty_+ \Coker J,
L_{K(2)} R)$.

  \end{abstract}

\end{frontmatter}

\maketitle
\tableofcontents

% Main Body

\completeblockZERO{}
\saveblockZERO{}

%Placed here to ensure logical naming order in the bibliography, all
% references are cited.
\nocite{Bor94,Bor94a}
\section{Introduction}

In the work of Ando, Hopkins, Rezk, and Strickland on the Witten genus
\cite{AHS01,AHS04,AHR06} the authors first construct a lift of the
Witten genus to a multiplicative map of cohomology theories, then to
an $H_\infty$ map (i.e., a map preserving power operations), and
finally to an $E_\infty$ map
\[
  M\mathit{String}\rightarrow \mathit{tmf}.
\]
In each of these steps they are asking that a map respects additional
structure and it is natural to ask if there are general techniques for
constructing such liftings.

Their construction of an $H_\infty$ map makes use of ideas from Ando's
thesis \cite{And92,And95}, where he defines $H_\infty$ maps from
complex cobordism to Lubin-Tate spectra using a connection to
isogenies of Lubin-Tate formal group laws. Their lift then arises from
a computation: Since the $H_\infty$ condition can be formulated in the
stable homotopy category, a map is $H_\infty$ if and only if an
associated sequence of cohomological equations hold. The applicability
of such techniques is one of the reasons that the category of
$H_\infty$ ring spectra is computationally more
accessible. Although every $E_\infty$ map forgets to an $H_\infty$ map,
constructing $E_\infty$ maps is much more subtle and requires rather
different techniques.

We construct an obstruction-theoretic spectral sequence to detect when
an $H_\infty$ map can be lifted to an $E_\infty$ map and other
problems of this type. As a consequence of our approach we can also
see how much information is lost under the passage from
$E_\infty$ to $H_\infty$ ring spectra. The first category can be
described as the category of algebras over a monad/triple $T$
in a category of spectra while the second is the category of such
algebras in the homotopy category. Phrased in these terms, it is
expected that a great deal is forgotten in the passage from $E_\infty$
to $H_\infty$ ring spectra. But to date, there have been no examples
demonstrating this. Since our methods apply more generally to studying
categories of algebras over a monad $T$ (satisfying some hypotheses),
we set up our machinery in the more abstract setting.

In \Cref{sec:t-algebras} we provide a rapid review of the theory of
monads and how they naturally encode algebraic structures. We
emphasize the examples coming from algebraic theories and from operads
since they make up the majority of our examples. In
\Cref{sec:model-structures-on-t-algebras}, we recall some conditions
which guarantee the existence of a simplicial model structure on the
category of algebras over a monad. These conditions are often
satisfied and cover a broad range of standard examples. We include
this standard material so the reader can easily apply it to the
application of their choosing.

Our first main result is:
\begin{thmA}
  \begin{samepage}
    Let $\cC$ be a simplicial model category and $T$ a
    simplicial Quillen monad (\Cref{def:quillen}) acting on $\cC$.
    Let $X$ and $Y$ be $T$-algebras.  Suppose that
    \begin{hypothenumerate}
      \item $T$ commutes with geometric realization and
      \item $X$ is resolvable with bar cofibrant replacement $\widetilde{X} \to X$
      (\Cref{def:resolvable}).
    \end{hypothenumerate}
  \end{samepage}
  Let $U\colon \cC_T\rightarrow \cC$ denote the forgetful functor from
  the category of $T$-algebras to $\cC$.  Then $T$ induces a monad $\h
  T$ on $\ho\cC$ and there
  exists an obstruction-theoretic spectral sequence, called the
  $T$-algebra spectral sequence, such that:
  \begin{enumerate} \setlength{\itemsep}{.4pc}
    \completeblockA{}%
    \item Provided a $T$-algebra map $\epz \cn X \to Y$ to serve as a
      base point, the spectral sequence conditionally converges to the
      homotopy of the derived mapping space
      \[
      \pi^s \pi_t \cC^d(T^{\bul}U\widetilde{X}, UY)\Longrightarrow \pi_{t-s}\cC_T^d(X,Y).
      \]
     \item In this case the differentials $d_r[f]$ provide obstructions to
       lifting $[f]$ to a map of $T$-algebras.
     \item The edge homomorphisms 
       \begin{align*}
	 \pi_0 \cC_T^d(X,Y)
	 &\twoheadrightarrow E_\infty^{0,0}\\
	 &\hookrightarrow E_2^{0,0} = (\ho\cC)_{\h T}(UX, UY)\\
	 &\hookrightarrow E_1^{0,0} = \ho\cC (UX, UY)
       \end{align*}
       are the corresponding forgetful functors.
       \item   If $\cC_T$ has functorial bar cofibrant and fibrant replacements, then the spectral sequence is contravariantly functorial in $X$ and covariantly functorial in $Y$.
       \item   A map of simplicial monads $T_1\rightarrow T_2$ satisfying the above hypotheses induces a contravariant map of spectral sequences provided that $X$ has a bar cofibrant replacement $\widetilde{X} \to X$ in $\cC_{T_2}$ such that $U_3 \widetilde{X}$ has a bar cofibrant replacement in $\cC_{T_1}$.
       \end{enumerate}
\end{thmA}
This result will be proven in \Cref{sec:main-theorem}. Note that since we
avoid using $E_2$ model structures or Bousfield localizations, we do
not require any properness assumptions on our model category. The
assumption that our monad is simplicial Quillen is innocuous and satisfied in 
practice. The remaining two assumptions guarantee convergence to the
desired target and allow us to identify the key terms in the spectral
sequence. 

Note that the convergence result in  \refthmA is stronger than that of
alternative approaches found in the literature, e.g., using a Reedy
cofibrant replacement of the bar resolution or taking homotopy colimits
instead of geometric realization.  These approaches give spectral
sequences which converge to mapping spaces from a $T$-cocompletion of
the source, as in \cite{Hes10,BR12}. We combine some standard results
recalled in \Cref{sec:simplicial-resolutions} with some crucial 
technical lemmas in \Cref{sec:reedy} to prove convergence without a
cocompletion under the assumptions of \refthmA.

In \Cref{sec:verification-of-examples} we show that these assumptions
hold in many general cases of interest such as nice categories of
algebras over an operad, $G$-spaces and $G$-spectra (provided $G$ is
sufficiently nice), and many algebraic categories such as simplicial
groups and rings.  In each of these examples, the resolvability conditions hold for every object, so the spectral sequence can be applied to any pair of objects in the category. 

This spectral sequence is a special case of the Bousfield-Kan spectral sequence. Bousfield has shown that this spectral sequence can be applied even without the existence of a base point---a useful generalization
since a space of $T$-algebra maps may well be empty. In this case
there is an obstruction theory (see \Cref{rem:obstructions}) for
lifting a map in $\cC$ to a map of $T$-algebras so that one can obtain
a base point \cite[\S~5]{Bou89}. The farther one can lift this base
point up the totalization tower, the greater the range in which one
can define the spectral sequence and differentials.

As shown in \Cref{thm:obstructions}, when the relevant mapping spaces in
$\cC$ have the homotopy type of $H$-spaces, e.g., if $\cC=\Spectra$,
then one can choose these obstructions to land in the $E_2$ page of the
spectral sequence. Under favorable circumstances we can then apply our second
main theorem, \refthmB of \Cref{sec:methodology}, to identify the
$E_2$ term with Quillen cohomology groups.

We then demonstrate the wide applicability of this spectral sequence and its
computability via \refthmB through a number of examples in \Cref{sec:computations}.
The reader interested in applications is encouraged to skip directly to
this section where we compute the homotopy groups of particular:

\begin{itemize}
\item Spaces of equivariant maps in $G$-spaces and $G$-spectra
  (\Cref{sec:g-action}).  This is a warm-up for the other examples.
  In two examples, we explicitly analyze the forgetful functor from the homotopy
  category of (strict) $G$-objects to (weak) $G$-objects in the
  homotopy category of spaces or spectra.

\item Spaces of $E_{\infty}$ maps between function spectra
  (\Cref{sec:examples-operads}).  In two examples arising from unstable rational
  homotopy theory, we show that the forgetful functor from $\ei$ to $\hi$
  ring spectra is generally neither full nor faithful.  To the authors'
  knowledge, these are the first such examples.
  
\item  Spaces of $A_\infty$ and $E_{\infty}$ self-maps of 
$Hk$-algebras, for $k$ a suitable field, whose homotopy rings are
polynomial algebras.

\item Spaces of $\ei$ maps from $\Sigma^{\infty}_+\Coker J$ to a
  $K(2)$-local $\ei$ ring spectrum (\Cref{sec:coker-j}).  This is a
  result of Nick Kuhn and the second named author, and gives
  a nontrivial example of when the set of $\hi$ maps coincides with the
  set of homotopy classes of $\ei$ maps.  As a consequence of the
  proof we obtain new information about $\Coker J$, including its
  $K(2)$-homology:
  \[
  K(2)_*\Coker J \cong \bigoplus_{n\geq 0}K(2)_*B\Sigma_n. 
  \]
\end{itemize}

\subsection*{Related work} 
The $T$-algebra spectral sequence arises by taking a functorial
resolution of the source $X$. Namely we replace $X$ by the two sided
bar construction $B(F_T,T,UX)$ where $U$ is the forgetful functor
$\cC_T\rightarrow \cC$ and $F_T$ is its left adjoint. For this
approach, one wants general conditions under which the replacement is
cofibrant, weakly equivalent to $X$, and equipped with a suitable
filtration for obtaining a spectral sequence.  A number of special
cases of this theory are well known, and the arguments for spaces and
spectra can be found in the literature.  Although the two-sided bar
construction has been a standard tool in homotopy theory for decades,
we know of no reference in which its homotopical properties
are developed with sufficient breadth for our purposes.  In 
\Cref{sec:reedy} we develop new tools for
this purpose and apply them in \Cref{sec:verification-of-examples} to
demonstrate the applicability of the $T$-algebra spectral sequence. 

There are a couple of alternative methods for constructing maps of
structured ring spectra. This work can be considered an extension of the
obstruction theory for maps of algebras in simplicial $R$-modules and
$A_\infty$ ring spectra that appears in Rezk's thesis \cite{Rez96} and
his presentation of the Hopkins-Miller theorem \cite{Rez97} respectively. Indeed
the latter work was a significant source of inspiration for this
project. Angeltveit \cite{Ang08a} has also constructed an obstruction
theory, which appears to be part of a spectral sequence, for computing
maps of $A_\infty$ ring spectra. 

The Goerss-Hopkins spectral sequence also computes the homotopy of the
derived mapping space between two spectra which are algebras over a
suitable operad, such as an $E_\infty$ operad \cite{GoH04,
  GoH05}. This spectral sequence uses an $E_2$ model structure
which guarantees an algebraic description of the $E_2$ term and is,
in general, \emph{not} the same as the $T$-algebra spectral sequence. In
particular, their edge homomorphism is generally a Hurewicz homomorphism
which usually is distinct from the forgetful functor above. In the
sequel \cite{Noe14} however, the second named author shows that in
special cases, such as the worked examples in \Cref{sec:examples-operads},
the spectral sequences do agree and computations can be done in either
framework.

When \refthmB does not apply, it is generally quite
difficult to determine the $E_2$ term of the $T$-algebra spectral
sequence. Indeed, when $T$ is the monad associated to the $E_{\infty}$
operad, the main results of \cite{AHS04,And95,JoN09b} could be
expressed as partial computations of $d_1\colon E_{1}^{0,0}\rightarrow
E_1^{1,0}$.  The difficulties here are generic: there are very few
examples where \refthmB does not apply yet one still has enough knowledge of the power operations to
compute the $E_2$ term explicitly.

\subsection*{Acknowledgements}
This work began when the authors were visiting the Max Planck
Institute for Mathematics. The authors would like to thank the MPIM
for its hospitality during their stay and the anonymous referee for
their careful reading of this paper and their many helpful comments.

The second named author would also like to thank Benoit Fress, Paul
Goerss, Stefan Schwede, Karol Szumilo, and Markus Szymik for helpful
conversations concerning the material below. The authors also
gratefully acknowledge partial support from the University of Bonn and
the Deutsche Forschungsgemeinschaft through Graduiertenkolleg 1150, as
well as the University of Georgia through VIGRE II.

\subsection*{Conventions/Terminology}

We will make the convention that a simplicial category is a simplicially
enriched category which is tensored and cotensored over simplicial sets.
This convention is standard when discussing simplicial model categories,
but unusual in enriched category theory. 

\completeblockB{}

\addtocounter{tocdepth}{1}
\section{Algebras over a monad}\label{sec:t-algebras}

This section reviews monads and their categories of algebras, focusing
on examples and conditions which ensure that limits and colimits in the
categories of algebras exist. The existence of these constructions is
not automatic, but will be essential for the material in
\Cref{sec:model-structures}.  We also show how these constructions are
computed in practice.

In \Cref{sec:monads} we begin with a familiar example, focusing on
points which are key to the general theory. A wealth of additional
examples can be found in the framework of algebraic theories which we
recall in \Cref{sec:algebraic-theories}. In \Cref{sec:simplicial-cat}
we extend this discussion to the simplicially enriched
context. Finally we recall some relevant facts about operads in 
\Cref{sec:monads-from-operads}. In these last two sections we introduce
two of our primary classes of examples: Simplicial algebraic theories
and operads.

\subsection{Monadicity and categories of algebras}\label{sec:monads}
Given a set $S$ we can form the free group $FS$ on $S$ whose
underlying set consists of all finite reduced words whose letters are
signed elements of $S$. 
Multiplication is then defined by composing words. We can also take a group
$G$, forget its group structure, and regard it is as a set $X=UG$. These
constructions are clearly functorial and participate in an adjunction 
\begin{equation*}
   \xymatrix{
   \Group\ar@<0.5ex>[r]^-{U}  & \Set\ar@<0.5ex>[l]^-{F}
 }
\end{equation*}
where $U$ is right adjoint to $F$. Let $T=UF$ denote
the endofunctor of $\Set$ given by the composite of these two functors.

The unit of this adjunction is a natural transformation
$e\colon \Id\rightarrow T$ given by taking an element of a
set to its associated word of length one. Using the underlying
group structure on $X=UG$ one can multiply the elements in a word to obtain
a \emph{structure map} 
\[
  \mu_{X}\colon TX\rightarrow X.
\] 
Alternatively we could construct this map by applying $U$ to the counit
\[
  \varepsilon\colon FU\rightarrow \Id
\] 
of this adjunction. In particular, we have such a map for anything in the
image of $T$ and obtain a natural transformation 
\[
  \mu_T\colon T^2\rightarrow T.
\] 
The (large) category of
endofunctors of $\Set$ admits a monoidal structure under composition and
we can see that $(T, e, \mu_T)$ is an associative monoid in this category,
in other words, $T$ is a \emph{monad} on $\Set$. 

In the case of $X=UG$ we see that the map $\mu_X$ is compatible with this
structure in the sense that the two double composites of straight
arrows in \eqref{eqn:resolution-of-t-algebras} are equal and each
composite of a curved
arrow followed by a straight arrow is the identity morphism.
\begin{equation}\label{eqn:resolution-of-t-algebras}
  \xymatrix{
  TTX\ar@<1ex>[rr]^{\mu_T}\ar@<-1ex>[rr]_{T\mu_X} & &
  TX\ar[rr]^{\mu_X}\ar@/_1.5pc/[ll]_{e_{TX}} & & X\ar@/_1.5pc/[ll]_{e_X}\\
  }
\end{equation}
An object $X\in \Set$ with a map $\mu_X\colon TX\rightarrow X$
satisfying these identities is called a $T$-algebra in $\Set$. We obtain a category
$\Set_T$ of $T$-algebras in $\Set$ by restricting to those set maps
which commute with the structure morphisms. To be explicit, the
morphisms between two $T$-algebras $(X, \mu_X)$ and $(Y, \mu_Y)$ are those
maps $f\colon X\rightarrow Y$ such that the following diagram commutes:
\[\xymatrix{
  TX\ar[r]^{Tf}\ar[d]_{\mu_X} & TY\ar[d]^{\mu_Y}\\
  X\ar[r]^{f} & Y
}\]
or alternatively: 
\begin{equation}\label{eqn:maps-of-t-algebras}
  \xymatrix@C=11ex{
 \Set_T\,(X,Y)= \eq\big[ \Set(X,Y)\ar@<1ex>[r]^-{(\mu_Y)_*\circ
 T}\ar@<-1ex>[r]_-{\mu_X^*}  & \Set(TX,Y) \big].\\
}
\end{equation} 

The category of $T$-algebras in $\Set$ admits an obvious forgetful
functor to $\Set$ and we saw above that the forgetful functor $U\colon
\Group\rightarrow \Set$ factors through $\Set_T$. It is not difficult to
see that the latter functor defines an equivalence of categories
$\Group\simeq \Set_T$.  Indeed, if $G$ is a group then we can see that
some of the maps in \eqref{eqn:resolution-of-t-algebras} can be realized
by applying $U$ to following diagram of groups:
\begin{equation}\label{eqn:resolution-of-t-algebras-above}
  \xymatrix@C=11ex{
  FTUG\ar@<1ex>[r]^-{\varepsilon FU}\ar@<-1ex>[r]_-{F\mu_{UG}} &
  FUG\ar[r]\ar@/_1.5pc/[l]_-{e} & G.\\
  }
\end{equation}
The map on the right exhibits $G$ as the coequalizer of the two
straight arrows on the left.  Moreover, the map $e$ exhibits this
coequalizer as a reflexive coequalizer. In this sense we see
that every group has a \emph{functorial resolution} by free
groups. The forgetful functor from $\Set_T$ to $\Set$ admits a left
adjoint $F_T$ which factors $T$ as $T = UF_T$.  Similarly, we see that
every $T$-algebra admits a functorial resolution by free $T$-algebras.
After forgetting down to $\Set$ these coequalizer diagrams become
split coequalizer diagrams \cite[Lem.~4.3.3]{Bor94a}, i.e.,
diagrams of the form \eqref{eqn:resolution-of-t-algebras}. Split
coequalizer diagrams have the useful property that they are preserved by
\emph{all} functors \cite[Prop.~2.10.2]{Bor94}.

Using these functorial resolutions and that a morphism of groups is an
isomorphism if and only if it induces an isomorphism between the
underlying sets, we can see that the lifted functor $U\colon \Group
\rightarrow \Set_{T}$ is essentially surjective. By applying the
functorial resolution again and \eqref{eqn:maps-of-t-algebras} one can
now see that this functor is full and faithful and $\Group \simeq \Set_T$.

These arguments are completely general: 
\begin{thm}[Barr-Beck/Monadicity]\label{thm:monadicity}
  Any functor $U\colon\cD\rightarrow \cC$ which admits a left adjoint
  $F$ lifts to a functor to the category of $T=UF$-algebras in $\cC$.
  Moreover this functor is an equivalence of categories if and only if
  \begin{hypothenumerate}
  \item \label{it:barr-beck-a} $U$ is conservative, i.e., a map
    $f$ in $\cD$ is an isomorphism if and only if $Uf$ is.
  \item \label{it:barr-beck-b} For every $T$-algebra $G$, if $U$ takes a pair of arrows of the
    form \eqref{eqn:resolution-of-t-algebras-above} to a split coequalizer, then the pair
    of arrows in \eqref{eqn:resolution-of-t-algebras-above} admits a
    coequalizer which is preserved by $U$.
  \end{hypothenumerate}
\end{thm}
\begin{proof}
  This version of the Barr-Beck theorem is a slight variation of
  \cite[Thm.~4.4.4]{Bor94a}. Here we assume the existence of a left adjoint,
  which does not appear there, and our condition \eqref{it:barr-beck-b} is
  slightly weaker than what is assumed there.  But Borceux's
  argument applies without change.
%% SAVE
  \saveblockA{}
\end{proof}

\Cref{thm:monadicity} can be used to identify many categories as
categories of algebras over a monad. Since we want $\cC_T$ to have an
ample supply of colimits and limits for constructing model structures we
postpone introducing these examples for a moment so that we can record
when such constructions exist. 

\begin{prop}\cite[Props.~4.3.1, 4.3.2]{Bor94a}\label{prop:monads-colimits-limits}
  Suppose $T$ is a monad acting on $\cC$, then
  \begin{enumerate}
	\item\label{enum:limits} The forgetful functor $U\colon
	  \cC_T\rightarrow \cC$ creates all limits. 
	\item\label{enum:colimits} The forgetful functor
	  $U\colon \cC_T\rightarrow \cC$ creates all colimits which
	  commute with $T$ in $\cC$.
  \end{enumerate}
\end{prop}
\begin{prop}\cite[Prop.~II.7.4]{EKMM97}\label{prop:coequalizers-imply-cocomplete}
  Suppose $\cC$ is cocomplete and $T$ commutes with reflexive
  coequalizers, then $\cC_T$ is cocomplete and the forgetful
  functor creates all reflexive coequalizers.
\end{prop}
\completeblockC{}

Alternatively, if we suppose that $\cC$ is bicomplete and $T$ preserves
$\kappa$-filtered colimits for some regular cardinal $\kappa$ then
$\cC_T$ is bicomplete by \cite[Prop.~4.3.6]{Bor94a}. We often want $T$, or
equivalently $U$, to preserve \emph{both} filtered colimits and
reflexive coequalizers (for some examples where this does not hold
see \cite[\S~4.6]{Bor94a}). In such a case we can apply the following
useful form of the Barr-Beck theorem provided we restrict to
\emph{locally presentable categories} \cite[\S~5.2]{Bor94a}.

\begin{prop}\label{prop:barr-beck-for-presentable-categories}
  Suppose $U\colon \cD\rightarrow \cC$ is a conservative functor between two locally
  presentable categories such that 
  \begin{hypothenumerate}
	\item $U$ preserves limits,
	\item $U$ creates $\kappa$-filtered colimits for some regular
	  cardinal $\kappa$,
	\item and $U$ creates reflexive coequalizers.
  \end{hypothenumerate}
  Then $U$ admits a left adjoint $F$, $\cD$ is equivalent to 
  the category of $T=UF$-algebras in $\cC$, and $T$ commutes with
  reflexive coequalizers and $\kappa$-filtered colimits. 
\end{prop}
\completeblockD{}

The above results illustrate the importance of reflexive coequalizers
and filtered colimits in $\cC_T$. These are particular
examples of \emph{sifted} colimits, which are colimits indexed over
$\cI$ such that the diagonal
map $\cI\rightarrow \cI\times \cI$ is final. Sifted colimits can also be
characterized as those colimits which commute with finite products in
$\Set$. One of the main results of \cite[Thm.~2.1]{ARV10} is that if $\cC$ is
finitely cocomplete then $T$ commutes with all ($\kappa$-)sifted
colimits if and only if $T$ commutes with all reflexive coequalizers and
($\kappa$-)filtered colimits.

\subsection{Algebraic theories}\label{sec:algebraic-theories}
Monads which commute with sifted colimits arise naturally in the study
of algebraic theories in the sense of Lawvere \cite{Law04}. Recall that
an (algebraic) theory is a category $\cT$ equipped with a product
preserving functor $i\colon\FinSet^{\op}\rightarrow \cT$ which is
essentially surjective. If we let $\underline{n}\in \FinSet^{\op}$ be a set
with $n$ elements, then, since we are working in the opposite category,
$\underline{n}\cong \underline{1}^{\times n}$. So $\cT$ is equivalent to
a category whose objects are $\{i(\underline{1})^n\}_{n\in \bN}$.  If
$\cC$ is a category with finite products, a $\cT$-model in $\cC$ is a
product preserving functor $A\colon\cT\rightarrow \cC$. The collection
of $\cT$-models in $\cC$ forms a category $\model{\cC}{\cT}$ whose
morphisms are natural transformations.

We should think of $\cT$ as encoding the operations on an object of
$\model{\cC}{\cT}.$ For example, suppose $k$ is a commutative ring and define a
theory $\cT$ as the subcategory of the opposite category of
$k$-algebras whose $n$th object $i(\underline{n})$ is the free $k$-algebra
$k\langle x_1,\cdots,x_n\rangle$.

Note that for each $k$-algebra $A$, we obtain a $\cT$-model in $\Set$ by
\[
i(\underline{n})\mapsto \kAlg(k\langle x_1,\cdots,x_n\rangle, A)\cong A^n.
\]
Conversely, if $A\in \model{\Set}{\cT}$ we can identify $A$ with the
set $A(i(\underline{1}))$ equipped with the operations encoded by the
functor $A$.  For example, consider the maps in
\[
\cT(i(\underline{2}),i(\underline{1})) \cong 
\kAlg(k\langle x_1\rangle,k\langle x_1,x_2\rangle)
\]
which send $x_1$ to $x_1+x_2$ and  $x_1\cdot x_2$ respectively. These
two maps define natural operations
\begin{align*}
  (-) + (-)\colon A(i(\underline{1}))^2
  &
  \rightarrow A(i(\underline{1})) \\
  (-) \; \cdot \; (-)\colon A(i(\underline{1}))^2
  &
  \rightarrow A(i(\underline{1})).
\end{align*}
The first map is commutative since $x_1+x_2=x_2+x_1$, while the latter
generally is not. By combining maps in $\cT$ we can see that the
latter operation will distribute over the former. All of these
operations and their relations coming from $\cT$ show that
$A(i(\underline{1}))$ is a $k$-algebra.

  \begin{example}\label{ex:theories}\ 
    \begin{enumerate}
      \completeblockE{}
      
    \item \label{enum:group} Let $\cT_{\Gp}$ be the category whose objects
      are indexed by natural numbers and whose morphisms are
      \[
      \cT_{\Gp}(m,n)=\Group(F\{n\},F\{m\}),
      \] 
      where $F\{m\}$ is the free group on $m$ elements, then $\model{\Set}{\cT_{\Gp}}$ is equivalent to the category of groups.

    \item Let $G$ be a group and let $\cT_{G}$ be the theory defined
      as in \eqref{enum:group} but with
      \[
      \cT_{G}(m,n)=G\text{-}\Set(F\{n\},F\{m\}), 
      \] 
      where $F\{m\}$ is the free $G$-set on $m$ elements, then
      $\model{\Set}{\cT_{G}}$ is equivalent to the category of $G$-sets.
      
    \completeblockF{}
  
    \completeblockG{}
    
    \item Let $k$ be a commutative ring and $\cT_{Lie_k}$ be the theory
      defined as in \eqref{enum:group} but with 
      \[
      \cT_{Lie_k}(m,n)=\Lie_k(F\{n\},F\{m\}), 
      \]
      where $F\{m\}$ is the free Lie algebra over $k$ on $m$ elements,
      then $\model{\Set}{\cT_{Lie_k}}$ is equivalent to the category of Lie
      algebras over $k$.

   \item Let $\cT_{C^\infty}$ be the theory defined as in
      \eqref{enum:group} but with
      $\cT_{C^\infty}(m,n)=\Cinf(\mathbb{R}^m,\mathbb{R}^n)$, the set of
      smooth maps from $\mathbb{R}^m$ to $\mathbb{R}^n$, then
      $\model{\Set}{\cT_{C^\infty}}$ is equivalent to the category of
      $C^\infty$-rings \cite{Dub81,MoR91}.
    \end{enumerate}
  \end{example}

The list in \Cref{ex:theories} is far from comprehensive and is
limited only by the authors' imagination and the readers' patience.

If $\cT$ is a theory, we obtain $\cT$-models $\cT\{\underline{m}\}$ in
$\Set$ by setting $\cT\{\underline{m}\}(-)=\cT(i(\underline{m}),-)$, which
we can think of as the free objects on a set of $m$-elements. This
construction lifts to a (covariant!) functor $\cT\{-\}\colon \FinSet
\rightarrow \model{\Set}{\cT}$. Since $\Set$ is the closure of $\FinSet$
under sifted colimits, and sifted colimits commute with products in
$\Set$ we see that we can canonically prolong this to a functor from
$\Set$. This functor admits a forgetful right adjoint given by
evaluating at $\underline{1}$.

Just as in \Cref{sec:monads}, we can compose these adjoints to obtain a
monad $T=U\cT\{-\}$ on $\Set$. More explicitly, the formula for the left Kan
extension shows
\begin{equation}\label{eqn:monad-associated-to-a-theory}
  TX = \int^{\underline{n}\in\FinSet^{\op}} \cT(i(\underline{n}),i(\underline{1}))\times X^{n}.
\end{equation}
Since $\model{\Set}{\cT}$ is locally presentable with a conservative right adjoint $U$ which 
creates sifted colimits we can apply
\Cref{prop:barr-beck-for-presentable-categories} and see that
$\model{\Set}{\cT}$ is equivalent to the category of $T=U\cT\{-\}$
algebras in $\Set$.  
\saveblockB{}

Given a $T$-algebra $X\in \model{\Set}{\cT}$ we can consider the
category $(\model{\Set}{\cT})_{\downarrow X}$ of algebras over
$X$. Generally this can not be realized as the category of
$\Set$-valued models for an algebraic theory. However, it can be
realized as the category of models of a \emph{graded} theory
$\cT\downarrow X$ (see \Cref{ex:graded-theories}). This additional generality will prove useful in the identification of the $E_2$ term of the $T$-algebra spectral sequence with Quillen
cohomology groups in \refthmB (see \Cref{sec:methodology}).

\begin{defn}\ 
  For a set $S$ of gradings, an \emph{$S$-graded theory} $\cT$ is a category
  $\cT$ equipped with a product preserving functor
  $i\colon(\FinSet^{S})^{\op}\rightarrow \cT$ which is essentially
  surjective.
\end{defn}

When working with graded theories, it is sometimes useful to use the isomorphism of categories 
\[
\phi\cn\FinSet^S \iso \cF_S
\]
where $\cF_S$ is the category whose objects are pairs $(X,f)$ where
$f\cn X \to S$ is a map of sets with finite fibers.
Morphisms in $\cF_S$ are given by commuting triangles over $S$. For $x
= (x_s)_{s \in S} \in \FinSet^S$, $\phi(x) = (\coprod_s x_s,
\coprod_sf_s)$ where each $f_s$ is the unique map of sets $x_s \to
\{s\}$.  Likewise, for $(X,f) \in \cF_S$, $\phi^{-1}(X,f) =
(f^{-1}(s))_{s \in S}$.

\saveblockC{}

Here are two prototypical examples of graded theories.
\begin{example}\ 
  \label{ex:graded-theories}
  \begin{enumerate}
    \item Fix an abelian group $M$. The category of abelian groups
      over $M$ is the category of $(\cT_{\Ab}\downarrow M)$-models in
      $\Set$ where $\cT_{\Ab}\downarrow M$ is a theory graded on the
      underlying set of $M$ and is defined as follows: The objects of
      $\cT_{\Ab}\downarrow M$ are the objects of $\cF_M^{\op}$.  For each $x =
      (x_m)_{m \in M}$, let $i(x) = \phi(x) = (X,f)$ be the corresponding
      object of $\cF_M^{\op}$ and let $F_{\Ab}X$ be the free abelian group
      over $M$ on the elements of $X$, with structure map $F_{\Ab}X \to
      M$ determined by the set map $f\cn X \to M$. 
      If $i(x) = (X,f)$ and $i(x') = (X',f)$ are two objects of $\cT_{\Ab}\downarrow M$, then we define
      \[
      (\cT_{\Ab}\downarrow M)((X,f), (X',f')) = 
      \AbGroup_{\downarrow M}\left(
        F_{\Ab}X', 
        F_{\Ab}X
      \right). 
      \]
  \item The category of $\bZ$-graded abelian groups is the category of
    $\cT^{\bZ}_{\Ab}$-models in $\Set$ where $\cT_{\Ab}^{\bZ}$ is the
    $\bZ$-graded theory defined as follows: For $j \in \bZ$ let
    $\bZ^{t}[j]$ denote the free abelian group on $t$ elements
    concentrated in degree $j$. If $i(t)$ and $i(t')$ are two
    elements of $\cT^{\bZ}_{\Ab}$ with $t  = (\underline{t}_j)_{j \in \bZ}, t' = (\underline{t'}_{k})_{k\in \bZ}\in (\FinSet^\bZ)^\op$ then we define  \[
    \cT^{\bZ}_{\Ab}(i(t), i(t')) = \AbGroup^{\bZ}\left(\bigoplus_{k\in
      \bZ}\bZ^{t^{\prime}_{k}}[k], \bigoplus_{j\in \bZ}
      \bZ^{t_j}[j]\right)\cong\prod_{k\in \bZ}\AbGroup\left(
      \bZ^{t^{\prime}_{k}}[k], \bZ^{t_k}[k]\right).
    \] 

  %%
  %% SAVE
  %%
  \longsaveblockH{}
  \end{enumerate}
\end{example}

The theories in \Cref{ex:theories} can all be extended to the graded
case similarly and in general overcategories coming from graded theories
are graded theories. 
 
\saveblockF

\begin{remark}\label{rem:three-way-correspondence}
  The evaluation operation on the category of models in $\Set$ of an $S$-graded theory $\cT$ defines
  a forgetful functor to $\Set^{S}\cong \prod_{S}\Set$.
  This functor is finitary (meaning it preserves countable filtered colimits)
  and monadic (meaning it satisfies
  \Cref{thm:monadicity}) with associated monad $T$,
  so the category of $\cT$-models in $\Set$ is equivalent to the category of
  $T$-algebras in $\Set^{S}$.

  This construction is part of a correspondence demonstrated in
  \cite[App.~A]{ARV11} between categories of $\Set$-valued models over
  $S$-graded theories and finitary monadic categories over $\Set^S$.  
  Their results can in turn be used to show a correspondence between the
  latter and algebraic categories in the sense of
  \cite{Qui67,Qui70}.
\end{remark}

\longsaveblockJ{}

\completeblockK{}

The following result will help us in \Cref{prop:abelianization-adjunction-model-cats} connect the machinery of algebraic theories to Quillen cohomology.
\begin{prop}\label{prop:abelianization-adjunction}
  Let $\cT$ be a graded theory and $X$ a $\cT$-model in $\Set$.  Let
  $\cT \downarrow X$ be the $S$-graded theory whose
  category of models is $(\model{\Set}{\cT})_{\downarrow X}$, the
  category of objects over $X$ in $\model{\Set}{\cT}$. Then there is an
  $S$-graded theory $(\cT \downarrow X)_{\ab}$ such that the category
  of $(\cT \downarrow X)_{\ab}$-models in $\Set$, $\model{\Set}{(\cT
    \downarrow X)_{\ab}}$, is equivalent to the category of abelian
  group objects in $\model{\Set}{(\cT \downarrow X)}$.

  These two categories are monadic over $\Set^S$ with associated
  monads $(T{\downarrow X})_\ab$ and $T{\downarrow X}$.  The forgetful functor 
  \[
  \Set^{S}_{(T{\downarrow X})_\ab}\rightarrow \Set^{S}_{T{\downarrow X}} 
  \]
  is monadic with left adjoint $\textrm{Ab}$.
\end{prop}

\begin{proof}
  We have already noted that the category of models over $X$ is a
  graded theory. Let $S$ denote the
  grading for this theory. As a consequence of
  \cite[Thm.~3.11.3]{Bor94a} the category of abelian group objects in
  the category of models for a theory is a category of
  models for a new theory.  As noted in \cite[\S~3.3]{Bla08} this
  argument passes to the $S$-graded case, mutatis mutandis, to yield
  an $S$-graded theory for the abelian group objects.

  These are both locally presentable categories. Since the forgetful functor is conservative and limits and sifted colimits in these categories are both calculated in $\Set^S$ we can apply
  \Cref{prop:barr-beck-for-presentable-categories} to complete the proof.
\end{proof}

\subsection{Simplicial categories of
\texorpdfstring{$T$}{T}-algebras}\label{sec:simplicial-cat}

The theory of $T$-algebra limits and colimits from
\cref{sec:algebraic-theories} admits a straightforward extension to
the enriched context. Since we are interested in studying the \emph{space} of maps between
two $T$-algebras, modeled as a simplicial set, we give this extension
in the case that $\cC$ is a simplicial category.\footnote{Although we
  will normally distinguish between topological spaces and simplicial
  sets, we will refer to both as spaces in the case of mapping
  objects. We justify this abuse by noting that we are primarily
  interested in \emph{derived} mapping spaces, which are only 
  homotopy types, so it is not necessary to distinguish between
  the choice of model.} To obtain categorical information analogous to
the previous section we will replace all of our categories with
simplicial categories, all of our functors with simplicial functors,
and all of our natural transformations with simplicial natural
transformations. For general background
on enriched categories and functors between them the reader is
encouraged to consult \cite[\S~6.2]{Bor94a} or \cite{Kel05}.

Recall that we require a simplicial category $\cC$ to have a
tensor bifunctor 
\begin{equation*}
  \otimes\colon \sSet \times \cC \rightarrow \cC.
\end{equation*}
This is related to the simplicial mapping functor $\sC(-,-)$ and the
simplicial cotensor $(-)^{-}$ via natural adjunction isomorphisms 
\begin{equation*}\label{eqn:hom-tensor-adjunction}
  \sSet(K,\sC(C,D))\cong \sC(K\otimes C, D)\cong \sC(C,D^{K}).
\end{equation*}
\completeblockL{}%

\begin{prop}\label{prop:bicomplete-cat-algebras}
    Suppose that
    \begin{hypothenumerate}
    \item $\cC$ is a bicomplete simplicial category.
    \item $T$ is a simplicial monad acting on $\cC.$
    \item $T$ commutes with either 
	  \begin{hypothenumerate}
		  \item reflexive coequalizers or\label{it:bicomplete-a}
		  \item filtered colimits.\label{it:bicomplete-b}
	  \end{hypothenumerate}
  \end{hypothenumerate}
  Then $\cC_T$ is a bicomplete simplicial category such that 
  \begin{enumerate}
    \item The forgetful functor
      $\sC_T\rightarrow \sC$ creates limits and cotensors.
      \item The simplicial tensor is constructed as follows:
    \begin{equation}\label{eqn:tensor}
      %% arrows pointing the other way:
      % X \otimes_T V =
      % \xymatrix@C=11ex{
      % \coeq \big[ F_T(X \otimes V)
      % \ar@/^1.75pc/[r]^-{F_T(e\otimes V)}& F_T(TX\otimes V)
      % \ar@<1ex>[l]^-{F_T(\mu \otimes V)}\ar@<-1ex>[l]_-{\mu\circ
      % \alpha}\big].\\
      % }
      K \otimes_T X = 
      \xymatrix@C=11ex{
	\coeq \big[ 
	F_T(K \otimes TUX) 
	\ar@<1ex>[r]^-{F_T(K\otimes \mu )}\ar@<-1ex>[r]_-{
	  \alpha} & 
	F_T(K\otimes UX)
	\ar@/_1.75pc/[l]_-{F_T(K\otimes e)} \big]. \\
      }
    \end{equation}
    Here $\alpha\colon F_T(K\otimes TUX)\rightarrow F_T(K\otimes UX)$ is
    adjoint to the assembly map $K\otimes TUX\rightarrow T(K\otimes UX)$.
\end{enumerate}

\end{prop}
\begin{proof}
  First we check that $\cC_T$ is bicomplete: By
  \Cref{prop:monads-colimits-limits} $\cC_T$ is complete and $U$ creates limits.  
  Under hypothesis \eqref{it:bicomplete-a} we can apply
  \Cref{prop:coequalizers-imply-cocomplete} to see that $\cC_T$ is
  cocomplete. When hypothesis \eqref{it:bicomplete-b} holds,
  cocompleteness follows from \cite[Prop.~4.3.6]{Bor94a}.

  The hom spaces of $\cC_T$ are defined by taking the equalizer, in
  $\sSet$, of the the obvious analogue of \eqref{eqn:maps-of-t-algebras}. The fact
  that $U$ creates cotensors appears in \cite[Prop.~VII.2.10]{EKMM97}.
  In order for the adjunctions to hold the tensor must be defined by
  \eqref{eqn:tensor}.
\end{proof}

Note that under the hypotheses of \Cref{prop:bicomplete-cat-algebras}, if $T$
commutes with reflexive coequalizers we can compute \eqref{eqn:tensor} in $\cC$ by \Cref{prop:monads-colimits-limits}.

Graded algebraic theories are extended similarly to the simplicial
context: Regarding the category of finite sets as a simplicially enriched
category with discrete mapping objects, a simplicial algebraic theory
is just a product preserving functor $(\FinSet^S)^{\op}\rightarrow
\cT$ to a simplicially enriched category $\cT$ which is essentially
surjective. Similarly, a $T$-model in a
simplicially enriched category $\cC$ with finite products is just a product
preserving simplicial functor $\cT\rightarrow \cC$.

\begin{example}\label{ex:simplicial-theories}
  Each of the examples listed in \Cref{ex:theories} and their graded
  counterparts naturally defines a
  simplicial theory. The $\cT$-models in simplicial sets are
  equivalent to their simplicial analogues.
\end{example}

\begin{prop}\label{prop:algebraic-theories-are-simplicial-and-bicomplete}
  Let $T$ be the simplicial monad acting on $\sSet^S$ associated to an
  $S$-graded simplicial algebraic theory $\cT$. Then the category
  $\sSet^S_T$ of $T$-algebras is a bicomplete simplicial category with
  tensor defined by \eqref{eqn:tensor}. 
  
  If $S=*$ and $\cT$ is an ordinary theory regarded as a constant
  simplicial theory then for each $K\in \sSet$ and $X\in
  \sSet_T$ we have the identification 
  \[
  (K\otimes X)_n =\coprod^{\model{\Set}{\cT}}_{k\in K_n} X_n.
  \]
\end{prop}
\completeblockM{}

\subsection{Monads from operads}
\label{sec:monads-from-operads}
A \emph{symmetric sequence} in $\sSet$ is a sequence 
\[
C = \{C(n)\}_{n \ge 0}
\] 
where $C(n)$ is a simplicial set with a right $\Sigma_n$-action.  A map of
symmetric sequences is a levelwise equivariant map.

\completeblockN{} %
For the remainder of this section we
assume that $\sC$ is simplicial symmetric monoidal category with
tensor $\otimes$ such that:
\begin{hypothenumerate}
  \item $\otimes$ distributes over countable coproducts in $\sC$ and
  \item there is a symmetric monoidal functor $i:\sSet\rightarrow \sC$,
  such that the tensor of a space $K$ and an object $X$ of $\sC$ is
  defined by $iK\otimes X$. 
\end{hypothenumerate}

Now given a symmetric sequence $C$, we have an associated functor $T_C \cn
\sC \to \sC$ defined on objects by
\begin{equation}\label{eq:monad-associated-to-operad}
  T_C(X) = \coprod_{n \ge 0} C(n) \otimes_{\Sigma_n} X^{\otimes n} \cong \int^{\underline{n}\in \mathrm{Iso}\Set} C(n)\otimes X^{\otimes n} .
\end{equation}
A map of symmetric sequences yields a natural transformation of
functors, and this construction yields a functor from
symmetric sequences to endofunctors of $\sC$. 

There is a symmetric monoidal product on symmetric sequences, which we will also 
denote by $\otimes$:
\[
(C \otimes D)(n) = \coprod_{i + j = n} C(i) \times D(j)
\times_{\Sigma_i \times \Sigma_j} \Sigma_{n}
\]
Since the symmetric monoidal structure on $\sC$ distributes over coproducts we see:
\[
  T_C \otimes T_D \iso T_{C \otimes D}.
\]
Now we define the circle product by:
\[
  (C \circ D)(n) = \left(\coprod_{i\geq 0} C(i) \times_{\Sigma_i} D^{\otimes i}\right)(n).
\]

This is part of a monoidal structure on symmetric sequences such that the construction $C \mapsto T_C$ defines a monoidal functor to the category of 
endofunctors with composition product. An operad $\sO$ is a symmetric
sequence which is a monoid for the circle product; the associated
endofunctor is then a monad (see \cite[\S~11]{Rez97} for additional details).

\begin{rmk}
  \label{rmk:operads-in-sSet-is-algebraic}
  The category of operads in $\sSet$ can be constructed as the
  category of $\sSet$-valued models for a graded simplicial algebraic theory.
  As in \cite[App.~A]{Rez96} one can construct the free monoid with
  respect to the circle product on a symmetric sequence. Regarding the
  $\Sigma_n$-set $\Sigma_n\times \underline{i}$ as a symmetric
  sequence concentrated in degree $n$ (and the empty set elsewhere),
  we can apply this free construction to the symmetric sequences
  $\{\Sigma_n \times \underline i\}_{(n,i)\in \bN\times \bN}$ to
  define an $\bN$-graded algebraic theory whose algebras are operads
  in $\Set$. The category of simplicial operads is the associated
  category of models in $\sSet$.
\end{rmk}

The following standard result gives criteria for identifying when the
category of algebras over an operad is simplicially enriched.
\begin{prop}\label{prop:algebras-over-an-operad-is-bicomplete}
  Suppose that $\cC$ is a bicomplete simplicial symmetric monoidal
  category such that:
  \begin{hypothenumerate}
  \item There is a symmetric monoidal functor $ i\colon \sSet\rightarrow
  \cC$
    defining the simplicial tensor.
  \item The monoidal product in $\cC$ commutes with countable coproducts
  and either\label{it:tensor-commutes}
    \begin{hypothenumerate}
      \item reflexive coequalizers or
      \item filtered colimits.
    \end{hypothenumerate}

  \end{hypothenumerate}
  Then for any operad $\sO$ of simplicial sets, the category of
  $\sO$-algebras in $\cC$ is a bicomplete simplicial category.
\end{prop}
\completeblockO{} 
The hypotheses concerning colimits for this proposition hold whenever
the symmetric monoidal structure comes from a \emph{closed} symmetric
monoidal structure and hence distributes over all colimits. For example,
simplicial sets, simplicial abelian groups, and simplicial $R$-modules
all satisfy the conditions of
\Cref{prop:algebras-over-an-operad-is-bicomplete} with their respective
closed symmetric monoidal structures. The categories of pointed compactly
generated weak Hausdorff spaces or pointed simplicial sets, each equipped with the smash product, satisfy these conditions. Any of the standard closed symmetric monoidal categories of spectra also satisfy the hypotheses.

\section{Homotopy theory of \texorpdfstring{$T$}{T}-algebras}\label{sec:model-structures}
In \Cref{sec:model-structures-on-t-algebras} we recall conditions that
guarantee that the category of $T$-algebras has a suitable homotopy
theory. After establishing the existence of a model structure, we
construct functorial simplicial resolutions of algebras in
\Cref{sec:simplicial-resolutions} which are used in the construction of the 
$T$-algebra spectral sequence.

Here, we choose to work in the context of simplicial model categories. A
disadvantage of this approach is that some of our assumptions---most
notably the existence of colimits/limits and the standard issues concerning
cofibrancy and fibrancy---should not be strictly
necessary (see for example \cite[\S~6.2]{Lur12}). An advantage of this approach is that the theory is well-developed, well-understood, and relatively straightforward to apply to many
categories of interest. 

We have gathered the relevant results from the
literature in the interest of having a single reference for determining
whether a category of $T$-algebras admits a simplicial model structure.
The background material for this section can be found in
\cite{Qui67,Hov99,Hir03} or the appendices of \cite{Lur09}.

\subsection{Model structures on \texorpdfstring{$T$}{T}-algebras}
\label{sec:model-structures-on-t-algebras}
\completeblockP{}
Let $\cC$ be a simplicial model category and $T$ a simplicial monad acting on $\cC$.
We will now recall conditions which guarantee the simplicial structure on $\cC_T$ is
part of a simplicial model structure \cite{Qui67}. Such model
categories $\cM$ satisfy the following two equivalent forms of
Quillen's corner axiom.

\begin{itemize}
  \item[SM7:] Given any cofibration $f\in \sSet(K, L)$ and fibration
  $g\in  \cM(A, B)$ , the induced morphism
  \[ A^L\longrightarrow A^K\times_{B^K}B^L \]
  is a fibration, which is a weak equivalence if either $f$ or $g$ is.

  \item[SM7a:] Given cofibrations $f\in \sSet(K, L)$ and 
  $g\in  \cM(A, B)$, the induced morphism
\[K\otimes B\coprod _{K\otimes A} L\otimes A\longrightarrow L\otimes B\]
  is a cofibration, which is a weak equivalence if either $f$ or $g$ is.
\end{itemize}

\completeblockQ{}
\begin{defn}\label{def:quillen}
  Let $\cC$ be a (simplicial) model category.
  A (simplicial) monad $T$ acting on $\cC$ is \emph{(simplicial) Quillen} if 
  \begin{hypothenumerate}
  \item $\cC_T$ has a (simplicial) model structure such that the
    forgetful functor $U\colon \cC_T\rightarrow \cC$ is a (simplicial)
    right Quillen functor.
  \item A map $f$ of $T$-algebras is a weak equivalence if and only if
    $Uf$ is a weak equivalence.
  \end{hypothenumerate}
\end{defn}

A convenient way to show that $T$ is (simplicial) Quillen is to
induce a (simplicial)
model structure on $\cC_T$ via $F_T$. We can do this if $\cC$ is a
cofibrantly generated (simplicial) model category and $T$ satisfies some mild hypotheses.  In this
case $\cC$ has sets of generating cofibrations $I$ and acyclic
cofibrations $J$ which are used to detect acyclic fibrations and 
fibrations respectively. These sets of maps satisfy smallness hypotheses which are
used to apply Quillen's small object argument and prove the lifting
axioms.

\completeblockR{}

Suppose that $\cC$ is a model category and a functor 
\[
U\colon \cD\rightarrow \cC 
\] 
admits a left adjoint. Then we say that $U$ \emph{right induces} a
model structure on $\cD$ if $\cD$ admits a model structure such that a
map $f$ is a fibration (resp.~weak equivalence) if and only if $Uf$ is
a fibration (resp.~weak equivalence).

\begin{thm}[{Cf.~\cite[App.~A Thm.~1.4]{Sch07untitled}}]\label{prop:model-cat-algebras}
  Suppose that $\cC$ is a cofibrantly generated simplicial model
  category with generating (acyclic) cofibrations $I$ (resp.~$J$) and
  $T= U F_T$ is a simplicial monad acting on $\cC$ satisfying
  \Cref{prop:bicomplete-cat-algebras}.
  
  If the domains of $F_T I$ (resp.~$F_T J$) are small relative to $F_T
  I$-cells (resp.~$F_T J$-cells) and applying $U$ to any $F_T J$-cell
  complex yields a weak equivalence in $\cC$ then $U$ right induces a
  cofibrantly generated simplicial model category structure  on $\cC_T$.
\end{thm}
\completeblockS{} 

\begin{remark}\label{rem:model-category-extra-conditions}
  In practice, checking the smallness conditions is relatively easy. In
  fact, it is automatic when the underlying categories are locally
  presentable. So most of the work required to apply
  \Cref{prop:model-cat-algebras} involves checking that applying $U$ to
  an $F_T J$-cell complex yields a weak equivalence. This can be verified (see \cite[Lem.~B2]{Sch99})
  by showing the following two properties
  are satisfied.
  \begin{enumerate}
    \item  There is a `fibrant replacement' functor $Q\colon
    \cC_T\rightarrow \cC_T$ and a natural transformation $\Id\rightarrow Q$
    such that for all $X\in\cC_T$, the natural map $UX\rightarrow UQX$ is a
    fibrant replacement.
    \item If $UX$ is fibrant then applying $U$ to the canonical
    factorization $ X\rightarrow X^{\Delta^1}\rightarrow X^{\partial
    \Delta^1}\cong X\times X $ of the diagonal yields a weak equivalence
    followed by a fibration.
  \end{enumerate}
  \longsaveblockT{}
  In our setup,	the second property follows from the fact that $U$ preserves cotensors and $\cC$ is a simplicial model category. To verify the first property one
  can sometimes show that the fibrant replacement functor $\cC$ lifts to
  an endomorphism of $\cC_T$. This is automatic if every object is fibrant
  in $\cC$. Since the two fibrant replacement functors
  $\Ex^{\infty}$ and $\Sing_{*}|-|$ on simplicial sets are product
  preserving we can use either of them as fibrant replacement functors
  for the category of simplicial $T$-algebras associated to a (graded) theory.
\end{remark}

\begin{prop}[{Cf.~\cite[Thm.~3.1]{Sch01b}}]
\label{prop:algebraic-theories-are-quillen}
  If $T$ is the monad associated to a (graded) algebraic theory on
  simplicial sets (such as, e.g., those in \Cref{ex:simplicial-theories} or
  \Cref{rmk:operads-in-sSet-is-algebraic}), then $T$ is simplicial Quillen.
\end{prop}

\begin{prop}\label{prop:abelianization-adjunction-model-cats}
 Let $\cT$ be a graded theory and $X$ a model in $\sSet$ for this
 theory.  Let $B$ be the category of $\cT$-models over $X$ in
 $\sSet$.  Then there is an $S$-graded theory whose models in $\sSet$,
 which we will denote by $A$, is equivalent to the category of abelian
 group objects in $B$.

 These two categories are monadic over $\sSet^S$ with associated
 monads $T_{A}$ and $T_{B}$ and the forgetful functor 
 \[ 
  \iota\colon\sSet^{S}_{T_{A}}\rightarrow \sSet^{S}_{T_{B}} 
 \]
 is monadic with left adjoint $\textrm{Ab}$. Its associated
 monad is simplicial Quillen.
\end{prop}
\begin{proof}
  The argument from \Cref{prop:abelianization-adjunction} proves the claims about the underlying categories. We just need to check that $\iota$ is a right Quillen functor (it is obviously simplicial). 
  In both categories a map is a weak equivalence (resp.~fibration) if
  and only if it is a weak equivalence (resp.~fibration) in $\sSet^S$.
  Since the forgetful functor $\sSet^S_{T_A}$ to $\sSet^S$ factors
  through $\iota$ the result follows.
\end{proof}

\begin{defn}\label{def:homotopy-classes-of-maps}
  Given a simplicial model category $\cC$ we define the derived mapping
  space 
  \[
  \cC^{d}(X,Y)= \cC(X^{c}, Y^{f})
  \] 
  where $X^{c}$ is a cofibrant replacement of $X$ and $Y^{f}$
  is a fibrant replacement of $Y$. We now define
  \[ 
  \ho\cC(X,Y)=\pi_0 \cC^d(X,Y). 
  \]
\end{defn}

It follows easily from Axiom [SM7] that the derived mapping space is
well-defined up to weak equivalence and therefore $\ho\cC(X,Y)$ is
well-defined. That this definition agrees with other constructions of
the set of homotopy classes of morphisms, and that these hom-sets
assemble into a homotopy category $\ho\cC$ can be found in
\cite[Prop.~II.3.10 + \S~II.3]{GoJ99}.

\begin{prop}\label{prop:homotopy-monad}
  Suppose that the forgetful functor $U\colon\cC_T\rightarrow \cC$ is a
  simplicial Quillen right adjoint. Then the monad $T$ induces a monad $\h T$ on 
  $\ho \cC$ such that the forgetful functor $\ho (\cC_T)\rightarrow
  \ho \cC $ factors through $(\ho \cC)_{\h T}$.
\end{prop}
\begin{proof}
  Quillen adjoints induce adjoints between the homotopy categories and
  consequently a monad action on $\ho \cC$ given by the composite. The
  right adjoint between the homotopy categories always lands in the
  category of algebras over this monad.
\end{proof}

\begin{defn}
  \label{def:andre-quillen-cohomology}
  Let $\cT$ be an $S$-graded theory and $X$ a $\cT$-model in $\sSet$.
  Let  \[ 
  \iota \colon \sSet^{S}_{(T{\downarrow X})_\ab}\rightarrow
  \sSet^{S}_{T{\downarrow X}} \] be the Quillen right adjoint from
  \Cref{prop:abelianization-adjunction-model-cats}.

  If $M$ is in $\Set^{S}_{(T{\downarrow
      X})_\ab}$ and
  $Y\in \sSet^{S}_{T{\downarrow X}}$, then the $s$th
  \emph{Quillen cohomology} of $Y$ with coefficients in $M$ is defined to
  be the group
   \[
   H^s_{Q,X}(Y; M) := 
   \ho \sSet^{S}_{T \downarrow X}( Y, \iota\Sigma^s M)
   \]
   where $\Sigma^s M$ is the $s$th suspension of $M$ (see \cite{Qui67}).
\end{defn}

\subsection{Simplicial resolutions of \texorpdfstring{$T$}{T}-algebras} \label{sec:simplicial-resolutions}

If $T$ is a monad acting on $\cC$, then applying $T$ levelwise to simplicial
objects in $\cC$ yields a monad acting on the category $s\cC$ of simplicial objects
which we will also denote by $T$. Now to construct a spectral sequence
computing the homotopy groups of the space $\cC_T(X,Y)$ we would like to
resolve $X$, meaning that we want to replace $X$ by a nice simplicial
$T$-algebra $X_{\bul}$ such that $\cC_T(|X_{\bul}|,Y) \simeq \cC_T^d(X,Y)$.
If $T$ is a monad acting on $\cC$, then applying $T$ levelwise to
simplicial objects in $\cC$ yields a monad, also denoted by $T$,
acting on the category of simplicial objects in $\cC$. 

\begin{defn}\label{def:bar-resolution}
  Suppose $X$ is a $T$-algebra in $\cC$.  The \emph{bar resolution}
  (also called the cotriple resolution) of $X$ is the simplicial $T$-algebra
  \[
  B_{\bul}X=B_{\bul}(F_T,T,UX)=B_{\bul}(F_TU, F_TU, X)
  \] 
  with $B_nX=(F_TU)^{n+1}X$ and face and degeneracy maps induced from the
  monad structure on $T=UF_T$ and the $T$-algebra structure on $X$.
  \completeblockU{}
\end{defn}

Note that the counit $F_TUX\rightarrow X$ extends to a map of
simplicial $T$-algebras
\begin{equation} \label{eqn:bar-resolution}
  \varepsilon\colon B_{\bul}X\rightarrow X
\end{equation}
where we regard the target as a constant simplicial object. 
By applying $U$ to \eqref{eqn:bar-resolution} we obtain a map
$\varepsilon\colon T^{\bul+1} UX\rightarrow UX$ in $s\cC$.
We also have a simplicial map $ e \colon UX\rightarrow T^{\bul+1}UX$ by iterating the unit map $UX\rightarrow TUX$.

\completeblockV{}

For a simplicial $T$-algebra $X$, there are two relevant geometric
realizations. One is realization in the category of $T$-algebras, and
another is realization in the underlying category.  We would like to
have conditions under which these two notions coincide, i.e., under
which $U$ commutes with geometric realizations. 

One such condition appears in \cite[Prop.~X.1.3.v]{EKMM97}: If $T$ is given by
a coend formula, then $U$ preserves geometric realizations. More
precisely, if $T$ is given by a formula such as the one in
\eqref{eqn:monad-associated-to-a-theory}, we will show that $T$ commutes
with geometric realization and then apply
\Cref{prop:two-different-realizations} to see that $U$ commutes with 
geometric realizations. 

\begin{prop}\label{prop:commuting-with-geometric-realization}
Let $\cC$ be a simplicial category and $F\colon \cC\times \cD\rightarrow \cC$ a functor such that for each $d\in \cD$, $F(-,d)$ commutes with geometric realizations in $\cC$. If $T$ is an endofunctor of $\cC$ of the form \[TX= \int^{d\in \cD} G(d)\otimes F(X,d) \] for some $G\colon \cD^\op\rightarrow \sSet$, then $T$ commutes with geometric realizations. 
\end{prop}

\begin{proof}
Since geometric realization is a coend, the result immediately follows from Fubini's theorem for iterated coends.
\end{proof}

\completeblockW{}

\begin{prop}\label{prop:two-different-realizations}
  Let $\cC$ be a bicomplete simplicial category and $T$ a simplicial
  monad acting on $\cC$.  Suppose that $\cC_T$ is a
  bicomplete simplicial category and that $T$ commutes
  with geometric realization.  
  If $X_\bul$ is a simplicial object in $\cC_T$, then $|UX_\bul|_{\cC}$ is a $T$-algebra
  and $U|X_\bul|_{\cC_T} \cong |UX_\bul|_{\cC}$ in $\cC_T$. So
  $U$ commutes with geometric realization if and only if $T$ does.
\end{prop}
\begin{proof}
Since $F_T$ is a left adjoint it commutes with geometric realization. So if $U$ commutes with geometric realization then so does $T=UF_T$. 

Now suppose that $T$ commutes with geometric realization. If we take the geometric realization in the category of $T$-algebras of a canonical presentation 
\begin{equation*}
  \xymatrix@C=11ex{
  F_TTUX_{\bul}\ar@<1ex>[r]\ar@<-1ex>[r] &
  F_TUX_{\bul}\ar[r]\ar@/_1.75pc/[l]_-{e} & X_{\bul},
  }
\end{equation*}
 and then apply $U$ we obtain the following commutative diagram with marked isomorphisms:
\begin{equation*}
  \xymatrix@C=11ex{
  U|F_TTUX_{\bul}|_{\cC_T}\ar[d]_{\cong}\ar@<1ex>[r]\ar@<-1ex>[r] &
  U|F_TUX_{\bul}|_{\cC_T}\ar[d]^{\cong}\ar[r]\ar@/_1.75pc/[l]_{e} &
  U|X_{\bul}|_{\cC_T}\ar@{=}[d]\\
  T|TUX_{\bul}|_{\cC}\ar[d]_{\cong}\ar@<1ex>[r]\ar@<-1ex>[r] &
  T|UX_{\bul}|_{\cC}\ar[d]^{\cong}\ar[r]\ar@/_1.75pc/[l]_{e} &
  U|X_{\bul}|_{\cC_T}\ar[d]^{\cong}\\
  |TTUX_{\bul}|_{\cC}\ar@<1ex>[r]\ar@<-1ex>[r] &
  |TUX_{\bul}|_{\cC}\ar[r]\ar@/_1.75pc/[l]_{e} &
  |UX_{\bul}|_{\cC}.
  }
\end{equation*}
The vertical isomorphisms between the first two rows follow from $F_T$ being a left adjoint. The next two vertical isomorphisms on the left follow from our assumption on $T$ and imply the desired lower right hand isomorphism. 
\end{proof}

One can interpret the following result as saying that the bar resolution
is indeed a resolution.
\begin{prop}\label{prop:correct-bar}
  Suppose $\cC$ is a simplicial model category and $T$ is a simplicial Quillen
  monad acting on $\cC$.  If $T$ (or equivalently $U$) commutes with geometric realization,
  then
  \[
  \varepsilon\colon |B_\bul X|_{\cC_T}\rightarrow X
  \] 
  is a weak equivalence of $T$-algebras.
\end{prop}
\begin{proof}
  Because $T$ is Quillen, it suffices to show that $U\varepsilon$ is
  a weak equivalence in $\cC$.
  This follows from \Cref{prop:two-different-realizations} and the
  following well known lemma.
\end{proof}

\begin{lem}{\cite[Prop.~9.8]{May72}}\label{lem:simp-retraction}
  Let $X \in \cC_T$. The maps $e$ and $\epz$ on realization
  \[
  UX
  \xrightarrow{e} 
  |T^{\bul +1} UX|_\cC
  \xrightarrow{\epz} 
  UX
  \]
  exhibit $UX$ as a strong deformation retract
  of $|T^{\bul +1} UX|_\cC$ in $\cC$.
\end{lem}
\saveblockJ{}

\subsection{Reedy cofibrant resolutions of \texorpdfstring{$T$}{T}-algebras}\label{sec:reedy}
To construct a spectral sequence using the bar resolution of $X$ we
require that this resolution is homotopically well behaved, that is we
will require it to be a Reedy cofibrant simplicial diagram as described below. If
we were to simply take a Reedy cofibrant replacement of the bar resolution we would no
longer be able to apply \Cref{prop:correct-bar} to deduce that the
geometric realization of our resolution has the correct homotopy type.
To show that a bar resolution is Reedy cofibrant we will apply a useful trick
(\Cref{prop:trick}) which makes use of a closely related \emph{almost
simplicial} diagram.

Let $\Delta_0$ be the subcategory of $\Delta$ with the same objects but
whose morphisms are those morphisms of linearly ordered sets which
preserve the minimal element. The restriction morphism \[
  \iota^*\colon \sC^{\Delta^{\op}}\rightarrow \sC^{\Delta^{\op}_{0}}
\]
takes a simplicial object and forgets the $d_0$ face maps (those
induced by injections missing the minimal element) while retaining all
of the other structure. So one can think of a $\Delta^{\op}_0$-shaped
diagram as almost a simplicial diagram; it simply lacks the $d_0$ face
maps.
\longsaveblockY{}

\begin{defn}\label{def:latching-obj}
  Let $X_\bul$ be in $\cC^{\Delta^{\op}}$
  (resp.~$\cC^{\Delta^{\op}_0}$).  The $n$th \emph{latching object} of
  $X_\bul$ is
  \[
  L_n (X_\bul) = \colim_{[n] \to [k]} X_k,
  \]
  where the colimit is indexed over the non-identity surjections in
  $\Delta$ (this is equal to the set of non-identity surjections in
  $\Delta_0$).
\end{defn}

The category $\Delta$ is the prototypical example of a Reedy category \cite[\S~15.1]{Hir03}. It is immediate from the definition of a Reedy category that the Reedy structure on $\Delta$ restricts to Reedy structure on $\Delta^\op$. This structure is used to make the following:
\begin{defn}
  Suppose that $\cC$ is a model category.  The \emph{Reedy} model structure
  on $\cC^{\Delta^{\op}}$ (resp.~$\cC^{\Delta_0^{\op}}$) is determined by 
  \begin{hypothenumerate}
    \item $f\colon X_{\bul}\rightarrow Y_{\bul}$ is a (Reedy) weak
    equivalence if $f_n\colon X_n\rightarrow Y_n$ is a
    weak equivalence in $\cC$ for all $n\geq 0$.

    \item $f\colon X_{\bul}\rightarrow Y_{\bul}$ is a (Reedy)
    cofibration if the induced map \[ X_n\coprod_{L_n X_{\bul}} L_n
    Y_{\bul }\rightarrow Y_n\] is a cofibration in $\cC$ for all $n\geq
    0$.
  \end{hypothenumerate}
\end{defn}

To show the bar resolution is Reedy cofibrant in particular cases we
will use the following trick:
\begin{prop}\label{prop:trick}
  Suppose $\cC$ and $\cD$ are model categories and $L\colon
  \cD\rightarrow \cC$ is a left Quillen functor. Let $X_{\bul}$ be a
  simplicial diagram in $\cC$ and $\iota^*X_{\bul}\in
  \cC^{\Delta^{\op}_0}$ its restriction. Suppose that there
  exists a Reedy cofibrant $\widetilde{X}_{\bul}\in
  \cD^{\Delta_{0}^{\op}}$ such that $L\widetilde{X}_{\bul}\cong
  \iota^*X_{\bul}$. Then $X_{\bul}$ is Reedy cofibrant.
\end{prop}
\begin{proof}
By definition $X_\bul$ is Reedy cofibrant if for each
non-negative $n$ the latching map 
\[
  L_n X_\bul \to X_n
\]
is a cofibration. Note that the latching object and map depend only on
the restriction of $X_\bul$ to the subcategory
$\Delta^{\op}_{\text{surj}}$, where $\Delta_{\text{surj}}$ consists of
all objects $[n]$ but only surjective maps.  In particular it suffices
to show $\iota^*X_{\bul}$ is Reedy cofibrant. Since $L$ is a Quillen left
adjoint it commutes with colimits and preserves cofibrations so it takes
the Reedy cofibrant $\widetilde{X}_{\bul}$ to the Reedy cofibrant
diagram $L\widetilde{X}_{\bul}$. Since being cofibrant is invariant
under isomorphism the result follows.
\end{proof}

For a $T$-algebra $X \in \sC_T$, we have a $\Delta_0^{\op}$-shaped diagram
\[
T^\bul UX \colon \Delta_0^{\op} \to \sC
\]
where
\[
(T^\bul UX)_n = T^n UX
\]
and the maps $(T^\bul UX)(s_i)$ and $(T^\bul UX)(d_i)$ are defined as in the
bar construction.  

\begin{defn}\label{def:resolvable}   
  Suppose that $T$ is a Quillen monad acting on $\cC$ and $X$ is a
  $T$-algebra. A \emph{bar cofibrant} replacement of $X$ is a
  cofibrant replacement $\widetilde{X}\rightarrow X$ in $T$-algebras
  such that $T^{\bul}U\widetilde{X}$ is Reedy cofibrant in
  $\cC^{\Delta^{\op}_0}$. We will say $X$ is \emph{resolvable} if it
  admits a bar cofibrant replacement.
\end{defn}

\begin{prop}\label{prop:bar-construction-is-reedy-cofib}
  Let $T$ be a Quillen monad acting on $\cC$ and $\widetilde{X}\rightarrow X$ a bar cofibrant replacement of a $T$-algebra $X$. Then the bar resolution $B_\bul \widetilde{X}$ is a Reedy cofibrant simplicial $T$-algebra and any two choices of bar cofibrant replacement yield weakly equivalent bar resolutions. 
\end{prop}
\begin{proof}
  By assumption $T^{\bul}U\widetilde{X}$ is Reedy cofibrant in
  $\cC^{\Delta^{\op}_0}$. By applying the left
  Quillen functor $F_T\colon \cC\rightarrow \cC_T$ levelwise to this
  diagram and using \Cref{prop:trick} we see that $B_{\bul}\widetilde{X}$ is Reedy cofibrant.

  Two different bar cofibrant resolutions $\widetilde{X}_1$ and $\widetilde{X}_2$ of $X$ are, in particular, cofibrant replacements for $X$ in $T$-algebras. It follows that there is a weak equivalence of $T$-algebras $f\colon \widetilde{X}_1\rightarrow \widetilde{X}_2$ which induces a map of bar resolutions. To see that the map is a levelwise weak equivalence we argue by induction. The weak equivalence in degree 0 is induced by applying $F_T$ to the weak equivalence $U\widetilde{X}_1\xrightarrow{Uf} U\widetilde{X}_2$. Since $\widetilde{X}_1$ and $\widetilde{X}_2$ are bar cofibrant this is a weak equivalence between two cofibrant objects of $\cC$. It follows that the induced map in degree 0 is a weak equivalence between two cofibrant $T$-algebras. The induction argument is similar. 
  \end{proof}

The remainder of this section is devoted to proving various technical results
which will assist in determining when a $T$-algebra is resolvable.

\subsubsection{Monads on diagrams of simplicial sets.}
\begin{lem}\label{lem:Delta0-Eilenberg-Zilber}
  Let $S$ be a set and $\cC = \sSet^{S}$ equipped with the product model
  structure.  Then any diagram $X_\bul \colon \Delta_0^{\op} \to \cC$ is
  Reedy cofibrant.
\end{lem}
\begin{proof}
  We observe that $\Delta_0$ is \emph{Eilenberg-Zilber}
  \cite[Def.~6]{BeR13}, i.e., $\Delta_0$ satisfies:
  \begin{enumerate}
  \item[(EZ1)] For all surjections $\sigma \colon [n+m] \to [n]$ in
    $\Delta_0$, the set of sections
    \[
    \Gamma(\sigma) = \{\tau \in \Delta_0 \;\big|\; \sigma\tau =
    \text{id}_{[n]} \}
    \]
    is nonempty.
  \item[(EZ2)] For any two distinct surjections $\sigma_1, \sigma_2 \colon
    [n+m] \to [n]$, the sets of sections $\Gamma(\sigma_1)$ and
    $\Gamma(\sigma_2)$ are distinct.
  \end{enumerate}
  The two conditions are verified as follows: For any surjection $\sigma$ in
  $\Delta_0$ consider the section $\sigma'(j) = \min (\sigma^{-1}(j))$.
  It is immediate that $\sigma'$ is also in $\Delta_0$ and that
  $\sigma_1' \neq \sigma_2'$ if $\sigma_1 \neq \sigma_2$.

  By \cite[Prop.~4.1]{BeR13}, every Eilenberg-Zilber Reedy category is
  \emph{elegant} (\cite[Def.~4]{BeR13}) and by
  \cite[Prop.~3.10]{BeR13} the product and Reedy model structures agree on
  categories of elegant diagrams in $\cC = \sSet^S$.  In particular,
  the object $X_\bul$ will be Reedy cofibrant because the cofibrations
  in the product model structure are the levelwise cofibrations and
  every simplicial set is cofibrant.
\end{proof}

\begin{prop}\label{prop:sset-resolvable}
  If $T$ is a simplicial Quillen monad acting on $\sSet^{S}$,
  then any $T$-algebra admits a bar cofibrant replacement.
\end{prop}
\begin{proof}
  By \Cref{lem:Delta0-Eilenberg-Zilber} any cofibrant replacement in $\sSet^{S}_{T}$ is bar cofibrant.
\end{proof}

\subsubsection{Cellular monads.}

\completeblockZ{}

\completeblockZA{}

\begin{prop}\label{prop:ekmm-argument-for-reedy-cofibrancy}
  Let $\cC$ be a cofibrantly generated model category in which relative
  cell complexes are monomorphisms and let
  $X_{\bul} \in \cC^{\Delta^{\op}_{0}}$ be a degreewise cellular
  diagram such that each degeneracy $s_i$ is a subcellular inclusion.
  Then the latching maps of $X_{\bul}$ are cellular inclusions
  and therefore $X_{\bul}$ is Reedy cofibrant.
\end{prop}
\begin{proof}
  The proof of \cite[Thm.~X.2.7]{EKMM97} can be modified to show inductively that
  \cite[(X.2.5)]{EKMM97} is a pushout-pullback diagram of subcell
  complexes defined as unions of the subcell complexes given by the
  degeneracies. Such unions are well-defined because relative cell
  complexes are monomorphisms \cite[Prop.~10.6.10]{Hir03}.
\end{proof}

\begin{prop}\label{prop:cellular-resolvable}
  Let $T$ be a Quillen monad acting on a cofibrantly generated model category $\cC$.
  Suppose that relative cell complexes in $\cC$ are
  monomorphisms and that for any cellular object $M$, $TM$ is
  cellular and the natural unit map $M \rightarrow TM$ is a cellular
  inclusion.  If $\widetilde{X} \to X$ is a cofibrant replacement of a 
  $T$-algebra $X$, such that $U \widetilde{X}$ is cellular, then $\widetilde{X}$ is a bar cofibrant replacement of $X$.
\end{prop}
\begin{proof}
  This is an immediate application of \Cref{prop:ekmm-argument-for-reedy-cofibrancy}.
\end{proof}

\subsubsection{Monads whose unit maps are inclusions of
summands.}

\begin{prop}\label{prop:summand-inclusion}
  Let $\cC$ be a pointed model category and
  let $X_{\bul} \in \cC^{\Delta^{\op}_{0}}$ be a diagram 
  such that $X_0$ is cofibrant and each degeneracy $s_i$
  is a cofibration and the inclusion of a summand.  Then the latching maps
  of $X_{\bul}$ are cofibrations and summand inclusions, and therefore
  $X_{\bul}$ is Reedy cofibrant.
\end{prop}
\begin{proof}
  First let $X_{\infty}=\colim_{i}X_i$ where $X_i$ maps to $X_{i+1}$ via
  $s_0$. Let $A_{\infty}$ be a set such that  
  \[
  X_{\infty}=\bigvee_{\alpha \in A_{\infty}} Y_\alpha
  \]  
  where $Y_\alpha$ cannot be written as a nontrivial coproduct.  Since
  each $X_n$ is an inclusion of a summand of $X_{n+1}$, there are
  subsets $A_n \subset A_\infty$ such that
  \[ 
  X_n = \bigvee_{\alpha \in A_n} Y_\alpha .
  \] 
  The degeneracies $X_{n-i} \to X_{n}$ are summand inclusions and
  therefore are induced by subset inclusions $A_{n-i} \to A_{n}$ which
  we call the set-level degeneracies.  We can now identify the
  degenerate simplices $X_n^{\mathrm{dg}}=\bigvee_{\alpha \in A_n^{\prime}}
  Y_\alpha$ where $A_n^{\prime}$ is the union (i.e., colimit) of all
  the $A_{n-i}$ under these set-level degeneracies for $1 \le i \le
  n$.  This union of sets indexes the colimit of objects yielding the
  latching object, so we can identify $X_{n}^{\mathrm{dg}}$ with $L_n X_{\bul}$
  and the latching map with that induced by the inclusion
  $A_{n}^\prime \rightarrow A_n$.

  Let $X_n^{\mathrm{nd}}$ be the complementary summand of $L_n
  X_\bul$---this is the nondegenerate part of $X_n$.  To see that the
  latching map is a cofibration we begin by observing that it is a
  coproduct of the identity map on the latching object with the map
  from the initial object into $X_n^{\mathrm{nd}}$.  Now $X_n$ is
  cofibrant because each of the degeneracies are cofibrations and
  $X_0$ is cofibrant.  The retract $X_n^{\mathrm{nd}}$ is therefore
  cofibrant, and hence the latching map is a coproduct of
  cofibrations.
\end{proof}

\begin{prop}\label{prop:summand-resolvable}
  Let $T$ be a simplicial Quillen monad acting on a pointed simplicial model category $\cC$.
  Suppose that for any cofibrant object $M$ the natural unit map $M
  \rightarrow TM$ is a cofibration and inclusion of a summand. If $\widetilde{X} \to X$ is a cofibrant replacement of  a $T$-algebra $X$, such that $U \widetilde{X}$ is cofibrant, then $\widetilde{X}$ is a bar cofibrant replacement of $X$.
\end{prop}
\begin{proof}
  This is an immediate application of \Cref{prop:summand-inclusion}.
\end{proof}

\section{The spectral sequence and examples}\label{sec:ss-hty-alg-maps}
\subsection{Proof of Theorem~A}\label{sec:main-theorem}
Now we recall and prove the central theorem of this paper:
\begin{thmA}
  \begin{samepage}
    Let $\cC$ be a simplicial model category and $T$ a
    simplicial Quillen monad acting on $\cC$.
    Let $X$ and $Y$ be $T$-algebras.  Suppose that
    \begin{hypothenumerate}
      \item $T$ commutes with geometric realization and
      \item $X$ is resolvable with bar cofibrant replacement $\widetilde{X} \to X$.
    \end{hypothenumerate}
  \end{samepage}
  Let $U\colon \cC_T\rightarrow \cC$ denote the forgetful functor from
  the category of $T$-algebras to $\cC$.  Then $T$ induces a monad $\h
  T$ on $\ho\cC$ and there
  exists an obstruction-theoretic spectral sequence satisfying:
  \begin{enumerate} \setlength{\itemsep}{.4pc}
    \item\label{it:e-1-0-0} $E_1^{0,0}=\ho\cC(UX,UY)$.
    \item\label{it:e-2-0-0} $E_2^{0,0}=(\ho\cC)_{\h T}(UX, UY)$.  That is,
    a homotopy class $[f] \colon UX \to UY$ survives to the $E_2$ page
    if and only if it is a map of $\h T$-algebras in the homotopy
    category.
    \item \label{it:e-2-s-t} Provided a $T$-algebra map $\epz \cn X \to
    Y$ to serve as a base point, the spectral sequence conditionally
    converges to the homotopy of the derived mapping space
      \[
      \pi^s \pi_t (\cC^d(T^{\bul}U\widetilde{X}, UY),\epz) \Longrightarrow \pi_{t-s}(\cC_T^d(X,Y),\epz).
      \]
     \item \label{it:differentials} In this case the differentials $d_r[f]$ provide obstructions to
       lifting $[f]$ to a map of $T$-algebras.
     \item \label{it:edges} The edge homomorphisms 
       \begin{align*}
	 \pi_0 \cC_T^d(X,Y)
	 &\twoheadrightarrow E_\infty^{0,0}\\
	 &\hookrightarrow E_2^{0,0} = (\ho\cC)_{\h T}(UX, UY)\\
	 &\hookrightarrow E_1^{0,0} = \ho\cC (UX, UY)
       \end{align*}
       are the corresponding forgetful functors.
     \item \label{it:functoriality-X-Y}
       If $\cC_T$ has functorial bar cofibrant and fibrant
       replacements, then the spectral sequence is contravariantly functorial in
       $X$ and covariantly functorial in $Y$.
     \item \label{it:map-of-monads}
       A map of simplicial monads $T_1\rightarrow T_2$ satisfying the above
       hypotheses induces a contravariant map of spectral sequences provided that
       $X$ has a bar cofibrant replacement $\widetilde{X} \to X$ in
       $\cC_{T_2}$ such that $U_3 \widetilde{X}$ has a bar cofibrant
       replacement in $\cC_{T_1}$.
       \end{enumerate}
\end{thmA}
\begin{proof}
  First, in order for the theorem to make sense there must be a
  derived mapping space of $T$-algebras and this follows from the
  assumption that $T$ is simplicial Quillen. 

  The conclusions of the theorem depend only on the weak equivalence
  classes of $X$ and $Y$, so without loss of generality we assume $Y$
  is a fibrant $T$-algebra and, replacing $X$ with $\widetilde{X}$ if
  necessary, that $X$ is a bar cofibrant $T$-algebra.  By
  \Cref{prop:bar-construction-is-reedy-cofib} the bar resolution
  $B_{\bul}X$ is a Reedy cofibrant simplicial $T$-algebra and any two choices of bar cofibrant replacement yield equivalent bar resolutions. Since $Y$
  is fibrant and $\cC$ is a simplicial model category, applying the
  mapping space functor $\cC_T(-,Y)$ to a Reedy cofibrant simplicial
  $T$-algebra yields a Reedy fibrant cosimplicial space.  In
  particular, $\cC_T(B_{\bul}X, Y)$ is Reedy fibrant.

  Applying \cite{Bou89}, the totalization tower for this Reedy fibrant
  cosimplicial space arising from the skeletal filtration on
  $|B_{\bul}X|$ yields an obstruction-theoretic spectral sequence
  computing the homotopy of the totalization
  \[
  \Tot(\cC_T(B_{\bul}X, Y))\cong \cC_T(|B_{\bul}X|, Y).
  \] 
  This spectral sequence conditionally converges provided there exists
  a base point at which to take homotopy groups.  (A list of
  obstructions to determining such a base point is also provided by
  the construction; see \Cref{rem:obstructions}.)

  Now since $B_{\bul}X$ is Reedy cofibrant and $\sC_T$ is a simplicial
  model category, $|B_{\bul} X|$ is a cofibrant $T$-algebra. Since $T$
  commutes with geometric realization, \Cref{prop:correct-bar} shows
  that the augmentation map 
  \[ 
  |B_{\bul} X|\rightarrow X 
  \]
  is a weak equivalence of $T$-algebras.  It follows that
  $\cC_T(|B_{\bul}X|, Y)$ is a model for $\cC_T^d(X,Y)$ and this gives
  the target of the spectral sequence in \eqref{it:e-2-s-t}.
  Conclusion \eqref{it:differentials} follows immediately from the
  conditional convergence of the spectral sequence.

  The $E_1^{0,0}$ term of the Bousfield-Kan spectral sequence is the set
  \[\pi_0 \cC_T(B_0 X,Y)=\pi_0 \cC_T(F_T UX, Y)\cong \pi_0 \cC(UX,UY).\]
  To prove \eqref{it:e-1-0-0} we will show the right-hand side can be
  identified with morphisms in the homotopy category. This follows if
  $UX$ is cofibrant and $UY$ is fibrant since $\cC$ is a simplicial
  model category. These conditions follow from the hypotheses that $X$ is
  bar cofibrant and that $T$ is Quillen: $T^{\bul}UX$ is Reedy
  cofibrant so the zeroth latching map shows that $UX$ is cofibrant.
  Since $T$ is Quillen, $U$ is a right Quillen functor and therefore 
  $UY$ is fibrant because $Y$ is fibrant.

  The edge homomorphism 
  \[
  \pi_0 \cC_T^{d}(X,Y) \rightarrow E_1^{0,0} 
  \]
  is induced by restricting along the inclusion 
  \[
  \mathrm{sk}_0 |B_{\bul} X|=F_T UX\rightarrow |B_{\bul} X| 
  \] 
  which by adjunction gives the second half of \eqref{it:edges}. The
  first half will follow from the identification of the $E_2^{0,0}$
  term in \eqref{it:e-2-0-0}.
  
  To prove \eqref{it:e-2-0-0} recall that the $E_2^{0,0}$ term of the
  Bousfield-Kan spectral sequence is defined to be the equalizer of the
  two face maps 
  \[ 
  \pi_0 \cC_T(B_0 X, Y)\rightrightarrows \pi_0 \cC_T(B_1 X, Y).
  \]
  \saveblockN{}%
   We again use the adjunction and the fact that $T^{\bul}UX$ is Reedy cofibrant
   to see that the diagram above is isomorphic to
   \[
   \ho\cC(UX,UY)\rightrightarrows \ho\cC(TUX, UY),
   \] 
   whose equalizer is, by definition, $(\ho\cC)_{\h T}(UX,UY)$ (see
   \eqref{eqn:maps-of-t-algebras} and \Cref{prop:homotopy-monad}).  In
   other words, a map lifts to $E_2^{0,0}$ precisely if it is a
   homotopy $T$-algebra map.

  Provided a base point $\epsilon$ for the spectral sequence, or even a
  point that lifts to $\Tot^2$ (see \Cref{rem:obstructions}), the $E_1$
  page of this spectral sequence is given by applying $\pi_t$ to the
  spaces $\cC_T(B_{s}X,Y)$ and normalizing as in \cite[\S~2.4]{Bou89}. The $E_2$
  term can be identified with the cohomotopy of this graded cosimplicial
  object which is typically denoted as follows:
  \[ 
  E_2^{s,t}=\pi^s\pi_t (\cC_T(B_{\bul} X, Y),\epsilon). 
  \]
  By adjunction we have 
  \[
  \cC_T(B_{n}X,Y) =\cC_T(F_T T^{n}UX,Y)\cong \cC(T^{n}UX,UY).
  \] 
  As in the previous steps, the right-hand side is a model for the
  derived mapping space since $UY$ is fibrant and $T^{\bul}UX$ is
  Reedy cofibrant.  This completes the proof of \eqref{it:e-2-s-t}.

Since the Bousfield-Kan spectral sequence is functorial in maps of fibrant cosimplicial spaces, to prove \eqref{it:functoriality-X-Y} it suffices to see that our construction of the fibrant cosimplicial space is functorial in $X$ and $Y$. This is immediate from the functoriality of the bar resolution and the conditions of \eqref{it:functoriality-X-Y}.

  To prove \eqref{it:map-of-monads} we note that a map of simplicial monads $T_1 \to T_2$
  between cocomplete categories determines a monadic adjunction
  $(F_{T_3},U_3)$ fitting into the following diagram \cite
  [Cor.~4.57, Prop.~4.5.9]{Bor94a}:
  \[\xymatrix@R=2.5pc@C=4pc{ 
    \cC_{T_1} \ar[rd]<.5ex>^{U_1}
    \ar[rr]<.5ex>^{F_{T_3}} & &
    \cC_{T_2}\ar[ld]<.5ex>^{U_2}\ar[ll]<.5ex>^{U_3}\\ &
    \cC_{}\ar[ru]<.5ex>^{F_{T_2}}\ar[lu]<.5ex>^{F_{T_1}} & }
  \] 
  To see that $U_3$ induces a morphism of simplicial mapping objects, observe that a map of simplicial monads
  induces a natural transformation between the equalizer diagrams in simplicial sets which determine the simplicial mapping objects in $T_2$ and $T_1$-algebras respectively (cf.~\eqref{eqn:maps-of-t-algebras}). Since both categories of algebras inherit their composition laws from the composition in $\cC$ we see that $U_3$ is a simplicial functor.

  \saveblockO{}

  To obtain a map between the spectral sequences corresponding to
  \[ 
  \cC_{T_2}(X,Y)\xrightarrow{U_3} \cC_{T_1}(U_3 X,U_3 Y)
  \]
  let $r\cn(U_3X)^b \to U_3X$ be a bar cofibrant replacement of
  $U_3X$ in $\cC_{T_1}$. We remind the reader that we have already replaced $X$ and $Y$ by their bar cofibrant and fibrant replacements respectively. Now apply $U_3$ to the $T_2$-bar resolution
  for $X$, precompose with the unit map $e\cn\Id_{\cC_{T_1}}\rightarrow
  T_3$, and take bar cofibrant and fibrant replacements of
  $U_3 X$ and $U_3 Y$ respectively\footnote{In practice, all of the
    model structures on $T$-algebras are usually right induced and
    $U_3$ is necessarily a simplicial right Quillen functor. In this case 
    taking a fibrant replacement is not necessary.}:
  \begin{align*}
     \cC_{T_2}(B_\bul (F_{T_2},T_2,U_2 X), Y)
    &\xrightarrow{U_3} \cC_{T_1}(U_{3} B_\bul (F_{T_2},T_2,U_2 X), U_3 Y)\\
    &= \cC_{T_1}( B_\bul (T_3F_{T_1},U_1T_3 F_{T_1},U_1U_3 X), U_3 Y)\\
    &\xrightarrow{e^*} \cC_{T_1}( B_\bul (F_{T_1},T_1,U_1U_3 X), U_3 Y)\\
    &\xrightarrow{r^*} \cC_{T_1}( B_\bul (F_{T_1},T_1,U_1(U_3 X)^b), (U_3 Y)^f).
  \end{align*}
  So by our assumptions on $X$, we see that $U_3$ induces a morphism between two fibrant cosimplicial spaces whose associated spectral sequences are the $T$-algebra spectral sequences calculating $\pi_*\cC_{T_2}^d(X,Y)$ and $\pi_* \cC_{T_1}^{d}(U_3 X,U_3Y)$ respectively. 
\saveblockP{}%
\end{proof}

We highlight two immediate corollaries of \refthmA.
\begin{cor}\label{cor:forget-full}
  The forgetful functor taking a non-empty $\ho(\cC_T)(X,Y)$ to
  $(\ho\cC)_{\h T}(X,Y)$ is surjective if and only if the differential
  $d_r$ on $E_r^{0,0}$ is trivial for all $r \ge 2$.
\end{cor}

\begin{cor}\label{cor:forget-faithful}

  Suppose the portion of the spectral sequence computing $\pi_0
  \cC_T^d(X,Y)$ converges \cite[\S~4.2]{Bou89}, i.e., there exists a base point $\epz$ and 
  \[
  \lim_s{\textstyle{\hspace{-1pt}^1\hspace{1pt}}} 
  \pi_{1}(\cC^{d}_{T}(\mathrm{sk}_s|B_{\bul}\widetilde{X}|, Y),\epz)=0.
  \]  
  Then the forgetful functor taking $\ho(\cC_T)(X,Y)$ to $(\ho\cC)_{\h
  T}(X,Y)$ is injective if and only if $E_\infty^{t,t} = 0$ for
  $t>0$.
\end{cor}

\begin{remark}
  As stated in \cite{Bou89}, every entry in the spectral sequence above
  should consist of pointed sets. We have chosen to omit the distinguished point
  $[\epsilon]$ in bidegree $(0,0)$ to simplify the statement
  of \refthmA. 
\end{remark}

\begin{rmk}\label{rem:obstructions} \ 
  There are, in fact, a variety of obstruction sequences whose
  vanishing can give a lift of $\epz$ through the totalization tower.
  The following are special cases of \cite[\S\S~2.4, 2.5, 5.2]{Bou89}
  for a cosimplicial object $X_\bul$ in a simplicial category $\cD$:
  \begin{enumerate}
  \item The $r$th spectral sequence page $E_r^{p,q}$ is defined if
    there is an element $\epz_{r-1} \in \Tot^{r-1} \cD(X_\bul, Y)$
    which lifts to $\Tot^{2r - 2} \cD(X_\bul, Y)$, and the page
    depends naturally on $\epz_{r-1}$.
  \item Let $\epz_p \in \Tot^p \cD(X_\bul, Y)$, and let $\epz_{k}$ be
    the projection of $\epz_p$ to $\Tot^{k} \cD(X_\bul , Y)$. If
    \[
    p/2 \leq k \leq p
    \]  
    then there is an obstruction element lying in $E_{p-k+1}^{p+1, p}$
    which vanishes if and only if $\epz_k$ lifts to $\Tot^{p+1}
    \cD(X_\bul, Y$).
  \end{enumerate}

  If Whitehead products vanish in each $\cD(X_s, Y)$ (e.g., when the mapping spaces of $\cD$ are $H$-spaces), then the
  range in which the obstruction classes are defined can be extended
  as follows:
  \begin{enumerate}
    \renewcommand{\theenumi}{\arabic{enumi}$^\prime$}
  \item The $r$th spectral sequence page $E_r^{p,q}$ is defined if
    there is an element $\epz_{r-2} \in \Tot^{r-2} \cD(X_\bul, Y)$
    which lifts to $\Tot^{2r - 3} \cD(X_\bul, Y)$, and the page
    depends naturally on $\epz_{r-2}$.
  \item \label{it:better-obstructions} Let $\epz_p \in \Tot^p
  \cD(X_\bul, Y)$, and let $\epz_{k}$ be the projection of $\epz_p$ to
  $\Tot^{k} \cD(X_\bul , Y)$. If \[(p-1)/2 \leq k \leq p\]  then there
  is an obstruction element lying in $E_{p-k+1}^{p+1, p}$ which vanishes
  if and only if $\epz_k$ lifts to $\Tot^{p+1} \cD(X_\bul, Y$).
  \end{enumerate}
\end{rmk}

Taking $p=1$ and $k=0$ in \eqref{it:better-obstructions} from
\Cref{rem:obstructions}, we obtain the following useful refinement of
\refthmA. 

\begin{thm}{(Cf.~\cite[Cor.~2.4.15]{GoH05})}\label{thm:obstructions}
  Let $\cC$ be a simplicial model category and $T$ a simplicial
  Quillen monad acting on $\cC$ such that $T$ commutes with geometric realization. Let $X$ and $Y$ be $T$-algebras, such that $X$ admits a bar cofibrant replacement $\widetilde{X}$. Moreover, assume that the derived mapping spaces
  $\cC^d(T^{n}U\widetilde{X}, UY)$ have the homotopy type of
  $H$-spaces.

  Then the $T$-algebra spectral sequence
  exists, its $E_2$ term is always defined, and there is a series of
  successively defined obstructions to realizing a map 
  \[ [f]\in E_2^{0,0} = (\ho\cC)_{\h T}(UX, UY) \] 
  in the groups 
  \[ E_{2}^{s+1,s}\cong \pi^{s+1}\pi_s (\cC^d(T^{\bul}U\widetilde{X},
  UY),f) \] 
  for $s\geq 1$.  In particular, if these groups are always zero, then the
  map induced by the forgetful functor 
  \[ \ho(\cC_T)(X,Y)\rightarrow (\ho\cC)_{\h T}(UX,UY) \]
   is surjective. If the portion of the spectral sequence computing $\pi_0
  \cC_T^d(X,Y)$ converges and
   \[ \pi^s \pi_s(\cC^d(T^{\bul}U\widetilde{X}, UY),f)=0\]  
  for each choice of $[f]$ and all $s \ge 1$, then this map is an injection.
\end{thm}

\subsection{Observations on \texorpdfstring{$E_1$}{E\_1}}
\label{sec:observations-e1}
To simplify notation for this section we assume that $\cC$, $T$, $X$,
and $Y$ are as in \refthmA and we have replaced $X$ with
$\widetilde{X}$ and $Y$ with its fibrant replacement if necessary.
%SAVE
\saveblockQ{}

Provided all of the terms in $E_1^{s,t}$ of the $T$-algebra spectral
sequence for $t>0$ are abelian groups, then we can avoid using the normalized cocomplex in \cite{Bou89} and instead use Moore cochains. For example,
this happens if the mapping spaces $\cC(T^nUX,UY)$ have
the homotopy type of $H$-spaces (\Cref{thm:obstructions}). For the unnormalized complex, we can apply the tensor-cotensor adjunction to obtain the following identification of $E_1^{s,t}$:
\[
E_1^{s,t} = \pi_t \big( \cC(T^s UX, UY), \epz\big) \iso
\pi_0 \cC_{\downarrow UY}(T^s UX, UY^{S^t}).
\]
This displays $E_1^{s,t}$ as a set of homotopy classes of lifts in
the diagram below, with homotopies over $\epz$:
\[
\xym{
  & UY^{S^t} \ar[d]\\
  T^s UX \ar[r]^-{\epz} \ar@{-->}[ur] & UY }
\]

As in the proof of \refthmA, $T^{\bul} UX$ is Reedy cofibrant
and it follows that each $T^n UX$ is cofibrant in $\cC$.
Since $U$ creates cotensors, preserves fibrations,
and $Y$ is fibrant in $\cC_{T}$, we see $UY^{S^t}$ is fibrant in $\cC$.
Now the overcategory is a cofibrantly generated simplicial model category whose 
cofibrations/fibrations/weak equivalences are those of $\cC$. So these 
objects are cofibrant and fibrant respectively in $\cC_{\downarrow UY}$.
Regarding $UX$ as an object over $UY$ by a chosen map $\epz \colon UX \to
UY$, we identify the unnormalized complex:
\begin{align*}\label{eq:E1st-maps-terms}
  E_1^{s,t}\cong \ho\cC_{\downarrow UY}(T^s UX, UY^{S^{t}})\quad \text{ for } t>0.
\end{align*}

\begin{figure}
\begin{tikzpicture}
    \matrix (m) [matrix of math nodes, row sep=1em, column sep=.5pc,
    text depth=.5ex, nodes={
      inner sep=.5em, minimum width=9.5pc, 
      minimum height=3.5em}] { 
      \ho\cC_{\downarrow UY} \left( T^2 UX, UY^{S^{1}} \right)
      & \ho\cC_{\downarrow UY} \left( T^2 UX, UY^{S^{2}} \right) 
      & \ho\cC_{\downarrow UY} \left( T^2 UX, UY^{S^{3}} \right) \\
      {\ho\cC \big( T UX, UY \big)^{\phantom{S^0}}}
      & \ho\cC_{\downarrow UY} \left( T UX, UY^{S^{1}} \right)
      & \ho\cC_{\downarrow UY} \left( T UX, UY^{S^{2}} \right) \\
      & {\ho\cC \big( UX,UY \big)^{\phantom{S^0}}}
      & \ho\cC_{\downarrow UY} \left( UX, UY^{S^{1}} \right)\\
   };
   % horizontal
   \draw[->] (-15pc,-7em)
   -- ++ (30.3pc, 0);
   \node at (16.5pc+1ex,-7em) {$t-s$};
   \draw (-15pc, -7em) 
   ++ (0, -1.3ex)
   ++ (4.75pc,0) node {\textit{\tiny -1}}
   ++ (10pc, 0) node {\textit{\tiny 0}}
   ++ (10pc, 0) node {\textit{\tiny 1}};

   % vertical
   \draw[->] (-5pc, -8em)
   -- ++ (0, 15em) node (vertaxis) {};
   \draw (vertaxis) ++ (0,1pc) node {$s$};
   \draw (-16pc, -8em)
   ++ (0,3.5em) node{\textit{\tiny 0}}
   ++ (0,4.5em) node{\textit{\tiny 1}}
   ++ (0,4.5em) node{\textit{\tiny 2}};
  \end{tikzpicture}   
\caption{Low-degree terms on the $E_1$ page of the $T$-algebra spectral sequence, interpreted as homotopy classes of lifts.}
\end{figure}
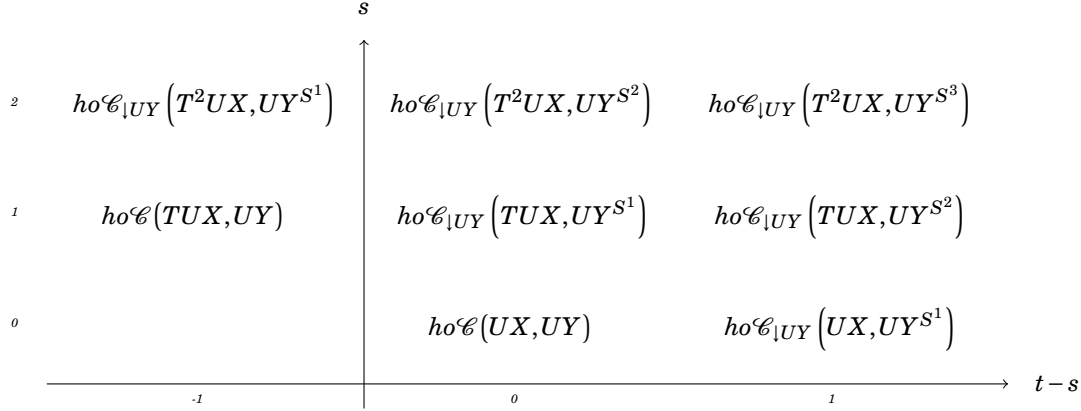
\subsection{Theorem B: Quillen cohomology and the
  \texorpdfstring{$E_2$}{E\_2}-term}
\label{sec:methodology}

The purpose of this section is to prove \refthmB, which
gives criteria for obtaining an algebraic description of the $E_2$ term from
\Cref{thm:obstructions}. 

\begin{thmB}
Let $T$ be a simplicial Quillen monad acting on a simplicial model category $\cC$ such that $T$ commutes with geometric realization. Let $X,Y\in \cC_T$. Suppose that bar cofibrant replacements exist in $\cC_T$ so that, without loss of generality, we can assume $X$ and $Y$ are bar cofibrant and fibrant respectively. Moreover, assume that the derived mapping spaces $\cC^d(T^{n}UX, UY)$ have the homotopy type of $H$-spaces and there is a functor
\[ 
\pi_*\colon \ho\cC\rightarrow \cD 
\] 
such that 
\begin{hypothenumerate}
  \item \label{item:universal-coefficient} The associated map 
  \[ 
  \pi_*\colon \ho \cC(T^{s}UX,UY^{S^t}) \to 
  \cD(\pi_* T^{s}UX,\pi_*UY^{S^t}) 
  \] 
  is an isomorphism for all $s,t\geq 0$.

  \item \label{item:algebraic-description} There is a natural
  isomorphism $\pi_*TX\cong T_{\alg}\pi_*X$ for a monad $T_{\alg}$
  compatible with the structure homomorphisms of $T$ and $T_{\alg}$.

  \item \label{item:theories} The categories $\cD$ and
  $\cD_{T_{\alg}}$ are categories of $\Set$-valued models for some graded
  algebraic theories (i.e., they are algebraic categories in the sense of
  Quillen).

  \item \label{item:abelian-groups} For $t\geq 1$, $\pi_*Y^{S^t}$ is
   naturally an abelian group object in the category of
   $T_{\alg}$-algebras over $\pi_* Y$. 
\end{hypothenumerate}
Then the $E_2$ term of the $T$-algebra spectral sequence 
exists and can be identified as follows:
\begin{align*}
  E_2^{0,0} &\cong \cD_{T_{\alg}}(\pi_* UX, \pi_* UY)\\
  E_2^{s,t} &\cong H^{s}_{Q, \pi_* UY}(\pi_*UX;\pi_*UY^{S^t}) \quad
  \text{ for }t> 0.
\end{align*}
Here the cohomology groups on
the second line are the associated Quillen cohomology groups of our
$T_{\alg}$-algebra $\pi_* X$ viewed as an algebra over $\pi_* Y$
via a choice of an element in $E_2^{0,0}$ as in
\Cref{def:andre-quillen-cohomology}.

\end{thmB}
\begin{proof}
First we identify $E_2^{0,0}$: As remarked in the proof of \refthmA, this is computed by the equalizer
\[ 
E_2^{0,0} = \eq\big(\ho \cC(UX,UY)\rightrightarrows
\ho \cC(TUX,UY)\big)
\] 
where the
morphisms are induced by the $T$-algebra structures on $X$
and $Y$ respectively. Using the isomorphisms from condition
\eqref{item:universal-coefficient} we obtain the identification
\[ 
E_2^{0,0}\cong \eq\big(\cD(\pi_*UX,\pi_*UY)\rightrightarrows
\cD(\pi_*TUX,\pi_*UY)\big).
\] 
Applying  $\pi_*$ and
condition \eqref{item:algebraic-description} we obtain
\[ 
E_2^{0,0}\cong
\cD_{T_{\alg}}(\pi_*UX,\pi_*UY)\cong
\eq\big(\cD(\pi_*UX,\pi_*UY)\rightrightarrows
\cD(T_{\alg}\pi_*UX,\pi_*UY)\big).
\] 

Now to identify the remainder of the $E_2$ term we pick a map $f\colon
UX\rightarrow UY$ representing some element $[f]\in
E_2^{0,0}$. By \refthmA for $t>0$ 
\[
E_2^{s,t} \cong
\pi^{s}\pi_t\big(\cC(T^{\bul}UX,UY),f\big).
\] 
Since the hypotheses of \Cref{thm:obstructions} are satisfied
\[
\pi_t\big(\cC(T^{\bul}UX,UY),f\big)
\] 
is a cosimplicial abelian group for $t>0$ and the cohomotopy group
$\pi^s$ can be calculated as the $s$th cohomology group of the
associated Moore cochain complex. By the discussion in
\Cref{sec:observations-e1} we see
\[ 
\pi_t\big(\cC(T^{s}UX, UY),f\big)\cong 
\ho\cC_{\downarrow UY}(T^{s}UX,UY^{S^t}). 
\] 
Applying the homotopy invariant functor $\pi_*$ and conditions
\eqref{item:universal-coefficient} and
\eqref{item:algebraic-description} we obtain
\begin{align*} 
  E_2^{s,t} & \cong H^s \bigg(
  \cD_{\downarrow \pi_*Y}(T_{\alg}^\bul \pi_*UX,\pi_*UY^{S^t})
  \bigg)\\
  & \cong H^s \bigg(
  \cD_{T_{\alg} \, \downarrow \pi_*Y}(F_{T_{\alg}}T_{\alg}^\bul \pi_*UX, \pi_*UY^{S^t})
  \bigg)
\end{align*}
where the last isomorphism uses the fact that $\pi_* UY^{S^t}$ is a $T_{\alg}$-algebra
over $\pi_* Y$.  These cohomology groups are, by definition, the
\emph{cotriple cohomology groups} of $\pi_*UX$ with respect to the
cotriple associated to the monad $T_{\alg}$.

To complete the proof we identify the cotriple cohomology with 
Quillen cohomology using \eqref{item:theories}, \eqref{item:abelian-groups}, and \cite[\S~II.5 Thm.~5]{Qui67}. 
\saveblockR{}%
\end{proof}

\subsection{Applicable contexts}\label{sec:verification-of-examples}
This section will be devoted to demonstrating that the hypotheses of
\refthmA are satisfied in many categories of interest.
\subsubsection{Simplicial algebraic theories}
\begin{thm}
  If $T$ is a monad on $\sSet^S$ associated to an $S$-graded algebraic
  theory as in \Cref{sec:algebraic-theories}, then the $T$-algebra
  spectral sequence can be applied functorially to any $X,Y\in
  \sSet^S_T$.
\end{thm}
\begin{proof}
  By \Cref{prop:algebraic-theories-are-quillen} we see that $T$ is
  simplicial Quillen and $\sSet^S_T$ is a cofibrantly generated model category. Applying \Cref{prop:commuting-with-geometric-realization} with the decomposition from \eqref{eqn:monad-associated-to-a-theory} shows that $T$ commutes with geometric realizations. Finally since $\sSet_T^S$ is cofibrantly generated, it admits functorial cofibrant and fibrant replacements. It follows from \Cref{prop:sset-resolvable} that $\sSet_T^S$ admits functorial bar cofibrant replacements.
\end{proof}

For example, by \Cref{rmk:operads-in-sSet-is-algebraic} we can apply the
$T$-algebra spectral sequence to analyze spaces of operad maps. Since
the space of operad maps from an operad $\sO$ to the endomorphism operad
of an object $X$ (when defined) is in correspondence with the space of algebra
structures on $X$ \cite{Rez96}, one could, in principle, use this spectral
sequence to analyze algebra structures on $X$. 

\subsubsection{\texorpdfstring{$G$}{G}-actions}\label{sec:g-objects-are-admissible}\ \newline
For the following result one can use any of the standard cofibrantly
generated models for the category of spectra which is enriched in spaces, whose relative cell complexes are monomorphisms, and such that the tensor product of a subcellular inclusion of spaces
with a cellular spectrum is naturally a subcellular inclusion of spectra. 

\begin{prop}
  Let $G$ be a topological group admitting a cellular structure such that the
  inclusion of the unit $e\rightarrow G$ is the inclusion of a sub-complex. Let $TX=G_+\wedge X=\Res_e^G \Ind_e^G X$ be the
  monad on pointed spaces/spectra whose algebras are pointed $G$-spaces/$G$-spectra.
  Then the $T$-algebra spectral sequence can be applied functorially to any $X,Y$ in these categories.
\end{prop}
\begin{proof}
  It is well known and straightforward to show using
  \Cref{prop:model-cat-algebras} and
  \Cref{rem:model-category-extra-conditions} that $T$ is simplicial Quillen. Since
  geometric realization commutes with smash products in either of these
  categories we see that $T$ commutes with geometric realization. Since
  the unit transformation applied to cellular spectra gives an inclusion
  of subcomplexes, by \Cref{prop:cellular-resolvable} we see that the equivariant cellular replacement of $X$ is bar cofibrant.
\end{proof}

In the case of pointed $G$-spaces or $G$-spectra the $T$-algebra spectral sequence
takes a familiar form. The bar resolution of a CW-complex $X$ is
the standard cofibrant replacement $B(G_+, G_+, X)\rightarrow X$ in the
`Borel' or projective model structure on $G$-objects. We emphasize that, although this is a common model structure to work with, it is not the standard model structure used in equivariant homotopy theory (cf.~\cite{May96}). In particular, the cofibrant replacement of $S^0$ in this model structure is $EG_+$, which is often infinite dimensional. 

Note that $G$ acts levelwise on the left of the simplicial bar construction and that the left copy of $G_+$ in each degree of the bar construction is actually notation for the left adjoint $\Ind_e^G$. The standard equivariant equivalence $\Ind_e^G(Y)\wedge X\cong \Ind_e^G(Y\wedge \Res_e^G X)$ \cite[(I.1.6)]{May96} induces a $G$-equivariant isomorphism between ${EG_{\bul}}_+\wedge X$ with its diagonal $G$-action and $B_\bul(G_+,G_+, X)$ with $G$ acting on the left. So the $T$-algebra spectral sequence computing the homotopy groups of the derived space of $G$-maps, in the projective model structure, between $X$ and $Y$ becomes a homotopy fixed point spectral sequence. More specifically the spectral sequence computes the homotopy groups of $F(X^c,Y^{f})^{hG}\cong F(EG_+\wedge X^c,Y^{f})^G$ where $F(X^c,Y^{f})$ is the corresponding $G$-space of maps and $X^c$ and $Y^{f}$ are functorial cofibrant and fibrant replacements of $X$ and $Y$ respectively. 

\begin{remark}\label{rem:homotopy-G-spaces}
  As one might expect, the homotopy $G$-spaces/spectra (i.e., the homotopy
  $T$-algebras for $T$ as above) will correspond to those spaces/spectra
  which admit a $G$-action in the homotopy category. Morphisms of
  homotopy $G$-spaces/spectra are maps in the homotopy category which
  commute with the $G$-action. In particular, any $G$-map which is
  non-equivariantly null-homotopic is necessarily trivial in the
  category of homotopy $G$-spaces (see \Cref{sec:g-action}).
\end{remark}

\subsubsection{Algebras over operads}
As in \Cref{sec:monads-from-operads}, we will continue to assume that all of our operads are defined in simplicial sets. We will say an operad is cofibrant if it is cofibrant in the model structure associated to the graded simplicial algebraic theory discussed in \Cref{rmk:operads-in-sSet-is-algebraic}. This is the same model structure considered in \cite[\S~3.3.1]{BeM03}. 

Suppose that $\cC$ is a simplicial symmetric monoidal model category satisfying the hypotheses of \Cref{prop:algebras-over-an-operad-is-bicomplete}. As shown in \Cref{sec:monads-from-operads}, this assumption implies that the category of $\sO$-algebras in $\sC$ is equivalent to a simplicial category of $T_{\sO}$-algebras. We will say that $\sO$ is \emph{admissible} if $T_{\sO}$ is a simplicial Quillen monad such that $\sC_{T_{\sO}}$ has a right induced model structure. 

\begin{prop}\label{prop:cofibrant-operad}
  Let $\cC$ be a symmetric monoidal simplicial model category
  satisfying the hypotheses of
  \Cref{prop:algebras-over-an-operad-is-bicomplete}.
  Let $T$ be the monad associated to a cofibrant admissible operad and suppose that
  \begin{hypothenumerate}
  \item geometric realization commutes with the symmetric monoidal
    structure on $\cC$ and
  \item one of the following conditions holds:
  \begin{hypothenumerate}
  \item The underlying category $\cC$ is $\sSet^{S}$ for some set $S$.   
  \item $\cC$ is cofibrantly generated, relative cell complexes in
    $\cC$ are monomorphisms, for each cellular $X\in \cC$ the unit map
    $X\rightarrow TX$ is a cellular inclusion, and the object
    underlying each cellular $T$-algebra is a
    cellular object of $\cC$.
  \item $\cC$ is pointed, for each cofibrant $X\in \cC$ the unit map
    $X\rightarrow TX$ is a cofibration and the inclusion of a summand,
    and the the object underlying each cofibrant $T$-algebra is
    cofibrant in $\cC$.
  \end{hypothenumerate}
  \end{hypothenumerate}
  Then the $T$-algebra spectral sequence can be applied to any $X,Y\in
  \cC_T$. Moreover if $\cC_T$ is cofibrantly generated, then the
  spectral sequence is functorial in $X$ and $Y$.
\end{prop}
\begin{proof}
  \Cref{prop:algebras-over-an-operad-is-bicomplete}
  shows that $\cC_T$ is a bicomplete simplicial category. By definition of admissibility, $T$ is simplicial Quillen.  Since geometric realization commutes with the monoidal structure, we can apply \Cref{prop:commuting-with-geometric-realization} to \eqref{eq:monad-associated-to-operad} and see that $T$ commutes with geometric realization.

  %Rezk's argument specifically shows that a cofibrant replacement in T-algebras is a retract of a cofibrant replacement in the underlying category. So any cellular/cofibrant replacement is automatically underlying cofibrant.
  Since our operad is admissible cofibrant we can replace any $T$-algebra by one which
is cofibrant in $\cC$ by the argument in \cite[Rem.~13.3]{Rez97}
(cf.~\cite[Thm.~3.5(b)]{BeM03}). Finally by the remaining hypothesis we can
apply either \Cref{prop:sset-resolvable}, \Cref{prop:cellular-resolvable},
or \Cref{prop:summand-resolvable} to see that the cofibrant/cellular replacement of any $T$-algebra is bar cofibrant. 
\end{proof}

\begin{prop}\label{prop:bv-operad}
  Let $\cC$ be a pointed symmetric monoidal simplicial model category
  satisfying the hypotheses of
  \Cref{prop:algebras-over-an-operad-is-bicomplete}.
  Let $T$ be a monad acting on $\cC$ and suppose that
  \begin{hypothenumerate}
  \item\label{it:compat-with-realization} geometric realization commutes
  with the symmetric monoidal structure on $\cC$ and
  \item $T$ arises from an admissible operad $W\sO$, where $W\sO$ is
    the Boardman-Vogt cofibrant replacement of an
    operad $\sO$ (see \cite{BeM06}) such that $\sO(0)=\sO(1)=*$.
  \end{hypothenumerate}
  Then the $T$-algebra spectral sequence
  can be applied to any $X,Y\in \cC_T$.  
\end{prop}
\begin{proof}
  We will apply \Cref{prop:cofibrant-operad} using the hypotheses that
  $\cC$ is pointed and that the unit map is a cofibration and the inclusion of a
  summand. As shown in \cite{BeM06} the Boardman-Vogt construction
  yields a functorial cofibrant replacement of our operad. We will show in 
  \Cref{lem:W-constr-is-pt-in-deg-one} that $(W\sO)(1)=*$, so the unit
  map $X\rightarrow TX$ is always the inclusion of a summand.
  
  Since $W\sO$ is cofibrant, by replacing $X$ with a cofibrant replacement
  if necessary, we can assume $X$ is cofibrant in $\cC$ by
  the argument in \Cref{prop:cofibrant-operad}. Since $\cC$ is a symmetric monoidal
  model category it is straightforward to apply the pushout-product
  axiom and induction on $n$ to see that $X^{\otimes n}$ is cofibrant.
  Finally, since our cofibrant operad is $\Sigma$-cofibrant \cite[\S~2.4]{BeM06} 
  we see that $W\sO(n)\otimes X^{\otimes n}$ is a retract
  of a cellular complex built with free $\Sigma_n$-cells. It follows
  that $W\sO(n)\otimes_{\Sigma_n} X^{\otimes n}$ is cofibrant which in
  turn implies $TX$ is cofibrant.
\end{proof}

\begin{lem}
  \label{lem:W-constr-is-pt-in-deg-one}
  Suppose that $\sO$ is an operad in $\sSet$ such that $\sO(0) = \sO(1) = *$.
  Then $W\sO(1) = *$.
\end{lem}
\begin{proof}
  The result follows immediately from the construction of $W\sO$ in
  \cite{BeM06} and we use the notation therein. Namely, under the given
  hypotheses all of the maps in the sequential colimit
  \[
  W(H,\sO)(n) = \colim \big(\sO(n) = 
  W_0(H, \sO)(n) \to W_1(H, \sO)(n) \to \cdots \big)
  \]
  are isomorphisms when $n = 1$ ($H$ is the unit interval here).  To
  see this, one observes that the right-hand (and therefore left-hand)
  vertical maps in the pushout \cite[(13)]{BeM06} are isomorphisms for
  $n = 1$: For trees $G$ with a single input edge, the objects
  $\underline{\sO}(G)$ and $\underline{\sO}^-(G)$ are equal (all
  vertices of $G$ are univalent, and if $\sO(1) = *$ then
  $\underline{\sO}_c(G) = \underline{\sO}(G)$ for any subset of
  univalent vertices $c$).  As an aside, note that this implies the
  vertical arrows in the pushout diagram at the end of
  \cite[\S~3]{BeM06} are isomorphisms for $n = 1$, and hence
  $\sF_*(\sO)(1) = \sO(1) = *$.  Moreover, this implies $(H \otimes
  \sO)^-(G) = H(G) \otimes \sO^-(G)$.  Therefore the vertical maps in
  \cite[(13)]{BeM06} are isomorphisms and $W(H,\sO)(1) = W_0(H, \sO)(1)
  = \sO(1)$.
\end{proof}

Let $R$ be a commutative ring spectrum. For the following corollary one
can use any symmetric monoidal category of $R$-modules satisfying
the conditions of \Cref{prop:algebras-over-an-operad-is-bicomplete}
and condition \eqref{it:compat-with-realization} of \Cref{prop:bv-operad}. These
conditions are easily verified in the standard cases such as
those of \cite{EKMM97,HSS00,MaM02}.

\begin{cor}
  Suppose $T$ is the monad associated to the Boardman-Vogt replacement
  of either the associative or the commutative operad (so it is an
  $A_\infty$ or $E_{\infty}$ operad) acting on $\RMod$. Then the
  $T$-algebra spectral sequence can be applied to
  any $T$-algebras $X$ and $Y$.
\end{cor}

\section{Computations}\label{sec:computations}
\subsection{\texorpdfstring{$G$}{G}-actions}\label{sec:g-action}
The next two examples provide, respectively, an example of a
non-trivial $G$-map which is trivial as a homotopy $G$-map and an
example of a non-trivial homotopy $G$-map which does not
lift to a $G$-map. We emphasize that we are working in the projective model structure for $G$-objects discussed in \Cref{sec:g-objects-are-admissible} and as a consequence the derived space of equivariant maps is modeled by the homotopy fixed points of the underlying derived space of maps.  

\begin{example}
  Regard $\bR$ as a $C_2$-space via the sign action.
  Then applying one point compactification to the inclusion
  \[\{0\}\rightarrow \bR\] yields an essential map
  \[e_{\sigma}\colon S^0\rightarrow S^{\sigma}\] 
  of pointed $C_2$-spaces.

  We use the trivial map as a base point for the $T$-algebra spectral sequence computing the derived space of pointed $C_2$-equivariant maps between $S^0$ and $S^\sigma$. We can identify the $E_2$ term as follows:
   \footnote{Normally instability, e.g., actions of the
     fundamental group, prevents getting such a simple description of
     the $E_2$ term, however in this case $S^\sigma$ is
     non-equivariantly an Eilenberg-MacLane space for $\bZ$ and so the
     second half of the refinements in \Cref{rem:obstructions} apply.}:
  \[ 
  E_2^{s,t}=H^s(C_2;\pi_{t} S^{\sigma})\Longrightarrow
  \pi_{t-s}(S^{\sigma})^{hC_2}. 
  \]  
  As noted in \Cref{rem:homotopy-G-spaces}, $e_{\sigma}$ must
  represent the trivial map in the category of pointed homotopy $C_2$-spaces.
  The spectral sequence confirms this since
  $E_2^{0,0}=(\pi_0S^{\sigma})^{C_2}=0$. In fact, since the homotopy
  groups of $S^1$ are concentrated in degree 1, this spectral sequence
  is concentrated on the line $t=1$ and necessarily collapses at
  $E_2$. Computing the group cohomology with coefficients $\pi_1(S^\sigma)\cong\bZ$ twisted by the sign action, we see
  that the only non-zero contribution is from $E_2^{1,1}=\bZ/2$,
  which detects the essential map $e_{\sigma}$ above.
\end{example}

\begin{example}
  Let $C_2$ act on $KU$ via complex conjugation. The $C_2$-action on
  $\pi_* KU$ is trivial precisely on those homotopy groups generated
  by even powers of the Bott map. In particular, if we regard $S^4$ as
  having a trivial $C_2$ action we obtain a non-trivial map
  \[\beta^2\colon S^4\rightarrow KU \] in the category of
  homotopy $C_2$-spectra. 
   
  The $T$-algebra spectral sequence computing
  the homotopy groups of the space of $C_2$-equivariant maps from $S^0$
  to $KU$ is a connective cover of the homotopy fixed point spectral
  sequence (see \Cref{fig:ku-ss}). The latter spectral sequence converges to the
  homotopy of $KO$ and there is a well-known differential
  $d_3(\beta^2)=\eta^3$ (see \cite[Prop.~5.3.1]{Rog08}) which follows from a comparison with the
  Adams-Novikov spectral sequence and the relation $\eta^4=0$ in $\pi_* S$.  
  \begin{figure}  
  \begin{center}
    \pgfimage[height=3.2cm]{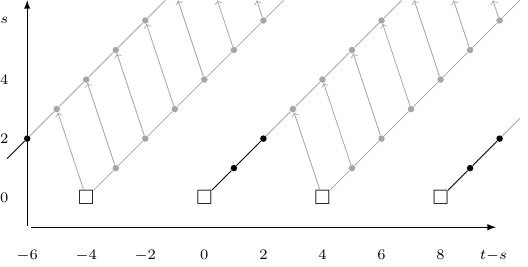}\quad \quad \quad \pgfimage[height=3.2cm]{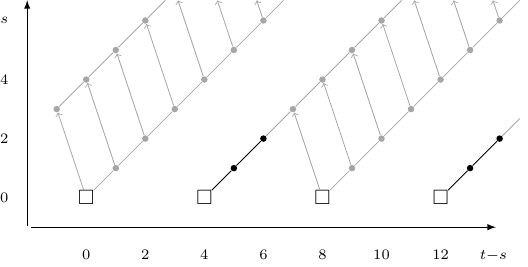}
    \end{center}
    \caption{\label{fig:ku-ss} Left: The homotopy fixed-point spectral
      sequence for $\pi_*(KU^{hC_2})$.  Right: the $T$-algebra spectral sequence for $\pi_* \Spectra^d_{C_2}(S^4, KU)$.}
  \end{figure}
  Since the $T$-algebra spectral sequence
  computing $\pi_* \Spectra_{C_2}^d(S^4,KU)$ is just a shift of the $T$-algebra spectral sequence computing $\pi_*\Spectra_{C_2}^d(S^0,KU)$, we see that the element $\beta^2\in E_2^{0,0}$
  supports a $d_3$ and does not lift to a map of $C_2$-spectra.
\end{example}

\subsection{Algebras over an operad in spectra}\label{sec:examples-operads}
\ 
\subsubsection*{$A_\infty$-algebras in $Hk$-modules}

For one example of \refthmB in action, let $k$ be a field
and $T$ the monad
\[ 
X\mapsto TX=\bigvee_{n\geq 0} K_{n} \otimes X^{\wedge_{Hk}n}
\] 
on $Hk$-module spectra associated to the $A_\infty$ operad.  Conditions
\eqref{item:universal-coefficient} and
\eqref{item:algebraic-description} from \refthmB follow from \Cref{lem:HR-mod-vs-R-mod-pi-closed,prop:a-infty-identification} respectively. There it is shown that
\[ 
T_{\alg}\pi_* X\cong \bigoplus_{n\geq 0} (\pi_*X)^{\otimes_k n} 
\] 
is a monad on graded $k$-modules whose algebras are graded
associative $k$-algebras.

Of course graded $k$-modules and graded associative $k$-algebras satisfy
condition \eqref{item:theories}.  In the category of associative
algebras over $\pi_* Y$, the abelian group objects are the square-zero
extensions of $\pi_* Y$ such as $\pi_*Y^{S^t}\cong \pi_* Y\oplus
\pi_*\Sigma^{-t}Y$ and hence condition \eqref{item:abelian-groups} from
\refthmB is satisfied.

So we obtain a spectral sequence 
\[ 
E_2^{s,t} \Longrightarrow
\pi_{t-s}
A_{\infty}\HkAlg^d(X,Y)
\]  
such that 
\[
E_2^{0,0}=\kAlg(\pi_*X,\pi_*Y)
\] 
and 
\[
E_2^{s,t}=H^{s}_{Q,\pi_*Y}(\pi_*X, \pi_*Y^{S^t}) \quad \text{ for } t>0,
\] 
where the cohomology groups are calculated in the category of graded
associative $k$-algebras over $\pi_*Y$. For $s=0$ these can be
identified with the derivations of $\pi_* X$ into $\pi_{*+t}Y$ \cite[\S 18]{Rez97} and for
$s>0$ these are the Hochschild cohomology groups
\[ 
HH^{s+1}(\pi_* X;\pi_{*+t}Y)\cong 
\Ext^{s+1}_{\pi_* X\otimes_{k}(\pi_*X)^{\op}}(\pi_* X, \pi_{*+t}Y)
\] 
of $\pi_* X$ with coefficients in $\pi_{*+t}Y$ \cite[Prop.~3.6]{Qui70}. Here $\pi_{*+t}Y$ obtains  a $\pi_*X\otimes_k (\pi_*X)^\op$-module structure by pulling back the $\pi_*Y\otimes_k (\pi_*Y)^\op$-module structure along a fixed algebra map coming from $E_2^{0,0}$.

\begin{example}\label{ex:su-n}
  In the category of $Hk$-modules, consider the $A_{\infty}$ algebra
  $Hk\wedge \Sigma^{\infty}_+ \Omega SU(n+1)$ for each $n \ge 1$. The
  homotopy of this algebra is a polynomial algebra $R = R_n$ on
  generators $\{x_{i}\}_{1\leq i\leq n}$ with $|x_i|=2i$.  To
  compute the $A_{\infty}$ self-maps we take the augmentation map as a base point and apply the $T$-algebra spectral sequence. By the discussion above we obtain:
  \begin{align*}
   E_2^{0,0} &\cong   \kAlg(R,R)\cong
   \prod_{1\leq i\leq n} (R)_{2i}\\ 
    E_2^{0,t} & \cong  \Der(R;\Sigma^{-t}R)\cong 
   \prod_{1\leq i\leq n} (R)_{2i+t}\\ 
      E_2^{s,t} & \cong  \Ext^{s+1}_{R\otimes_k R^{\op}}(R,\Sigma^{-t}R)
      \quad \text{ for } s>0.
   \end{align*}
   %The Der term is maps over R: R-->R\oplus \Sigma^{-t}R, which are determined by where they send the x_{2i}. Each lands in (\epsilon x_{2i}, f(x_2i)), for some map f. This is extends uniquely to a derivation f(xy)=f(x)\epsilon{y}+ \epsilon{x}f(y) using the free structure.  So we just have to see what is in degree 2i of \Sigma^{-t}R=\pi_*+t R. when *=2i *+t=2i+t
   In particular, these groups are zero for $t$ odd, hence
   $E_{2i}=E_{2i+1}$.  The Hochschild cohomology groups can be
   calculated by first pulling back the $R\otimes_k R^{\op}=R\otimes_k R$ action to
   a different $R\otimes_k R$ action via the isomorphism defined by
   \[
   x_i \otimes 1 \mapsto x_i\otimes 1, \quad 
   1\otimes x_i \mapsto x_i\otimes 1 - 1\otimes x_i. 
   \]
   Since $1\otimes R\subset R\otimes R$ acts trivially on the source we obtain
   an $(R\otimes_k R)$-free resolution of $R$ by applying $R \otimes_k -$ to
   the Koszul resolution of $k$: 
   \[ 
   \left(\Lambda_k[\sigma x_{1}, \cdots,\sigma x_n]\otimes_k R
   \rightarrow k\right) \stackrel{R\otimes_k -}{\mapsto} \left(R\otimes_k \Lambda_k[\sigma x_{1}, \cdots,\sigma x_n]\otimes_k R
   \rightarrow R \right)   \]
   Here $\sigma x_{i}$ has bidegree $(1,2i)$ and $d(\sigma
   x_i)=1\otimes x_i$. Using this resolution and the fact that $1\otimes x_i$ acts by 0 on the target we see that
   \[ 
      \Ext^{*}_{R\otimes_k R^{\op}}(R,R) \cong (\Lambda_k[\sigma x_{1},
      \cdots,\sigma x_n])^*\otimes_k R.
   \]

   So the Hochschild cohomology groups vanish above
   cohomological degree $n$ and hence the $T$-algebra
   spectral sequence is concentrated on
   the first $n-1$ rows and must collapse at $E_n$ for $n\geq 2$. In
   particular, if $n=1$ then the spectral sequence collapses at $E_2$
   onto the 0-line.

   Using the vanishing results in the spectral sequence above, we see that there are no obstructions to lifting an arbitrary map of $k$-algebras 
   \[ 
   H_*\Omega SU(n+1)\rightarrow H_* \Omega SU(n+1) 
   \]
   to a map of $A_{\infty}$ algebras if $n\leq 3$ and such a map is
   unique up to homotopy if $n\leq 2$.  For $n = 1$ this result is
   expected since $\Omega SU(2)\cong \Omega \Sigma S^2$ is stably a free
   $A_{\infty}$ algebra. 
   \saveblockS{}%
\end{example}

The previous computation did not depend on the $A_\infty$ algebra $Hk
\wedge \Sigma^{\infty}_{+}\Omega SU(n+1)$ so much as the fact that its ring of
homotopy groups is polynomial on generators in even degrees. In
particular, there are no obstructions to lifting an abstract algebraic
isomorphism to an equivalence if there are sufficiently few generators:

\begin{prop}
  Let $R_n$ be a polynomial algebra on $n$ generators in even degrees.
  Then for $n\leq 3$, there is a unique $Hk$-algebra $V$ up to
  homotopy such that $\pi_* V\cong R_n$. In particular, all such algebras
  are weakly equivalent to the commutative $Hk$-algebra $HR_n$.
  \saveblockT{}%
\end{prop}

\ 
\subsubsection*{$E_\infty$-algebras in rational $Hk$-modules}

For the remainder of this section we will study $E_{\infty}$-algebra maps in the category of $Hk$-modules for some field $k$ of characteristic 0.  To identify the $E_2$ term we will again apply \refthmB. We first check the hypotheses. Conditions \eqref{item:universal-coefficient} and \eqref{item:algebraic-description} follow from \Cref{lem:HR-mod-vs-R-mod-pi-closed,prop:algebraic-description-for-e-infty-monad} respectively. The verification of condition \eqref{item:theories} is immediate from the identification of $T_{\alg}$ in \Cref{prop:algebraic-description-for-e-infty-monad} as the free graded commutative $k$-algebra monad. The verification of condition \eqref{item:abelian-groups} proceeds exactly as in the $A_{\infty}$-case above and the associated Quillen cohomology groups are the classical Andr\'e-Quillen cohomology groups for graded commutative $k$-algebras.

\begin{example}\label{ex:BU}
  If we allow $n$ to go to infinity in \Cref{ex:su-n} then $\Omega
  SU\simeq BU$ is an infinite loop space and consequently $Hk\wedge
  \Sigma^{\infty}_{+} \Omega SU$ is an $E_{\infty}$ algebra in
  $Hk$-modules. 
  
We have the following identification of the
  $E_2$-term, where 
  \[
  R=H_*(\Omega SU)\cong k[x_i]_{i\geq 1}
  \]
  and $\CommkAlg$ is the category of commutative $k$-algebras:
  \begin{align*}
   E_2^{0,0} &\cong  \CommkAlg(R,R)\cong
   \prod_{i\geq 1} (R)_{2i}\\ 
    E_2^{0,t} & \cong  \Der(R;\Sigma^{-t}R)\cong 
   \prod_{i\geq 1} (R)_{2i+t}\\ 
      E_2^{s,t} & \cong H_{AQ,k}^{s}(R;\Sigma^{-t}R) \quad \text{ for } t>0.
   \end{align*}
   
  Since $R$ is a polynomial algebra over $k$, it is smooth and by
  \Cref{prop:vanishing-conditions} all higher Andr\'e-Quillen cohomology
  groups vanish. As a consequence we see the spectral sequence
  collapses at $E_2$ onto the 0-line. Hence every map of homology rings
  lifts to a homotopically unique map of $E_{\infty}$ algebras in
  $Hk$-modules.
\end{example}

In general, free algebras have collapsing spectral sequences:
\begin{example}\label{ex:free}
  If $X=T M$ is a free $\ei$ ring spectrum then the unit map
  $X\rightarrow TX$ is a map of $\ei$ ring spectra and defines a
  section of the bar resolution. Consequently the $T$-algebra spectral sequence
  computing the homotopy of $\ei^d(X,Y)$ collapses at
  $E_2$ onto the 0-line. So in this case the edge homomorphism 
  $\pi_0 \ei^d(X,Y)\rightarrow H_{\infty}(X,Y)$ is an isomorphism.
  Moreover there is a homotopy equivalence $\ei^d(X,Y)\simeq \Spectra^d(M,Y)$.
\end{example}

\begin{rem}
Since rational localization is smashing, the extension functor from
$\ei$ algebras to $\ei$ algebras in $H\bQ$ modules is an equivalence for
every rational $\ei$ ring spectrum. From this we obtain for any rational
$\ei$ ring spectra $X$ and $Y$ 
\[
\ei^d(X,Y)\simeq  {\ei}\HQAlg^d(H\bQ\wedge X, Y)\simeq {\ei}\HQAlg^d(X,Y). 
\] 
So there is no difference homotopically between the space of
$\ei$ maps between two rational $\ei$ rings and the space of
$\ei$ algebra maps in $H\bQ$-modules.
\end{rem}

\begin{example}
  \label{ex:hopf-map}
  We will now construct infinitely many homotopically distinct $\ei$ maps that all induce the
  same $\hi$ map. For a space $X$, recall that the cotensor $H\bQ^X$ is
  an $\ei$ ring spectrum satisfying $\pi_* H\bQ^X\cong H^{-*}(X;\bQ)$. Now
  to calculate the homotopy groups of $\ei^d(H\bQ^{S^2}, H\bQ^{S^3})$ we
  apply the $T$-algebra spectral sequence and use the
  identification of the $E_2$-term above. As a base point we will take
  a `trivial' map $\epsilon$ of $\ei$ rings induced by a map of the form
  $S^3\rightarrow *\rightarrow S^2$.
  
Let $\Ind_* A$ denote the graded module of indecomposables of an augmented graded commutative algebra $A$.  To calculate the $E_2$ term we have 
  \[
  E_{2}^{0,0}\cong \CommQAlg \left( \pi_* H\bQ^{S^2}, \pi_{*} H\bQ^{S^3} \right)
  \cong \Ind_{-3} \left(\pi_*H\bQ^{S^2} \right)=0=\{\epsilon\}.
  \] 
  For $t>0$ we use the map $\epsilon$ to regard
  $\pi_*H\bQ^{S^{2}}$ as a commutative algebra over
  $\pi_*H\bQ^{S^3}$ and obtain 
  \[
  E_2^{s,t} \cong
  H_{AQ}^{s} \left( \pi_{*}H\bQ^{S^2};\pi_{*+t}H\bQ^{S^3} \right)\cong
  \ho (\sCommQAlg_{\downarrow \pi_*H\bQ^{S^3}}) \left( \pi_{*}H\bQ^{S^2},\pi_{*}H\bQ^{S^3}\oplus
  \Sigma^{s}_{\bQ}\pi_{*+t}H\bQ^{S^3} \right)
  \] 
  where the right-hand side is the derived homomorphisms of simplicial
  graded commutative $\bQ$-algebras over the simplicially constant graded algebra
  $\pi_*H\bQ^{S^3}$ into the square-zero extension of this algebra by
  the $s$th suspension of $\pi_{*+t}H\bQ^{S^3}$ in simplicial
  $\pi_{*}H\bQ^{S^3}$-modules. To simplify the notation, we will present the slightly more general calculation of the Andr\'e-Quillen cohomology of an exterior $\bQ$-algebra on an even-degree generator and for arbitrary coefficients. To calculate these derived
  homomorphisms we must first construct a cofibrant replacement of the source.

  The cofibrant replacement of an exterior algebra (viewed as a
  constant simplicial object) on an element in degree $2n$ can be
  constructed as a two stage cell complex: We take the free algebra
  $\bQ[e_{2n}]$ on an element of degree $2n$ and then cone off (in
  simplicial graded commutative $\bQ$-algebras) $e_{2n}^2$ via an
  element $f_{4n}$. In other words, $A$ is defined by the following homotopy pushout diagram in simplicial graded commutative $\bQ$-algebras:
     \[\xymatrix{
    \bQ[f_{4n}] \ar[r]^{f_{4n}\mapsto e_{2n}^2}\ar[d] & \bQ[e_{2n}]\ar[d] \\
    \bQ[\mathrm{Cone}f_{4n}]\simeq \bQ \ar[r] & A\simeq \bQ\otimes^L_{\bQ[f_{4n}]}\bQ[e_{2n}]
  }\]    

  Each simplicial homotopy group $\pi_s A$ of this
  two-stage complex is a graded $\bQ$-module. Since homotopy pushouts
  of simplicial commutative algebras are derived tensor products we
  find 
  \[ 
  \pi_s A\cong \Tor^{\bQ[f_{4n}]}_s(\bQ, \bQ[e_{2n}]))) \cong 
  \begin{cases} 
    \bQ[e_{2n}]/(e_{2n}^2) & \mbox{if } s=0,\\
    0 & \mbox{otherwise.}
  \end{cases}
  \] 
  By construction $A$ is a cofibrant replacement of our exterior
  algebra. On further inspection we can see this is a cofibrant
  replacement in the category of simplical graded commutative algebras
  augmented over $\bQ$. 

  To calculate the Andr\'e-Quillen cohomology of a commutative
  algebra over $B$ we first construct a cofibrant replacement $A$ as
  above and regard it as a commutative algebra over $B$
  via the given map. For example, in our case we have 
  \[
  A\xrightarrow{\simeq } \bQ[e_{2n}]/(e_{2n}^2)\xrightarrow{\epsilon}
  B.
  \] 
  If $M$ is a $B$-module we can construct the square-zero extension
  $B\oplus \Sigma^s{M}$ whose simplicial homotopy groups are
  concentrated in degrees $0$ and $s$.  By applying the Quillen left
  adjoint $B\otimes-$ to the domain algebra we obtain an isomorphism:
  \[ 
  \sCommQAlg_{\downarrow B}(A, B\oplus \Sigma^s{M}) \cong
  \mathpzc{sAug} \mahy B \mahy \mathpzc{CAlg}(B\otimes A, B\oplus \Sigma^s{M}).
  \] 

  There is an equivalence of categories between augmented commutative
  algebras over $B$ and non-unital commutative $B$-algebras. This
  equivalence sends an algebra $B\rightarrow C\xrightarrow{f} B$ to
  $I_{B}(C)=\ker{f}$.  Maps into the square-zero extension
  $B\oplus \Sigma^s M$ in augmented commutative $B$-algebras 
  correspond to $B$-module maps from $I_B(C)/I_{B}(C)^{\otimes_B
    2}$ to $\Sigma^s M$. In our case, since the augmentation $A\rightarrow B$ factors through an augmentation to $\bQ$, and every $\bQ$-module is flat, we have
  \[
  I_B(B \otimes A)/\left(I_B(B \otimes A)^{\otimes_B 2}\right) \cong B\otimes\left( I_\bQ
  (A)/I_{\bQ}(A)^{\otimes 2}\right) = B\otimes \Ind(A).
  \] 
  % 0--> I_B(B\otimes A) --> B\otimes A --> B 
  % Using flatness:
  % 0--> B\otimes Ind(A) --> B\otimes A --> B=B\otimes Q 
  % I_B(B\otimes A)^\otimes_B 2= B\otimes Ind(A)^\otimes 2
  % 0--> I_B(B\otimes A)^\otimes_B 2--> I_B(B\otimes A)-->I_B/I_B^2 -->0
  % 0--> B\otimes Ind(A)^\otimes 2--> B\otimes  I(A)--> B\otimes Ind(A)/Ind(A)^2 -->0
  Here we are taking indecomposables levelwise to obtain a simplicial
  $B$-module $B\otimes \Ind(A)$.   Putting this together we see that 
  \[
  \sCommQAlg_{\downarrow B}(A, B\oplus \Sigma^s{M}) \cong
  \mathpzc{s}B\mahy\Mod(B\otimes \Ind(A), \Sigma^s{M}) \cong 
  \sQMod(\Ind(A), \Sigma^s(M)).
  \]

  Since the indecomposables functor is a Quillen left adjoint by \Cref{prop:abelianization-adjunction-model-cats}, it
  takes the cellular pushout diagram of simplicial augmented graded
  commutative $\bQ$-algebras defining $A$ to a homotopy pushout
  diagram of simplicial graded commutative $\bQ$-modules:
   \[\xymatrix{
    \bQ\{f_{4n}\} \ar[r]^{0}\ar[d] & \bQ\{e_{2n}\}\ar[d] \\
    \Ind(\bQ[\mathrm{Cone}f_{4n}])\simeq 0 \ar[r] & \Ind(A)\simeq \bQ\{e_{2n}\}\oplus \Sigma^1\bQ\{f_{4n}\}
  }\]    

  If we let $M_k$ denote the degree $k$ portion of $M$ we see that 
  \[
  \ho(\sQMod)(\Ind(A), \Sigma^s{M}) = 
  \begin{cases} 
    M_{2n} & \mbox{if } s=0,\\ 
    M_{4n} & \mbox{if } s=1,\\ 0 & \mbox{otherwise.}
  \end{cases}
  \]

  In our case, $n=-1$ and $M=\pi_{*+t}H\bQ^{S^3}$, so $\ho
  (\sQMod)(\Ind(A),\Sigma^s{M})$ consists of a copy of $\bQ$ when
  $(s,t)\in \{ (0,-1), (0,2), (1,1), (1,4)\}$ and is zero otherwise.
  Ignoring the non-contributing term in bidegree $(0,-1)$, we obtain the $E_{2}$-term
  in \Cref{fig:hopf-e2}.

\begin{figure}
  \begin{center}
  \pgfimage{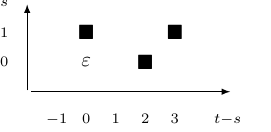}
  \end{center}
  \caption{$T$-algebra spectral sequence for $\ei$ maps $H\bQ^{S^2} \to H\bQ^{S^3}$.}
  \label{fig:hopf-e2}
\end{figure}
All other entries are trivial so the spectral sequence collapses at
$E_2.$  The $\bQ$ in $E_2^{1,1}$ detects an infinite family of homotopically distinct $\ei$ maps which,
because they land in positive filtration, induce the same $\hi$ map
$\epsilon$. It can be shown that this infinite family is
generated by the morphism of $\ei$ rings induced by the Hopf map
$S^3\rightarrow S^2$ \cite[\S~3.2]{Noe14}.
\end{example}

In the previous example, the spectral sequence vanished above the
1-line. This guaranteed the collapse of the spectral sequence and provided an algebraic
description of the homotopy groups of the space of $\ei$ maps. This is because the map
$\bQ\rightarrow \pi_*H\bQ^{S^2}$ is a local complete intersection
morphism and hence the higher Andr\'e-Quillen cohomology groups
vanish.  We will say a morphism $A\rightarrow B$ of graded
commutative rings is formally a local complete intersection, resp.~smooth,
resp.~\'etale, if the relative cotangent complex $L_{B/A}$
\cite{Qui70} has projective dimension at most one, resp.~is
projective, resp.~is contractible.

\begin{prop}\label{prop:vanishing-conditions}
  Suppose $f\colon k\rightarrow R$ and $k\rightarrow S$ are morphisms of
  rational $\ei$ rings. Suppose the $E_2$ term of the $T$-algebra spectral sequence
 computing the space of $\ei$ ring maps under $k$ between $R$ and
  $S$ can be identified with the Andr\'e-Quillen cohomology of $\pi_*R$ in $\pi_*k$-algebras as above. Then if the morphism $f$ on homotopy groups is
  \begin{enumerate}
  \item \emph{formally a local complete intersection} then the spectral sequence
    collapses at the $E_2$ page onto the 0 and 1 lines and every $\hi$
    map in $k$-modules can be realized by an $\ei$ map, although possibly non-uniquely.
  \item \emph{formally smooth} then the spectral sequence
    collapses at the $E_2$ page onto the 0-line and every $\hi$ map in $k$-modules can
    be realized,  uniquely up to homotopy, by an $\ei$ map.
  \item \emph{formally \'etale} then the spectral sequence
    collapses at the $E_2$ page and $E_2^{s,t}=0$ if $t>0$. As a
    consequence, the mapping space is homotopically discrete and every
    $\hi$ map in $k$-modules can be realized, up to a contractible space of choices, by
    an $\ei$ map.
  \end{enumerate}
\end{prop}
\begin{proof}
  All of the results follow from the vanishing of the relevant
  Andr\'e-Quillen cohomology groups \cite[Thm.~2.4
    (\textit{ii})]{Qui70} and our identification of $E_2^{0,0}$ with
  the set of $\hi$ maps.
\end{proof}

We remark that in the previous proposition one can make the necessary $E_2$ identification when $\pi_*R$ is a free $\pi_*k$-module.
%%SAVE
\saveblockU{}

\begin{example}
  \label{ex:heisenberg-man}
We now construct examples of $\hi$ ring maps that do not lift to $\ei$
ring maps. The argument below does not make explicit use of the spectral
sequence beyond the identification of the $\hi$ maps, although it does
have consequences for its behavior.

Let $M$ be the Heisenberg 3-manifold: the quotient of the group of
uni-upper triangular $3 \times 3$ real matrices by the subgroup with
all integer entries. Since $M$ is a quotient of a contractible group
by a discrete subgroup it is a $K(\pi, 1).$ The commutator subgroup
of $\pi$ is free abelian of rank one and $\pi$ fits into the short
exact sequence of groups 
\[ 
  1 \to \bZ \to \pi \to \bZ \times \bZ \to 1.  
\]
In particular $M$ is a nilpotent space.

Applying the classifying space functor to the above exact sequence
we see that up to homotopy, $M$ can also be realized as the total
space of an $S^1$ bundle over the torus $T^2$. This $S^1$ bundle is
classified by the generator of $\mathbb{Z}\cong H^2(T^2; \bZ) \cong
[T^2, BS^1]$.

A computation with the Serre spectral sequence shows $\pi_*H\bQ^M$ is
generated by exterior classes $x$ and $y$ in degree -1, polynomial
classes $\alpha$ and $\beta$ in degree -2, which satisfy
\[
0=xy=\alpha^2 = \beta^2 = \alpha\beta = x\alpha = y\beta =x\beta - y \alpha. \] 
%%SAVE
\saveblockV{}%
As a consequence we see:
\begin{equation}\label{eq:ind}
  \hi (H\bQ^M , H\bQ^{S^2})  \iso E_2^{0,0} \iso \Ind_{-2}(\pi_*H\bQ^M) = \bQ\{
  \alpha, \beta \}.
\end{equation}
There are also Massey product identifications $\alpha \in \langle x, x, y\rangle$
and $\beta \in \langle y, y, x\rangle$ with indeterminacy \[0=x H^1(M;\bQ)+H^1(M;\bQ)y.\]

Any map from $H\bQ^M$ to $H\bQ^{S^2}$ sends $x$ and $y$ to zero
for degree reasons.  Now $\alpha$ and $\beta$ are Massey products in
$x$ and $y$ and Massey products in $H\bQ^* M$ correspond to Toda
brackets in $H\bQ^M$.  Since $\ei$ maps preserve Toda brackets, they
must also send $\alpha$ and $\beta$ to zero. So $\alpha$ and
$\beta$ must support differentials and correspond to $\hi$ maps which do
not lift to $\ei$ maps. 

The behavior of this spectral sequence is explained further in
\cite{Noe14}, where it is shown that this spectral sequence is
converging to 
\[
\pi_* \Top^d(S^2, K(\pi,1)_\bQ) \cong 
\pi_* (\Omega^2 K(\pi,1)\times K(\pi,1)_\bQ)
\cong \pi_* K(\pi,1)_\bQ.
\]
In particular all non-trivial elements in the $t-s=0$ column, including $\alpha$ and $\beta$, must support a differential.

\end{example}

\subsection{\texorpdfstring{$\Coker J$}{Coker J} and maps of \texorpdfstring{$E_{\infty}$}{E\_∞} ring spectra}\label{sec:coker-j}
The following example is a joint result of the second named author and
Nick Kuhn.

For this example we will need to recall the definitions of some
classical infinite loop spaces and their associated connective spectra
\cite[p.271]{HoS78}. Let $SL_1 S^0=GL_1 S^0\langle 0\rangle$
denote the 1-component of $QS^0$. Let us fix an odd prime $p$ and let $q$ be an
integer generating $(\bZ/p^2)^{\times}$. 
\saveblockW{}%
Define $J$ to be the fiber of the map \[ BU^{\otimes}
\xrightarrow{\psi^q/\mathrm{Id}} BU^{\otimes} \] where $BU^{\otimes}$ is the
1-component of $p$-local $K$-theory and $\psi^q$ is the $q$th Adams
operation. The $d$-invariant defines a map $S^0\rightarrow KU$ which
restricts to a map $SL_1 S^0\xrightarrow{D} BU^{\otimes}$ which in turn
lifts to a map $SL_1 S^0\xrightarrow{D} J$. 
\saveblockX{}%
% KEEP THIS COMMENT LINE FOR PROPER SPACING
Let $\Coker J$ be the fiber of this last map.
\saveblockY{}

At the prime $2$ there are several possible definitions of $J$ and
consequently several possible definitions of $\Coker J$. A perfectly
reasonable approach is to set $J$ to be the fiber of the map
\begin{equation} \label{eqn:j-def}
BO^{\otimes} \xrightarrow{\psi^3/\mathrm{Id}} BO^{\otimes}.
\end{equation} 
However this introduces some homotopy groups in low degrees that are not
in the image of $D$. To rectify this there are variations where one
replaces one or both copies of $BO$ by either its 1- or 2-connected
cover. Rather than go through all the variations we note that all
possible choices will yield the same definition of $J$ after taking
1-connected covers. So we define $J$ to be the 1-connected cover of the
fiber of the map in \eqref{eqn:j-def}.  We then set $\Coker J$ to be the
fiber of the map $SL_1 S^0\langle 1\rangle \xrightarrow{D} J$ defined as above.

\saveblockZ{}

It is a non-trivial classical result that all spaces and maps in sight are infinite
loop maps (see \cite{HoS78} for a survey of these results), but this can be easily deduced from a more modern construction: For each prime $p$ the above $J$-space is the simply connected cover of the unit component of the 0th space of the $K(1)$-local sphere. The latter spectrum fits into the analogous fiber sequence with $K$-theory spectra and the $D$-invariants are induced by the unit map from the sphere spectrum. It is clear from this construction that all maps are infinite loop maps. We will follow tradition and denote the connective spectra associated to these infinite loop spaces with lower case letters.

\begin{example}
  \label{ex:e-infty-maps-coker-j}
  Let $\Sigma^{\infty}_+ \Coker J$ be the unreduced suspension spectrum of the
  infinite loop space $\Coker J$ and let $R$ be any $\ei$ ring spectrum. 
  We will show that the $T$-algebra spectral sequence computing the homotopy of
  $\ei^d(\Sigma^{\infty}_+ \Coker J,L_{K(2)} R)$ collapses at the $E_2$ page onto the 0-line. So in
  this case the edge homomorphism 
  \[
  \pi_0 \ei^d(\Sigma^{\infty}_+ \Coker J,L_{K(2)}R)\rightarrow
  H_{\infty}(\Sigma^{\infty}_+ \Coker J,L_{K(2)}R)
  \] 
  is an isomorphism. Moreover there is a homotopy equivalence of spaces
  \[
  \ei^d(\Sigma^{\infty}_+ \Coker J,L_{K(2)}R)\simeq
  \Omega^{\infty}L_{K(2)}R.
  \]
  This will follow from the canonical equivalence 
   \[
  \ei^d(\Sigma^{\infty}_+ \Coker J,L_{K(2)}R)\simeq
  \ei^d(L_{K(2)}(\Sigma^{\infty}_+ \Coker J),L_{K(2)}R),
  \]
  the following result, and \Cref{ex:free}.
\end{example}

\begin{thm}{[Kuhn-Noel]}
  \label{thm:suspension-spectrum-coker-j-is-K2-locally-free}
  There is a $K(2)$-equivalence of $\ei$ ring spectra \[T
  S^0\simeq \Sigma^{\infty}_+ \Coker J\] where $T$ is the above monad whose
  algebras are $\ei$ ring spectra.
\end{thm}
\begin{proof}
  A consequence of \cite[Thm.~2.21]{Kuh06a} is that for any
  spectrum $X$  there is a natural map of $\ei$ ring spectra
  \begin{equation} \label{eqn:kuhns-map}
    T X \rightarrow L_{K(2)} \Sigma^{\infty}_{+} \Omega^{\infty} X 
  \end{equation} 
  which is an equivalence if $X$ is 2-connected, $\pi_3 X$ is torsion, and $K(1)_*\Omega^{\infty} X$ is trivial.
  
  First we consider the $p$-local case for an odd prime $p$. In this
  case the $D$-invariant \[
  SL_1 S^{0}\xrightarrow{D} J \]
  is at least $2p-3$ connected, hence $\Coker J$ is at least
  $3$-connected. %
  \saveblockZA{}%
  As shown in \cite[Thm.~2.7]{HoS78}, $K(1)_*\Coker J$ is trivial for every prime $p$. Hence
  \eqref{eqn:kuhns-map} is an equivalence for $X=\coker j$.

  Delooping the defining fibration for $\Coker J$ we obtain a fiber sequence
  \[
   \coker j \rightarrow sl_{1}S^0 \xrightarrow{d} j.
  \]
  Since $j$ is $K(2)$-acyclic we have a $K(2)$-equivalence $\coker
  j\rightarrow sl_{1}S^0.$  The additive $H$-space structure on $QS^0$ induces a homotopy equivalence $SL_1
  S^0\simeq QS^0_{0}$ between the 1 and 0 components of $QS^0$. Although
  this is not a map of infinite loop spaces, applying the Bousfield-Kuhn
  functor $\phi_2\colon \ho{\pTop}\rightarrow \ho(L_{K(2)}\Spectra)$ (which factorizes the $K(2)$-localization functor as $L_{K(2)}\simeq \phi_2\circ \Omega^{\infty}$ \cite[Thm.~1.1]{Kuh84}) to this equivalence does yield an equivalence
  $L_{K(2)} sl_1 S^0\simeq L_{K(2)} S^0\langle 0\rangle$. 

  Since Eilenberg-MacLane spectra are
  $K(n)$-acyclic \cite{RaW80}, the defining cofiber sequence for
  the 0-connected cover
  \[
    S^0\langle 0 \rangle \rightarrow S^0 \rightarrow H\bZ
  \]
  shows that $L_{K(2)}S^0 \langle 0\rangle \simeq L_{K(2)} S^0.$  
  Finally we use naturality of the spectral sequence 
  \[ H_*(\Sigma_n; K(2)_{*}(X)^{\otimes_{K(2)_*} n}) \Longrightarrow K(2)_{*}
   ((E\Sigma_n)_{+} \wedge_{\Sigma_{n}} X^{n})\]
  to see that the functor $T$ preserves
  $K(2)$-equivalences. 
  
  Assembling these results, we obtain the desired zig-zag of equivalences of $\ei$
  ring spectra in the $K(2)$-local category
  \[ T S^0 \leftarrow T (S^0\langle 0\rangle) \leftarrow T sl_1 S^0
   \leftarrow T \coker j \rightarrow \Sigma^{\infty}_{+} \Coker J.
  \]

  At the prime 2 our defining fibration sequence is 
  \[
    \Coker J \rightarrow SL_1 S^0 \langle 1 \rangle\xrightarrow{D} J.
  \]
  Here $D$ is 3-connected so $\Coker J$ is sufficiently
  connected. Again \cite[Thm.~2.7]{HoS78} shows that $K(1)_*\Coker J$
  is trivial. %
  \saveblockZB{}%
  The rest of the argument proceeds as before to obtain
  a zig-zag of $K(2)$-local equivalences of $\ei$ ring spectra
   \[ T S^0 \leftarrow T (S^0\langle 1\rangle) \leftarrow T (sl_1
   S^0\langle 1 \rangle)
   \leftarrow T \coker j \rightarrow \Sigma^{\infty}_{+} \Coker J.
  \]
\end{proof}

\begin{remark}
  Combining this result with \cite{Kas98, Str98} one can determine the $K(2)$ and $E_2$-cohomology of $\Coker J$.
\end{remark}

\begin{remark}
  The $K(1)$-local analogue of \Cref{ex:e-infty-maps-coker-j} is
  considered in \cite[\S~3.3]{Noe14}.
\end{remark}

\subsection{Computational lemmas}\label{sec:computational-lemmas}

One of the key steps in obtaining a calculational description of the
$E_2$ term is verifying conditions \eqref{item:universal-coefficient} and \eqref{item:algebraic-description} from \refthmB. In our examples, condition \eqref{item:universal-coefficient} follows from the following result in the case $R=Hk$ for some field $k$:
\begin{lemma}\label{lem:HR-mod-vs-R-mod-pi-closed}
  If $M$ and $N$ are $R$-modules  such
  that $\pi_*M$ is projective as a $\pi_* R$-module then 
  \[
  \pi_t \left( \RMod^d(M,N) \right) \cong \Mod_{\pi_*R}(\pi_* M, \pi_{*+t}N).
  \]
\end{lemma}
\begin{proof}
  The Ext spectral sequence of \cite[Thm.~IV.4.1]{EKMM97} collapses.
\end{proof} 

To verify condition \eqref{item:algebraic-description} from \refthmB we need to find a monad $T_{\alg}$ such that there is a natural isomorphism \[\pi_* T \cong T_{\alg}\pi_*. \] When $T$ is the monad associated to $A_{\infty}/E_{\infty}$-algebras in $Hk$-modules for suitable $k$ we will identify $T_{\alg}$. The associated categories of $T_{\alg}$-algebras will be equivalent to graded associative/commutative $k$-algebras respectively.

In both of these examples our monad takes the form
\[TM=\bigvee_{n\geq 0} K_{n}\otimes_{\Sigma_n} M^{\wedge_{Hk}n}.\] In the
$E_\infty$ case $K_n$ is a contractible free $\Sigma_n$-space while in the
$A_{\infty}$ case it is a free $\Sigma_n$-space which is weakly equivalent to $\Sigma_n$. The identification of
$\pi_* TM$ as a functor of $\pi_* M$ in these cases follows from the following sequence of elementary spectral sequence arguments which we have stated in terms of a fixed commutative $S$-algebra $R$. 

\saveblockZC{}

\begin{lemma}\label{lem:HR-mod-vs-R-mod-pi-monoidal}
  If $M$ and $N$ are $R$-modules such
  that either $\pi_*M$ or $\pi_* N$ is flat as a $\pi_* R$ module then 
  \[
  \pi_* (M\wedge_{R}N)\cong \pi_*M \otimes_{\pi_* R} \pi_*N
  \]
\end{lemma}
\begin{proof}
  The Tor spectral sequence of \cite[Thm.~IV.4.1]{EKMM97} collapses.
\end{proof} 

\begin{prop}\label{prop:a-infty-identification}
  Suppose that $M$ is an $R$-module spectrum, $\pi_*R$ is a graded
  field, and $T$ is the above monad on $R$-module spectra whose category of
  algebras is the category of $A_{\infty}$ algebras in $R$-module
  spectra. Then there is a natural isomorphism 
  \[ \pi_* TM \cong T_{\alg}\pi_* M :=
  \bigoplus_{n\geq 0} (\pi_*M)^{\otimes_{\pi_*R}  n}.\] 
  Here $T_{\alg}$ is the monad on $\pi_*R$-modules whose algebras are the
  associative algebras in that category.
\end{prop}

For the $E_{\infty}$ case we will need the following:
\begin{lemma}\label{lem:HR-mod-vs-R-mod-extended-powers}
   If $E\Sigma_n$ is a contractible $\Sigma_n$-CW-complex and $M$ is an
   $R$-module such that $n!$ is a unit in $\pi_0 R$ then 
   \[
    \pi_* (E\Sigma_n\otimes _{\Sigma_n}M^{\wedge_{R}n}) \cong
    \pi_* (M^{\wedge_{R}n})/\Sigma_n.
  \]
\end{lemma}
\begin{proof}
  The homotopy orbit spectral sequence 
  \[ 
  H_s(\Sigma_n; \pi_t(M^{\wedge_{R}n})) \Longrightarrow 
  \pi_{s+t}((E\Sigma_n)_{+} \wedge_{\Sigma_n} M^{\wedge_R n})
  \] 
  collapses by a standard transfer argument since $|\Sigma_n|=n!$
  acts invertibly on the coefficients.
\end{proof} 

\begin{prop} \label{prop:algebraic-description-for-e-infty-monad}
  Suppose that $M$ is an $R$-module spectrum, $\pi_*R$ is a graded field with $\pi_0
  R$ a field of characteristic 0, and $T$
  is the above monad on $R$-module spectra whose category of algebras is
  the category of $E_{\infty}$ algebras in $R$-module spectra. Then there
  is a natural isomorphism 
  \[ \pi_* TM \cong T_{\alg}\pi_* M :=
  \bigoplus_{n\geq 0} (\pi_*M)^{\otimes_{\pi_*R}  n}/{\Sigma_n}.\] 
  Here $T_{\alg}$ is the monad on $\pi_*R$-modules whose algebras are the
  commutative algebras in that category.
\end{prop}

%%% Local Variables: ***
%%% mode: latex ***
%%% TeX-master: "obstruction-theory" ***
%%% End: ***

% amsalphaabbrv is a tweaked version of amsalpha which
% abbreviates first names as initials.  This makes for a consistent
% bibliography when not all author names are spelled out in the 
% .bib source file.

\bibliographystyle{model2-names}

%% for working copy of article
%\IfFileExists{../../biblio/biblio.bib}{\bibliography{notes/biblio/biblio}}{\bibliography{biblioJN/biblio}}

%% for arxiv and journal submission
\bibliography{obstruction-theory.bbl}

\end{document}